\documentclass{amsart} 
\usepackage{amssymb}
\usepackage{latexsym}
\usepackage{amsthm}
\usepackage[dvips]{graphicx}
\input xy
\xyoption{all}
\theoremstyle{plain}
 
\newtheorem{thm}{Theorem}[section]
\newtheorem{cor}[thm]{Corollary}
\newtheorem{lem}[thm]{Lemma}
\newtheorem{prop}[thm]{Proposition} 

\theoremstyle{remark} 

\newtheorem{ex}[thm]{Example}

\newtheorem{defn}{Definition}[section]
\newtheorem{rem}{Remark}[section]
\newtheorem{notation}{Notation}

\numberwithin{equation}{section}

\newcommand{\kv}{\mathbf{k}}
\newcommand{\mv}{\mathbf{m}}
\newcommand{\nv}{\mathbf{n}}

\newcommand{\ZZ}{\mathbb Z}
\newcommand{\CC}{\mathbb C}
\newcommand{\QQ}{\mathbb Q}
\newcommand{\RR}{\mathbb R}
\newcommand{\NN}{\mathbb N}
\newcommand{\PP}{\mathbb P}

\newcommand{\FF}{\mathbb F}
\newcommand{\LL}{\mathbb L}
\newcommand{\Aff}{\mathbb A}
\newcommand{\MM}{\mathcal M}
\newcommand{\OO}{\mathcal O}
\newcommand{\EE}{\mathcal E}

\newcommand{\w}{\mathbf w}
\newcommand{\z}{\mathbf z}
\newcommand{\g}{\mathbf g} 
\newcommand{\bi}{\mathbf i} 
\newcommand{\bj}{\mathbf j} 
\newcommand{\ba}{\mathbf a} 
\newcommand{\bb}{\mathbf b}

\newcommand{\bmu}{\boldsymbol{\mu}} 
\newcommand{\bgam}{\boldsymbol{\gamma}} 
\newcommand{\fA}{\mathfrak{A}}

\newcommand{\fL}{\mathfrak{L}} 
\newcommand{\fp}{\mathfrak{p}}

\newcommand{\ok}{\overline{k}}
\newcommand{\lra}{\longrightarrow} 

\newcommand{\xtil}{{\widetilde X}} 
\newcommand{\ytil}{{\widetilde Y}}

\DeclareMathOperator{\prim}{prim}
\DeclareMathOperator{\spec}{Spec}
\DeclareMathOperator{\lcm}{lcm}
\DeclareMathOperator{\red}{red}

\DeclareMathOperator{\Frob}{Frob}

\DeclareMathOperator{\trace}{Trace}

\begin{document}

\title{Zeta-functions of certain $K3$ fibered 
Calabi--Yau threefolds} 
\author{Yasuhiro Goto}
\address{Department of Mathematics, Hokkaido University of Education, 
1-2 Hachiman-cho, Hakodate 040-8567 Japan}
\email{ygoto@hak.hokkyodai.ac.jp}
\author{Remke Kloosterman}
\address{Institut f\"ur Algebraische Geometrie, Universit\"at Hannover,
Welfengarten 1, D-30167, Hannover, Germany}
\email{kloosterman@math.uni-hannover.de}
\author{Noriko Yui}
\address{Department of Mathematics and Statistics, Queen's University,
Kingston, Ontario Canada K7L 3N6}
\email{yui@mast.queensu.ca}
\date{\today}
\begin{abstract}
We consider certain $K3$-fibered Calabi--Yau threefolds. 
One class of such Calabi--Yau threefolds are constructed by 
Hunt and Schimmrigk in \cite{HS} using twist maps. They are realized 
in weighted projective spaces as orbifolds of hypersurfaces. 
Our main goal of this paper is to investigate arithmetic properties 
of these $K3$-fibered Calabi--Yau threefolds. In particular, we give 
detailed discussions on the construction of these Calabi--Yau varieties, 
singularities and their resolutions. We then determine the zeta-functions 
of these Calabi--Yau varieties. Next we consider deformations of our 
$K3$-fibered Calabi--Yau threefolds, and we study the variation of 
the zeta-functions using $p$-adic rigid cohomology theory.  
\end{abstract}

\maketitle
\setcounter{tocdepth}{1}
\tableofcontents
\newpage

\section{Introduction}\label{sect1}

We will study the so-called {\it split-type} (or {\it product-type}) 
Calabi--Yau threefolds. By this, we mean Calabi--Yau threefolds with 
fibrations of lower-dimensional Calabi--Yau varieties, that is, with 
elliptic, or $K3$ fibrations, or both. Hunt and Schimmrigk \cite{HS} 
described a method of constructing such split-type Calabi--Yau threefolds 
using twist maps, and showed many examples defined by hypersurfaces 
of diagonal (or Fermat) type in weighted projective spaces. In this 
paper, we also consider non-diagonal hypersurfaces which we call 
{\it quasi-diagonal} threefolds. They have many geometric properties 
in common with the diagonal type and they are birational to some 
quotient of diagonal threefolds. But, their arithmetic is slightly 
different and allows us to construct many more interesting examples. 
(For arithmetic purposes, we may work on more general hypersurfaces 
such as {\it weighted Delsarte hypersurfaces}.) 

The Calabi--Yau threefolds constructed in this paper, except for those 
discussed in Sections \ref{sect12} to \ref{sect13}, have components of 
diagonal or quasi-diagonal hypersurfaces and they are of CM type in the 
sense that their Hodge groups are commutative. However, when hypersurfaces
used in our construction are not of diagonal type, it is very unlikely 
that the quotient varieties are to be of CM type. For instance,
deformations of our Calabi--Yau threefolds are not of CM type. 

In the earlier sections of this paper, we describe how to construct 
split-type Calabi--Yau threefolds. The way we construct the Calabi--Yau 
threefolds is as follows: let $S$ be a surface and $C$ be a curve, and 
take a product $S\times C$. Assume that a finite group $\bmu$ acts on 
both $S$ and $C$. Then define the quotient $V:=S\times C/\bmu$ by this 
action of $\bmu$. Resolving its singularities, we obtain a Calabi--Yau 
threefold $\tilde V$. We focus our attention on the case where $\bmu$ 
is a cyclic group acting on one variable of each component. Then $V$ is 
birational to a quasi-smooth weighted hypersurface $X$ and we can use 
toroidal desingularizations on $X$ to construct a resolution $\tilde V$. 
In fact, singularities of $X$ are easier to handle than those of $V$ 
so that we can keep track of the fields of definition under 
desingularization. This is important for arithmetic investigations. 
 
The construction of $\tilde V$ from a product $S\times C$ naturally 
induces fibrations on $\tilde V$. They are dependent on the choice of 
each component. We consider the cases where $\tilde V$ has $K3$ and/or 
elliptic fibrations, as done in the paper of Hunt and Schimmrigk~\cite{HS}. 

We will then determine the zeta-functions of split-type Calabi--Yau 
threefolds. We discuss the cases where both $S$ and $C$ are defined by 
either diagonal or quasi-diagonal equations, so that $\tilde V$ is 
birational to a hypersurface of diagonal or quasi-diagonal type in a 
weighted projective $4$-space. The zeta-functions of such hypersurfaces 
are computed by using Weil's method Our task is therefore to determine 
their singularities and resolutions explicitly to describe the 
zeta-function of $\tilde V$. 

In order to express their zeta-functions in a simple form, we compute 
them over some finite extensions of $\FF_p$. But, by calculating them 
over several extensions of $\FF_p$, we can determine zeta-functions 
over $\FF_p$. We explain the idea of this in Section \ref{sect7} 
(cf. Lemma \ref{lem7-2e} and Remark \ref{zeta-over-Fp}) and discuss 
more details in a subsequent paper. Also, our zeta-functions in Sections 
\ref{sect7} and \ref{sect10} are described in terms of Jacobi sums. 
They are endowed with a group action which induces a natural decomposition 
(called {\it motivic decomposition}) of zeta-functions. This is explained, 
for instance, in \cite{GY} and \cite{Yui06} and we investigate more 
details about this in a subsequent paper. 

Now the \'etale cohomology groups of $\tilde V$ are built up from those of 
components by the K\"unneth formula. The eigenvalues of the Frobenius map 
on the \'etale cohomology groups of $\tilde V$ are products of the 
eigenvalues of the components which are compatible with the finite 
group actions. We compute the zeta-functions of $K3$-fibered 
Calabi--Yau threefolds, using the splitting property, and computing 
eigenvalues of Frobenius for the components which are compatible with 
group actions. 

Next, we will study the variation of the zeta-functions of the Calabi--Yau
threefolds considered in the previous sections. Some deformations can 
be constructed by twist maps, but not in general. We ought to introduce 
a different cohomology theory to study the zeta-functions, namely the  
$p$-adic rigid cohomology theory. 

The paper is organized as follows. 

In Section $2$, we review some facts on weighted diagonal hypersurfaces 
and twist maps, respectively, which are needed for the subsequent 
discussions. In Sections $3$ and $4$, we construct $K3$-fibered or 
elliptic fibered Calabi--Yau threefolds via twist maps. We start with 
the product $C\times Y$ where $C$ is a curve and $Y$ is a surface. The 
twist map induces a group action, $\bmu$, on the product $C\times Y$ and 
we take the quotient $V=C\times Y/\bmu$. Resolving singularities of $V$, 
we obtain a Calabi--Yau threefold with $K3$ or elliptic fibration. We 
list various examples of these using diagonal or quasi-diagonal 
hypersurfaces in weighted projective spaces. 

In Section $5$, we study in detail singularities appearing in our 
constructions. We choose $C$ and $S$ to be quasi-smooth curves and 
surfaces, so that they have cyclic quotient singularities. In our case, 
$V=S\times C/\bmu$ is birational to a quasi-smooth hypersurface $X$ which 
has at most cyclic quotient singularities. 

In Section $6$, we consider resolutions of singularities. Since weighted 
projective spaces are toric varieties, we can employ toric 
desingularizations and they also resolve singularities of quasi-smooth 
subvarieties $X$. To find a crepant resolution of $X$ (when it is 
Calabi--Yau) however, we usually need only a partial desingularization 
of the ambient space. We describe an algorithm for it and give several 
examples of explicit resolutions. We also discuss the field of definition 
for such resolutions and exceptional divisors. 

In Section $7$, we compute the cohomology groups of smooth Calabi--Yau 
threefolds constructed from the products $C\times S$ where $C$ is a
diagonal curve and $S$ a diagonal surface. In Sections $8$ and $9$, 
we explicitly calculate the zeta-functions of our Calabi--Yau threefolds.

{From} Section $10$ on, we consider a family of quasi-smooth hypersurfaces, 
focusing on a two-parameter family. Some of the deformations may not be 
constructed from the twist map. We compute the zeta-function of the family 
using rigid cohomology theory.  In Section $11$, the deformation matrix 
is computed in terms of hypergeometric functions. In the final section 
$12$, an explicit example is discussed. 

\medskip 
In this paper, our discussions are focused on the local arithmetic, namely, 
the determination of the zeta-functions of our Calabi--Yau varieties.  
In subsequent papers, we plan to give global considerations and 
hope to determine the $L$-series of our Calabi--Yau varieties and discuss
their modularity (at least at the motivic level.)  Also, in this article, we
have not fully utilized elliptic or $K3$-fibrations in our arithmetic quests,
and this is left as a topic of our future investigation. This involves
some function fields arithmetic of elliptic curves or $K3$ surfaces.

\medskip 
{\bf Acknowledgments} Y. Goto's research was partially supported by the  
Grants-in-Aid for Scientific Research (C) 18540005 and 21540003 of the 
Japan Society for the Promotion of Science (JSPS). 
The second author wishes to thank the third author for the invitation 
to Queen's University and he wishes to thank Queen's University for 
their hospitality. 
N. Yui was supported in part by a Discovery Grant of Natural Science 
Research Council of Canada (NSERC). During the preparation of this
paper, N. Yui held a visiting professorship at various institutions.
This includes Fields Institute (Toronto), Tsuda College (Tokyo) and Nagoya
University (Nagoya).  She thanks these institutions for their
hospitality and support. 

\section{Preamble: weighted diagonal hypersurfaces } 
\label{sect2}

We first recall the definition of weighted projective spaces, 
which are realized as certain singular quotients of the usual projective
spaces. The standard references on weighted projective spaces 
are Dolgachev \cite{Dol} and Dimca \cite{Di}. Let $k$ be a field 
and fix an algebraic closure $\ok$ of $k$. For instance, $k=\QQ$
or $k=\FF_q$, a finite field of characteristic $p=char(k)$ with $q$ elements. 
In this section, we work only over $\ok$ and often omit to specify the 
field of definition. Let $(w_0, w_1,\cdots, w_n)\in\NN^{n+1}$ be a 
{\it weight}. When $k=\FF_q$, assume that $\mbox{gcd}(q,w_i)=1$ for 
every $i$. We may assume without loss of generality that weights are 
{\it normalized}, that is, no $n$ of the $n+1$ weights have common 
divisor $>1$. For each $i, 0\leq i \leq n$, let $\bmu_{w_i}$ denote 
the group of $w_i$-th roots of unity. 
Let $\bmu:=\bmu_{w_0}\times\cdots\times \bmu_{w_n}$ act on the usual
projective $n$-space $\PP^n$ as follows:
For $\g=(g_0,g_1,\cdots, g_n)\in\bmu$ and for the homogeneous
coordinate $\z=(z_0:z_1:\cdots: z_n)$ on $\PP^n$, the action is
given by 
$$(\g, \z)\mapsto (g_0z_0: \cdots: g_nz_n).$$
The quotient $\PP^n/\bmu$ is a weighted projective $n$-space and
denoted by $\PP^n(w_0, \w)=\PP^n(w_0,w_1,\cdots, w_n)$ where we put
$\w=(w_1,\cdots, w_n)$. The usual projective space $\PP^n$
is identified with $\PP^n(1,1,\cdots, 1)$. 
A {\it weighted hypersurface} $V$ is the zero locus of a weighted
homogeneous polynomial. Such $V$ is said to be {\it transversal} 
if the singular locus of $V$ is contained in the singular locus of
$\PP^n(w_0,\w)$, and {\it quasi-smooth} if the affine cone of $V$ 
is smooth outside the vertex. 

In this paper, we take the products of quasi-smooth weighted 
projective hypersurfaces of diagonal type and consider their 
quotients under the action of finite groups. 

Let $d$ be a positive integer such that $w_i\mid d$ for every 
$i$ $(0\leq i\leq n)$ and write $d_i:=d/w_i$. The weighted 
hypersurface, $V$, defined by the equation 
$$c_0x_0^{d_0}+c_1x_1^{d_1}+\cdots +c_nx_n^{d_n}=0\qquad 
(c_i\neq 0)$$ 
will be called a {\it weighted diagonal hypersurface of degree 
$d$ in $\PP^n(w_0,w_1,\cdots, w_n)$}. $V$ is quasi-smooth if and 
only if $d_i\neq 0$ in $k$ for $0\leq i\leq n$. For simplicity, we 
consider the case where the coefficients are $c_0=c_1=\cdots =c_n=1$, 
namely the hypersurface 
$$x_0^{d_0}+x_1^{d_1}+\cdots +x_n^{d_n}=0$$ 
(it may be called the {\it weighted Fermat hypersurface of degree $d$}). 

We note that weighted diagonal hypersurfaces of degree $d$ enjoy the 
same geometry as weighted Fermat hypersurfaces of degree $d$. However, 
their arithmetic properties are different. For instance, when $V$ is a 
surface over $\FF_q$ or $\QQ$, bringing in non-trivial coefficients $c_i$ 
allows us to construct examples of surfaces with particularly small 
(or large) Picard numbers over $\FF_q$ or $\QQ$. 

Write $Q:=(w_0,w_1,\cdots, w_n)$ and let $V$ be the weighted 
diagonal hypersurface of degree $d$ in $\PP^n(Q)$. $V$ is of 
dimension $n-1$ and its properties are investigated by various authors 
(see, for instance, \cite{Go1} and \cite{Yu}). Here we recall its 
cohomology groups. 
Choose a prime $\ell$ different from $\mbox{char}\, (k)$. Let 
$H^i(V,\QQ_{\ell})$ denote the $i$-th $\ell$-adic \'etale 
cohomology group of $V$ over $\ok$. If $i\neq \dim V=n-1$, we 
have 
$$H^i(V,\QQ_{\ell}) \cong 
\begin{cases} 
\QQ_{\ell} & \mbox{ if $i$ is even } \\ 
\{ 0\} & \mbox { if $i$ is odd } 
\end{cases}$$ 
and for $i=n-1$, $H^{n-1}(V,\QQ_{\ell})$ is decomposed into 
a direct sum of one-dimensional subspaces as follows: let 
$H^{n-1}_{prim}(V,\QQ_{\ell})$ denote the primitive part of 
the cohomology $H^{n-1}(V,\QQ_{\ell})$, namely 
$$H^{n-1}(V,\QQ_{\ell})\cong 
\begin{cases} 
H^{n-1}_{prim}(V,\QQ_{\ell}) & \mbox{ if $\dim V$ is odd } \\ 
V(0) \oplus H^{n-1}_{prim}(V,\QQ_{\ell}) & \mbox { if $\dim V$ 
is even } 
\end{cases}$$
where $V(0)$ denotes the subspace corresponding to the hyperplane 
section. Let 
$$\fA_{n-1}(Q):=\biggl\{ \ba =(a_0,a_1,\cdots,a_n)\,|\, a_i\in 
(w_i\ZZ/d\ZZ), a_i\neq 0, \sum_{i=0}^{n} a_i\equiv 0\pmod d 
\biggr\}. $$ 
Then for $V$ over $\ok$, we have a decomposition 
\begin{equation} \label{coh1} 
H^{n-1}_{prim}(V,\QQ_{\ell}) \cong \bigoplus_{\ba \in \fA_{n-1}} 
V(\ba ) 
\end{equation}  
where    
$$V(\ba )=\{ v\in H^{n-1}_{prim}(V,\QQ_{\ell}) \mid  
\gamma^{*}(v)={\zeta}_0^{a_0}{\zeta}_1^{a_1}\cdots  
{\zeta}_n^{a_n} v,\ \forall \gamma 
=({\zeta}_0,{\zeta}_1,\cdots ,{\zeta}_n)\in \Gamma \}$$ 
with 
$$\Gamma :=\mu_{d_0} \times \mu_{d_1} \times \cdots \times 
\mu_{d_n} /(\mbox{diagonal elements}),$$ 
$\mu_d$ is the group of $d$-th roots of unity in $\ok$ and 
$\gamma^{*}$ is the automorphism of $H^{n-1}_{prim}(V,\QQ_{\ell})$ 
induced by $\gamma$. (Here we choose a prime $\ell$ satisfying 
$\ell \equiv 1\pmod{d}$ so that $\mu_d$ can be embedded into 
$\QQ_{\ell}$ multiplicatively.) 

For the product of two weighted diagonal hypersurfaces, we have 
the K\"unneth formula to compute its cohomology and we find 
\begin{equation} \label{coh2} 
H^{j}(V_1\times V_2,\QQ_{\ell}) =\bigoplus_{i_1+i_2=j} 
H^{i_1}(V_1,\QQ_{\ell})\otimes H^{i_2}(V_2,\QQ_{\ell}).  
\end{equation} 
Since $H^{i_1}(V_1,\QQ_{\ell})$ and $H^{i_2}(V_2,\QQ_{\ell})$ 
are decomposed into one-dimensional pieces, so is $H^{j}(V_1\times 
V_2,\QQ_{\ell})$ and each summand takes such a form as $V(\ba ) 
\otimes V(\bb )$. In later sections, we look at the case where 
$V_1$ is a curve and $V_2$ is a surface so that $V_1\times V_2$ 
is a threefold. 

Let $V_1$ and $V_2$ be weighted diagonal hypersurfaces and write 
$Y:=V_1\times V_2$. The direct product $\Gamma_{V_1} \times 
\Gamma_{V_2}$ acts on $Y$ component-wise. In this paper, we 
choose a subgroup, $\Gamma_Y$, of $\Gamma_{V_1} \times 
\Gamma_{V_2}$ and consider the quotient variety 
$$X:=Y/\Gamma_Y.$$ 
Since $\Gamma_Y$ is a finite abelian group, $X$ has at most 
abelian quotient singularities. 

The quotient variety $X=Y/\Gamma_Y$ is usually singular. But if the order 
of $\Gamma_Y$ is invertible in $k$, we can compute the cohomology of $X$ as 
the $\Gamma_Y$-invariant subspace of the cohomology of $Y$. For this, 
we work over an algebraic closure $\ok$ of $k$. Write $p:=
\mbox{char}\, (k)$ and assume $(p,\# \Gamma_Y)=1$. We choose 
a prime $\ell$ satisfying $\ell \equiv 1\pmod{d}$. It is known 
that the $\ell$-adic \'etale cohomology satisfies 
\begin{equation} \label{coh3} 
H^{i}(X,\QQ_{\ell})\cong H^i(Y,\QQ_{\ell})^{\Gamma_Y} 
\end{equation} 
for every $i$ $(0\leq i\leq 2\dim Y)$. (In fact, the 
Hochschild-Serre spectral sequence holds for $Y$ and the 
Galois cohomology for a finite group $\Gamma_Y$ vanishes 
by tensoring with $\QQ_{\ell}$. This yields the isomorphism 
in question.) 

\section{Twist maps} 
\label{sect4}
 
In this section, we recall the method of Hunt and Schimmrigk~\cite{HS} 
to construct $K3$ fibered Calabi--Yau threefolds as quotients of weighted 
hypersurfaces not necessarily of diagonal type, using the twist maps. 
This construction is a generalization of the construction Shioda and 
Katsura~\cite{SK79} for non-singular hypersurfaces in the usual projective
spaces.  
We will describe the construction of quotients of weighted hypersurfaces by
twist maps in any dimension.

Let $V_1$ and $V_2$ be two weighted hypersurfaces defined as follows.

\begin{equation*}
\begin{split}
V_1:=&\{x_0^{\ell} + f(x_1,\cdots, 
x_n)=0\}\subset\PP^n(w_0,\w)\quad\mbox{deg}(V_1)=\ell w_0 \\
V_2:=&\{y_0^{\ell} + g(y_1,\cdots, y_m)=0\}\subset\PP^m(v_0,\mathbf v) 
\quad
\mbox{deg}(V_2)=\ell v_0
\end{split}
\end{equation*}

\noindent where $\w=(w_1,\cdots, w_n)$ and $\mathbf v=(v_1,\cdots, v_m)$ 
and both $f$ and $g$ are assumed to be quasi-smooth, and so are $V_1$ and 
$V_2$. Consider the hypersurface: 

\begin{equation*}
X:=\{f(z_1,\cdots, z_n)-g(t_1,\cdots, t_m)=0\}\subset
\PP^{n+m-1}(v_0 \w,\, w_0\mathbf v).
\end{equation*}
where ${\mbox{deg}}(X)=v_0w_0\ell=v_0{\mbox{deg}}(f)=w_0{\mbox{deg}}(g)$.  

In order to define the twist map we need to assume that $\gcd(v_0,w_0)=1$. This 
condition seems to be missing in \cite{HS}. If this is the case then fix 
$s_0,t_0\in \ZZ$ such that 
$0\leq s_0 <v_0$, $0\leq t_0 < w_0$, $s_0w_0+1\equiv 0 \bmod v_0$ and 
$t_0v_0+1\equiv 0 \bmod w_0$. Let $s=(s_0w_0+1)/v_0$ and $t=(t_0v_0+1)/w_0$. 
Note that $s,t$ are non-zero integers.  

\begin{defn}\label{defn2.1} 
The rational map
\begin{equation*}
\Phi: \PP^n(w_0,\w)\times \PP^m(v_0,\mathbf v)\dashrightarrow 
\PP^{n+m-1}(v_0\w, 
w_0\mathbf v) 
\end{equation*}
\begin{equation*}
\begin{split}
((x_0,x_1,\cdots, x_n),\quad  &(y_0,y_1,\cdots, y_m)) \mapsto \\
 & (x_0^{s_0w_1} y_0^{tw_1} x_1,\cdots, x_0^{s_0w_n}y_0^{tw_n} x_n,\, 
x_0^{sv_1} y_0^{t_0v_1}y_1, \cdots, x_0^{sv_m} y_0^{t_0v_m} y_m)
\end{split}
\end{equation*}
restricted to $V_1\times V_2$ is a generically rational finite map onto 
$X$. The map $\Phi$ is called the {\it twist map}.
\end{defn}

\begin{rem} \label{rem4.1}In \cite{HS} a slightly different definition of the 
twist map is given, namely 

\begin{equation*}
\begin{split}
((x_0,x_1,\cdots, x_n),\quad  &(y_0,y_1,\cdots, y_m)) \mapsto \\
 & (y_0^{w_1/w_0}x_1,\cdots, y_0^{w_n/w_0}x_n,\, x_0^{v_1/v_0}y_1,
\cdots, x_0^{v_m/v_0}y_m).
\end{split}
\end{equation*} 

If $v_0$ or $w_0$ is different from 1 then one takes some $v_0$-th or $w_0$-th
roots 
of $x_0$ or $y_0$. It is not directly clear that this map is a rational map, 
that is, can be given in terms of polynomials. In \cite{HS} it is then argued 
that this map is well-defined, but no proof is given for the fact that this map 
is given by polynomials. 
Below, we construct a counterexample to the claim of Hunt and Schimmrigk 
where the twist map is neither a polynomial map nor well-defined. 
\end{rem}

\begin{rem}\label{rem4.2}
The rational map $\Phi$ can be extend to points with $x_0=0$ and $y_0\neq0$ as
follows
\begin{equation*}
\begin{split}
((0,x_1,\cdots, x_n),\quad  &(y_0,y_1,\cdots, y_m)) \mapsto \\
 & ( x_1,\cdots, x_n,\, 0,
\cdots, 0)
\end{split}
\end{equation*}
A similar extension exists for points with $x_0\neq 0$ and $y_0=0$. The map
$\Phi$ is 
not defined at points with $x_0=y_0=0$. 
\end{rem}

\begin{rem}\label{rem4.3}
The condition $\gcd(v_0,w_0)=1$ is necessary, both for our definition and for
the 
definition in \cite{HS}.  We give an example for which the twist map, as
defined 
in \cite{HS}, is not well-defined. 

Take $\mathbf{w}=\mathbf{v}=(2,1,1)$. Let $V_1=\{ x_0^2+x_1^4+x_2^4 \}$ and 
$V_2=\{y_0^2+y_1^4+y_2^4\}$. The twist map of \cite{HS} in this case should be
the 
rational map $\PP^2(2,1,1) \times \PP^2(2,1,1) \rightarrow
\PP^3(2,2,2,2)=\PP^3$ 
given by 
\[ (x_0,x_1,x_2)\times (y_0,y_1,y_2) \mapsto 
( \sqrt{y_0} x_1,\sqrt{y_0}x_2,\sqrt{x_0}y_1,\sqrt{x_0}y_2).\]
This map is not well-defined, since it depends on the choice of $\sqrt{x_0y_0}$.
\end{rem} 

Now we let the group $\bmu_{\ell}$ of 
$\ell$-th roots of unity act on $V_1\times V_2\subset
\PP^{n}(w_0,\w)\times \PP^m(v_0,\mathbf v)$ by $\gamma\in\bmu_{\ell}$. 
The action is defined as follows. 

\begin{defn}\label{defn2.2}
Assume that $\gcd (w_0, v_0,\ell)=1$. Then the group $\bmu_{\ell}$ 
acts on $V_1\times V_2$ by    
\begin{equation*}
(\gamma, (x_0:\cdots: x_n),\, 
(y_0:\cdots: y_m))\mapsto ((\gamma x_0:x_1,\cdots: x_n),
(\gamma y_0:y_1,\cdots:y_m)).
\end{equation*}
for every $\gamma \in \bmu_{\ell}$. 
\end{defn}  

The quotient space $V_1\times V_2/{\bmu_{\ell}}$ is a projective 
variety and the rational map $V_1\times V_2\to X$ is generically 
$\ell: 1$.

$$\begin{array}{ccc} 
V_1\times V_2 & \small{\Phi \mid_{V_1\times V_2}} & \\ 
\downarrow & \searrow  & \\ 
V_1\times V_2/{\bmu_{\ell}} & \cdots \rightarrow & X 
\end{array}$$ 

Now we discuss singularities on the varieties $V_1\times V_2/ 
{\bmu_{\ell}}$ and $X$. First we know that only singularities 
occurring in ambient weighted projective spaces are cyclic quotient 
singularities. Since $f$ and $g$ are quasi-smooth, so are $V_1$ and 
$V_2$, and they have cyclic quotient singularities all due to the 
ambient spaces. Further, a threefold $X$ is defined by $f-g$ and 
$f$ and $g$ have no common variable. Hence $X$ is also quasi-smooth 
and it possesses at most cyclic quotient singularities. 

Cyclic quotient singularities are resolved by toroidal resolutions 
and moreover quasi-smooth varieties can be desingularized by 
applying toroidal resolutions to their ambient spaces. An example 
is that we obtain a resolution of $X$ by restricting a resolution 
of $\PP^{n+m-1}(v_0 \w,\, w_0\mathbf v)$ onto $X$. We note that it 
is often sufficient to take a partial resolution of 
$\PP^{n+m-1}(v_0 \w,\, w_0\mathbf v)$ to desingularize $X$; we 
explain this in Section \ref{sect-resol}. 

On the other hand, $V_1\times V_2$ and $V_1\times V_2/{\bmu_{\ell}}$ 
are no longer quasi-smooth in most cases and they have abelian 
quotient singularities. As $V_1\times V_2/{\bmu_{\ell}}$ is 
birational to $X$, we may work on $X$ (rather than on 
$V_1\times V_2/{\bmu_{\ell}}$) to construct their smooth models 
and discuss their Calabi--Yau properties. 

Let $\pi_X: \tilde X\to X$ be a smooth resolution of $X$.  
A natural question is {\bf When is $\tilde X$ Calabi--Yau?}
\smallskip 

In search of an answer to this question, we will look into the 
fibrations. Project the quotient $V_1\times V_2/\bmu_{\ell}$ (via 
$V_1/\mu_{\ell}\times V_2/\mu_{\ell}$) to the first and the second 
components, respectively. This gives rise to the following two 
rational fibrations: 

\begin{equation*}
\phi_1:\tilde X\to V_1/\bmu_{\ell}\quad\mbox{and}\quad 
\phi_2:\tilde X\to V_2/\bmu_{\ell}.
\end{equation*}

$$\begin{matrix} \quad &  X & \quad \\
                 \phi_1\swarrow & \quad & \searrow \phi_2 \\
                 \quad & \qquad & \qquad \\
                  V_1/{\mu_{\ell}}\quad\quad & \quad \quad \qquad &
V_2/{\mu_{\ell}}
                \end{matrix}$$
If $\pi_X :\tilde X\to X$ is a smooth resolution, then the composite map 
$\phi_1 \circ \pi_X$ is a fibration of $\tilde X$ onto 
$V_1/\mu_{\ell}$ whose generic fibers are copies of resolutions of $V_2$. 
A similar property holds for the fibration $\phi_2$. The situation is 
illustrated as follows, where $Y_i$ denotes a smooth resolution of 
$V_i/\mu_{\ell}$: 
$$\begin{matrix} 
          X &\overset{\pi_X}\longleftarrow &\tilde X \\
          \phi_i\downarrow& \quad\quad &  \\ 
          V_i/\mu_{\ell} &\overset{\pi_i}\longleftarrow & Y_i 
\end{matrix} $$
           
\begin{prop}\label{prop4.1}
Suppose that $\tilde X$ is Calabi--Yau and $w_0>1$. Write $\tilde{\phi_1}:=
\phi_1 \circ \pi_X$ and let $\tilde V_2$ be the generic smooth fiber of 
$\tilde{\phi_1}$. Then the rational fibration $\tilde{\phi_1}: \tilde X\to 
V_1/\mu_{\ell}$ lifts to a genuine fibration $\tilde X\to Y_1$ for some 
smooth resolution $Y_1$ of $V_1/\mu_{\ell}$, if and only if $\tilde V_2$ 
is also Calabi--Yau. A similar property holds for $\phi_2$.
\end{prop}
\begin{proof} (Cf. Lemma 3.4 in \cite{HS}.)  
\end{proof} 

The upshot of Proposition \ref{prop4.1} is that it provides a method of 
constructing split-type (product-type) Calabi--Yau varieties. That is, 
a covering of a product is Calabi--Yau if and only if one of the 
components is Calabi--Yau. 

\begin{prop}\label{prop4-2} 
Let $V_1$, $V_2$ and $X$ be quasi-smooth varieties as before. Then the 
following assertions hold. 

{\em (1)} A sufficient condition for $\tilde V_2$ to be Calabi--Yau  
hypersurface is: 
\begin{equation*}
v_0+v_1+\cdots + v_m=\ell v_0=\deg (V_2), 
\end{equation*}

{\em (2)} A sufficient condition for $\tilde X$ to be Calabi--Yau 
hypersurface is:
\begin{equation*}
\ell w_0 v_0= w_0\sum_{i=1}^m v_i + v_0\sum_{j=1}^n w_j. 
\end{equation*}
\end{prop}
\begin{proof} (Cf. \cite{Dol}) The Calabi--Yau (sufficient) condition for a 
hypersurface in weighted projective spaces is that
the sum of all weights equal the degree of the variety. 
\end{proof}

\section{Construction of $K3$-fibered Calabi--Yau threefolds via twist maps} 
\label{sect5}

Now we apply the construction by twist maps to elliptic curves, 
$K3$ surfaces, and Calabi--Yau threefolds. 
\medskip 

\noindent{\bf Dimension $1$ Calabi--Yau varieties (Elliptic curves)}:
There are three elliptic curves in weighted projective spaces, that are of
diagonal form. 
They are given as in Table \ref{table1}. All three elliptic curves
have complex multiplication, the first and the third by $\ZZ[\sqrt{-3}]$
and the second by $\ZZ[\sqrt{-1}]$.

\begin{table} 
\caption{Elliptic curves in weighted projective 2-spaces}  
\label{table1} 
\[
\begin{array}{c|c|c}\hline \hline
\# & E_i &  \mbox{equation} \\ \hline
1 & E_1 & y_0^3+y_1^3+y_2^3=0\subset\PP^2(1,1,1) \\ \hline
2 & E_2 & y_0^4+y_1^4+y_2^2=0\subset\PP^2(1,1,2) \\ \hline 
3 & E_3 & y_0^6+y_1^3+y_2^2=0\subset\PP^2(1,2,3) \\ \hline
\hline
\end{array}
\] 
\end{table} 
\smallskip 

\noindent{\bf Dimension $2$ Calabi--Yau varieties ($K3$ surfaces)}: 
Let $C_{(w_0,w_1,w_2)}\in \PP^2(w_0,w_1,w_2)$ be a curve defined by
\begin{equation*}
C_{(w_0,w_1,w_2)}:=\{x_0^{\ell}+f(x_1,x_2)=0\}\subset\PP^2(w_0,w_1,w_2)
\quad \mbox{with deg $C=\ell w_0$} 
\end{equation*}
and the group $\ZZ/\ell \ZZ$ acts on $C$. Then applying the twist map to 
$C\times E_i$ where $i=1$ (resp. $3$) if $\ell=3$ (resp. $6$), and $i=2$ 
if $\ell=4$. Each $E_i$ has the automorphism group $\ZZ/\ell\ZZ$. We take 
the quotient of the product $C\times E_i$ under the action of the group
$\ZZ/\ell\ZZ$. 
Then we get a hypersurface $X\subset\PP^3(v_0w_1,v_0w_2, w_0v_1,w_0v_2)$ of 
degree $d$, where $(v_0,v_1,v_2)\in\{(1,1,1), (1,1,2), (1,2,3)\}$.

\begin{prop}\label{prop5-1}
There are eleven $K3$ surfaces arising from this construction. See 
{\em Table \ref{table2}} for the list, where we put $(k_0,k_1,k_2,k_3)= 
(v_0w_1,v_0w_2,w_0v_1,w_0v_2)$, $d=\sum_{i=0}^3 k_i$. 
For $f(x_1,x_2)$ and $g(y_1,y_2)$, we may take, for instance,
$$f(x_1,x_2)=x_1^{\ell w_0/w_1}+x_2^{\ell w_0/w_2}\quad\mbox{and}\quad
g(y_1,y_2)=-(y_1^{\ell v_0/v_1}+y_2^{\ell v_0/v_2}).$$
In this case, $C_{(w_0,w_1,w_2)}$ is covered by the diagonal
curve of degree $d$ in $\PP^2(1,1,1)$.
\end{prop}

\begin{table} 
\caption{$K3$ surfaces appearing in our construction} 
\label{table2}  
\[
\begin{array}{c|c|c|c|c|c|c}\hline \hline
\# & C_{(w_0,w_1,w_2)} & g(C) & E_i & \ell  & (k_0,k_1,k_2,k_3) & d \\ \hline
1  & (2,1,1)         & 4    & E_1 & 3     & (1,1,2,2)         & 6 \\ \hline
2  &                 &      & E_2 & 4     & (1,1,2,4)         & 8 \\ \hline
3  &                 &      & E_3 & 6     & (1,1,4,6)         & 12 \\ \hline
4  & (3,1,2)         & 7    & E_2 & 4     & (1,2,3,6)         & 12 \\ \hline
5  &                 &      & E_3 & 6     & (1,2,6,9)         & 18 \\ \hline
6  & (4,1,3)         & 3    & E_1 & 3     & (1,3,4,4)         & 12 \\ \hline
7  &                 &      & E_3 & 6     & (1,3,8,12)        & 24 \\ \hline
8  & (5,1,4)         & 6    & E_2 & 4     & (1,4,5,10)        & 20 \\ \hline
9  & (7,1,6)         & 15   & E_3 & 6     & (1,6,14,21)       & 42  \\ \hline
10 & (5,2,3)         & 11   & E_3 & 6     & (2,3,10,15)       & 30  \\ \hline
11 & (11,5,6)        & 5    & E_3 & 6     & (5,6,22,33)       & 66  \\ \hline
\hline
\end{array}
\]
\end{table}  

\begin{rem}\label{rem5-1}
All $K3$ surfaces can be realized as orbifolds of diagonal 
or quasi-diagonal hypersurfaces. For instance, $\#3$ may be realized 
by a diagonal hypersurface
$z_1^{12}+z_2^{12}+z_3^3+z_4^2=0\subset\PP^2(1,1,4,6)$.
Similarly $\#8$ may be realized by a diagonal hypersurface
$z_1^{20}+z_2^5+z_3^4+z_4^2=0\subset\PP^3(1,4,5,10)$.
(By Goto \cite{Go1}, Yonemura \cite{Yo}, there are in total $14$ 
weighted diagonal K3 surfaces obtained as quotients of diagonal 
hypersurfaces by finite abelian group actions.)
The last example $\#11$ can be realized by the polynomial
$$z_0^{12}z_1+z_1^{11}+z_2^3+z_3^2=0\subset 
\PP^3(5,6,22,33).$$
Goto \cite{Go2} considered $K3$ surfaces of the latter kind, the so-called
{\it quasi-diagonal} $K3$ surfaces.
\end{rem}
\smallskip

\noindent{\bf Dimension $3$ Calabi--Yau varieties (Calabi--Yau threefolds)}: 
We now apply the twist map to construct $K3$-fibered Calabi--Yau threefolds,
which are the quotients of either $S\times E$, where $S$ is a surface 
and $E$ is an elliptic curve, or $C\times Y$ where $C$ is a curve
and $Y$ is a $K3$ surface in weighted projective spaces, under the action of
the group $\ZZ/\ell\ZZ$. 

\noindent{\bf (A) Elliptic fibered Calabi--Yau threefolds $S\times E$}:  
Let $(w_0,\w)=(w_0,w_1,w_2,w_3)$ be a weight, and consider the surface
\begin{equation} \label{eqn-s}
S_{(w_0,w_1,w_2,w_3)}:=\{x_0^{\ell}+f(x_1,x_2,x_3)=0\}\subset\PP^3 
(w_0,w_1,w_2,w_3)
\end{equation}
of degree $\ell w_0$.

Apply the twist map to the product $S_{(w_0,w_1,w_2,w_3)}\times E_i$
where $i=1$ (resp. $3$) if $\ell=3$ (resp. $6$), and $i=2$ if $\ell=4$ 
onto a threefold of degree $=\ell w_0$, and for the sake of simplicity,
we may take $g(y_1,y_2)=-(y_1^{\ell v_0/v_1}+y_2^{\ell v_0/v_2})$.
\smallskip

A natural question is : {\bf What combination of weights $(w_0,w_1,w_2,w_3)$ 
and $(v_0,v_1,v_2)$ would give rise to Calabi--Yau threefolds?}
\smallskip

We divide our constructions into two cases:

$1^{\circ}$: $S_{(w_0,w_1,w_2,w_3)}$ is not a $K3$ surface.

$2^{\circ}$: $S_{(w_0,w_1,w_2,w_3)}$ is a $K3$ surface.

\begin{prop} \label{prop5-2}
Let $E_i,\, i\in\{1,2,3\}$ be elliptic curves in {\em Table 
\ref{table1}}. Suppose that $S_{(w_0,w_1,w_2,w_3)}$ is not a $K3$ surface.
Then there are $14$ elliptically but not $K3$-fibered Calabi--Yau 
threefolds obtained as quotients of $S_{(w_0,w_1,w_2,w_3)}\times E_i$.
Here we put $(k_0,k_1,k_2,k_3,k_4)=(v_0w_1,v_0w_2,v_0w_3, w_0v_1,w_0v_2)$
and $d=\sum_{i=0}^4 k_i$. We may take, for instance, 
$$f(x_1,x_2,x_3)=x_1^{\ell w_0/w_1}+x_2^{\ell w_1/w_0}+x_3^{\ell w_0/w_3}.$$ 
With this choice, $S_{(w_0,w_1,w_2,w_3)}$ is a weighted diagonal surface,
though not $K3$.
\end{prop} 

\begin{table} 
\caption{Elliptic but not $K3$-fibered Calabi-Yau threefolds}  
\label{table3} 
\[
\begin{array}{c|c|c|c|c|c|c|c|c}\hline \hline
\# & (w_0,w_1,w_2,w_3) & E_i & \ell & (k_0,k_1,k_2,k_3,k_4) & d  & h^{1,1} &
h^{2,1} 
& \chi 
\\ \hline
1 & (5,1,1,3)      & E_1 &  3   & (1,1,3,5,5)           & 15 &  7 & 103 & -192
\\
2 &                & E_3 &  6   & (1,1,3,10,15)         & 30 &  5 & 251 & -492 
\\ \hline
3 & (5,1,1,2)      & E_2 &  4   & (1,2,2,5,10)          & 20 &  6 & 242 & -228
\\
4 &                & E_3 &  6   & (1,2,2,10,15)         & 30 &  4 & 208 & -408 
\\ \hline
5 & (7,1,2,4)      & E_2 &  4   & (1,2,4,7,14)          & 28 &  8 & 116 & - 216 
\\ \hline
6 & (7,1,3,3)      & E_1 &  3   & (1,3,3,7,7)           & 21 &  17& 65  & -96 \\ 
7 &                & E_3 &  6   & (1,3,3,14,21)         & 42 &  6 & 180 & -348 
\\ \hline
8 & (7,2,2,3)      & E_3 &  6   & (2,2,3,14,21)         & 42 &  7 & 151 & -288 
\\ \hline
9 & (9,1,4,4)      & E_2 &  4   & (1,4,4,9,18)          & 36 & 19 & 91  & -144 
\\ \hline
10 & (10,2,3,5)    & E_1 &  3   & (2,3,5,10,10)         & 30 & 12 & 48  & -108
\\
11 &               & E_3 &  6   & (2,3,5,20,30)         & 60 & 10 & 106 & -192 
\\ \hline
12 & (10,1,3,6)    & E_1 &  3   & (1,3,6,10,10)         & 30 & 19 & 67  & -96 
\\ 
13 &               & E_3 &  6   & (1,3,6,20,30)         & 60 & 10 & 178 & -336 
\\ \hline
14 & (13,1,6,6)    & E_3 &  6   & (1,6,6,26,39)         & 78 & 23 & 143 & -240 
\\ \hline 
\hline
\end{array}
\]
\end{table} 

The next proposition gives examples of elliptically and $K3$-fibered
Calabi--Yau threefolds, which generalize the construction of Livn\'e--Yui
\cite{LY05}.

\begin{prop}\label{prop5-3} Let $E_i\, (i\in\{1,2,3\})$ be 
elliptic curves in {\em Table \ref{table1}}, and let 
$S_{(w_0,w_1,w_2,w_3)}$ be a $K3$ surface. Then there are $23$ 
elliptically fibered Calabi--Yau threefolds obtained as quotients of 
$S_{(w_0,w_1,w_2,w_3)}\times E_i$ where $i\in\{1,2,3\}$, which have 
also $K3$-fibrations are tabulated in {\em Table \ref{table4}}. 
Here we put again $(k_0,k_1,k_2,k_3,k_4)=(v_0w_1,v_0w_2,v_0w_3,w_0v_1,w_0v_2)$
and $d=\sum_{i=0}^4 k_i$. 

For the polynomials $f(x_1,x_2,x_3)$ and $g(y_1,y_2)$,
we may take
$$f(x_1,x_2,x_3)=x_1^{\ell w_0/w_1}+x_2^{\ell w_0/w_2}+x_3^{\ell w_0/w_3},\qquad
g(y_1,y_2)=-(y_1^{\ell v_0/v_1}+y_2^{\ell v_0/v_2}).$$
In this case, $S_{(w_0,w_1,w_2,w_3)}$ is covered by the diagonal surface
of degree $d$ in $\PP^3(1,1,1,1)$.
\end{prop} 

\begin{table} 
\caption{Elliptically and $K3$ fibered Calabi--Yau threefolds}  
\label{table4} 
\[
\begin{array}{c|c|c|c|c|c|c|c} \hline \hline
\# & (w_0,w_1,w_2,w_3) & E_i & \ell & (k_0,k_1,k_2,k_3,k_4) & d & \chi & 
K3-\mbox{fiber}\\ \hline
1 & (3,1,1,1)         & E_1 &  3   & (1,1,1,3,3)           & 9 & -216 &
(1,1,1,3)\\
2 &                   &E_2  &  4   & (1,1,1,3,6)           & 12 & -324 & \\
3 &                   & E_3 &  6   & (1,1,1,6,9)           & 18 & -540 & \\
\hline
4 & (4,1,1,2)         & E_1 &  3   & (1,1,2,4,4)           & 12 & -192 &
(1,1,2,4) 
\\
5 &                   & E_2 &  4   & (1,1,2,4,8)           & 16 & -288 & \\
6 &                   & E_3 &  6   & (1,1,2,8,12)          & 24 & -480 & \\
\hline
7 & (6,1,1,4)         & E_2 &  4   & (1,1,4,6,12)          & 24 & -312 &
(1,1,4,6) 
\\ 
8 &                   & E_3 &  6   & (1,1,4,12,18)         & 36 & -528 &  \\
\hline
9 & (6,1,2,3)         & E_1 &  3   & (1,2,3,6,6)           & 18 & -144 &
(1,2,3,6) 
\\
10 &                   & E_2 &  4   & (1,2,3,6,12)          & 24 & -480 &  \\
11 &                   & E_3 &  6   & (1,2,3,12,18)         & 36 & -360 & \\
\hline
12 & (8,1,1,6)         & E_1 &  3   & (1,1,6,8,8)           & 24 & -240 &
(1,3,4,4) 
\\
13 &                   & E_3 &  6   & (1,1,6,16,24)         & 48 & -624 & \\
\hline
14 & (8,1,3,4)         & E_3 &  6   & (1,3,4,16,24)         & 48 & -312 &
(1,1,4,6) 
\\ \hline
15 & (8,2,3,3)         & E_3 &  6   & (2,3,3,16,24)         & 48 & -240 &
(1,3,8,12) 
\\ \hline
16 & (9,1,2,6)         & E_2 &  4   & (1,2,6,9,18)          & 36 & -228 &
(1,2,6,9) 
\\
17 &                   & E_3 &  6   & (1,2,6,18,27)         & 54 & -408 & \\
\hline
18 & (9,2,3,4)         & E_2 &  4   & (2,3,4,9,18)          & 36 & -120 &
(1,2,3,6) 
\\ \hline
19 & (10,1,1,8)        & E_2 &  4   & (1,1,8,10,20)         & 40 & -432 &
(1,4,5,10) 
\\ \hline
20 & (10,3,3,4)        & E_3 &  6   & (3,3,4,20,30)         & 60 & -192 &
(2,3,10,15) 
\\ \hline
21 & (12,1,2,9)        & E_1 &  3   & (1,2,9,12,12)         & 36 & -168 &
(1,3,4,4) 
\\
22 &                   & E_3 &  6   & (1,2,9,24,36)         & 72 & -240 & \\
\hline
23 & (14,1,1,12)       & E_3 &  6   & (1,1,12,28,42)        & 84 & -960 &
(1,6,14,21) 
\\ \hline
\hline
\end{array}
\]
\end{table} 

\begin{rem}\label{rem5-2}
With our choice of the polynomials $f(x_1,x_2,x_3)$ and $g(y_1,y_2)$,
all Calabi--Yau threefolds constructed in Propositions 5.2 and 5.3
can also be realized as orbifolds of diagonal hypersurfaces 
defined by the equation: 
\begin{equation*}
V: x_0^d+ x_1^d+x_2^d+x_3^d+x_4^d=0\subset\PP^4
\end{equation*}
of degree $d$. Take a finite abelian group
$\bmu=\bmu_{k_0}\times\cdots\times \bmu_{k_4}$, where
$\bmu_{k_i}=\mbox{Spec} (\QQ[T]/(T^{k_i}-1)$. Now we impose the 
condition that each $k_i$ divides $d$. Then $\bmu$ acts on $V$ 
component-wise. Then a smooth resolution of the quotient $V/\bmu$ 
is a Calabi--Yau threefold in the weighted projective space 
$\PP^4(k_0,k_1,k_2,k_3,k_4)$ only when $d=k_0+k_1+\cdots + k_4$.  
This construction was carried out in Yui \cite{Yu}. 

With different choices of homogeneous polynomials $f(x_1,x_2,x_3)$
and $g(y_1,y_2)$, we may obtain more defining equations for these
Calabi--Yau hypersurfaces.
\end{rem} 

Next we list Calabi--Yau threefolds with elliptic fibrations
that are not realized as orbifolds of diagonal hypersurfaces.

\begin{prop}\label{prop5-4}
Take the product $S_{(w_0,w_1,w_2,w_3)}\times E_3$, where 
$S_{(w_0,w_1,w_2,w_3)}$ is defined in $(\ref{eqn-s})$, and $E_3$ 
is the elliptic curve defined in {\em Table \ref{table1}}. 
Consider the image of the twist map
\begin{equation*}
\PP^3(w_0,w_1,w_2,w_3)\times \PP^2(1,2,3)\to \PP^4(w_1,w_2,w_3,2w_0,3w_0)
\end{equation*}
Examples of elliptic fibered Calabi--Yau threefolds with constant fiber
modulus $E_3$ with positive Euler characteristic are listed in {\em Table 
\ref{table5}}. 
\end{prop} 

\begin{table} 
\caption{Calabi--Yau threefolds with fiber $E_3$ and positive Euler 
characteristic}   
\label{table5}
\[
\begin{array}{c|c|c|c|c|c|c} \hline \hline
\# & (w_0,w_1,w_2,w_3) & (k_0,k_1,k_2,k_3,k_4) & d & h^{1,1} & h^{2,1} & \chi \\
\hline
24 & (581,41,42,498) & (41,42,498, 1162,1743) & 3486 & 491  & 11        & 960  
\\ \hline
25 & (498,36,41,421)& (36,41,421,996,1494)    & 2988 & 491  & 11        & 960  
\\ \hline
26  & (539,36,41,462)& (36,41,462,1078,1617)   & 3234 & 462  & 12        & 900  
\\ \hline 
27  &(463,31,41,391)& (31,41,391,926,1389)    & 2778 & 462  & 12        & 900  
\\ \hline
28 & (433,31,36,366)& (31,36,366,866,1299)    & 2598 & 433  & 13        & 840  
\\ \hline
29 & (414,24,41,349)& (24,41,349,828,1242)    & 2484 & 416  & 14        & 804  
\\ \hline
30 & (385,28,31,326)& (28,31,326,770,1155)    & 2310 & 387  & 15        & 744  
\\ \hline
31 & (372,18,41,313)& (18,41,313,744,1116)    & 2232 & 377  & 17        & 720  
\\ \hline
\hline
\end{array}
\]
\end{table} 

\begin{rem}\label{rem5-3} The examples listed in Table 5 are not realizable as 
orbifolds of diagonal hypersurfaces.  These provide examples of Calabi--Yau
threefolds with large positive Euler characteristics.  In fact, it
realizes the largest positive Euler characteristic $960$ known today
for Calabi--Yau threefolds.  
\end{rem}

\begin{proof} (Propositions \ref{prop5-2}, \ref{prop5-3} and \ref{prop5-4}.) 
We test the sufficient condition in Proposition \ref{prop4-2} (2) for 
the product $S_{w_0,w_1,w_2,w_3}\times E_i$ ($i=1,2,3)$ to be Calabi--Yau. 
\end{proof} 

\noindent{\bf (B) $K3$ fibered Calabi--Yau threefolds $C\times Y$}: 
Now we construct Calabi--Yau threefolds with $K3$-fibrations.

Let $Y_i$ be one of the eleven $K3$ surfaces constructed in Table 
\ref{table2}, where $i$ corresponds to the number $\#$ in Table \ref{table2}. 
Pick $i$, and let $(k_0,k_1,k_2,k_2)$ be a weight, and 
$d=k_0+k_1+k_2+k_3=\ell w_0$.  
If $i\neq 11$, for each $i$, $Y_i$ is defined by

\begin{equation*}
Y_i:\{y_0^{d/k_0}+y_1^{d/k_1}+y_2^{d/k_2}+y_3^{d/k_3}=0\}\subset\PP(k_0,k_1,
k_2,k_3)
\end{equation*}

We consider the product $C_{(w_0,w_1,w_2)}\times Y_i$ where $C_{(w_0,w_1,w_2)}$
is not an elliptic curve.  Suppose
that $\ZZ/\ell\ZZ$ is a subgroup of the automorphism group
$\mbox{Aut}(Y_i)$. 

Here is a typical example.  Let

\begin{equation*}
\begin{split}
Y_1:&=\{y_0^6+y_1^6+y_2^3+y_3^3=0\}\subset\PP^3(1,1,2,2)\quad \mbox{with 
$\ell=6$} \\
Y_4:&=\{y_0^{12}+y_1^6+y_2^4+y_3^2=0\}\subset\PP^3(1,2,3,6)\quad\mbox{with 
$\ell=12$} \\ 
Y_6:&=\{y_0^{18}+y_1^9+y_2^3+y_3^2=0\}\subset\PP^3(1,2,6,9)\quad\mbox{with 
$\ell=18$} \\
Y_9:&=\{y_0^{42}+y_1^7+y_2^3+y_3^2=0\}\subset\PP^3(1,6,14,21)\quad\mbox{with 
$\ell=42$} \\
Y_{10}:&=\{y_0^{15}+y_1^{10}+y_2^3+y_3^2=0\}\subset\PP^3(2,3,10,15)
\quad\mbox{with $\ell=15$} 
\end{split}
\end{equation*}
For $i=11$, the $K3$ surface $Y_{11}$ is given by
\begin{equation*}
Y_{11}:=\{y_0^{12}y_1+y_1^{11}+y_2^3+y_3^2=0\}\subset\PP^3(5,6,22,33)\quad 
\mbox{with $\ell=12$}
\end{equation*}

We will determine the lattice structures of the above K3 surfaces in a later
section.

\begin{rem} 
Among the $K3$ surfaces listed above, we know at least that $Y_9$ has a 
unimodular lattice; i.e., the Picard lattices (which coincide with the 
N\'eron-Severi groups for $K3$ surfaces) of the minimal resolutions of 
them are unimodular. See \cite{Go2}. 
\end{rem}  

\begin{prop}\label{prop5-8}
Take the product $C_{(v_0,v_1,v_2)}\times Y_i$ where $C_{(v_0,v_1,v_2)}
\subset\PP^2(v_0,v_1,v_2)$ is a curve and $Y_i$ $(i\in\{1,2,\cdots,10\})$ 
is a $K3$ surface. Then {\em Table \ref{table6}}  
gives examples of $K3$-fibered Calabi--Yau threefolds obtained by the 
twist map. Here we put $(k_0,k_1,k_2,k_3,k_4)=    
(v_0w_1,v_0w_2,v_0w_3,w_0v_1,w_0v_2)$ and $d=\sum_{i=0}^4 k_i$. 
$($Under the Remark, Calabi--Yau threefolds labeled ``new'' are those not 
found in the CYdata of Kreuzer and Skarke \cite{KS}.$)$ 
\end{prop} 

\begin{table} 
\caption{Examples of $K3$-fibered Calabi--Yau threefolds}  
\label{table6} 
\[
\begin{array}{c|c|c|c|c|c|c|c} \hline \hline
\# & (v_0,v_1,v_2) &  Y_i & \ell & (k_0,k_1,k_2,k_3,k_4) & d   &  \chi  & 
\mbox{Remark}\\ \hline
1 & (2,1,1)           &  Y_1 & 6    & (1,1,2,4,4)           & 12  &  -192  & \\ 
2 &               &  Y_4 & 12   & (1,1,4,6,12)          & 24  & -312   & \\ 
3 &               &  Y_6 & 18   & (1,1,4,12,18)         & 36  & -528   & \\ 
4 &               &  Y_9 & 42   & (1,1,12,28,42)        & 84  & -960   & \\
\hline 
5 & (3,2,1)           &  Y_1 & 6    & (1,3,3,6,6)           & 18  & -144   & \\ 
6 &               &  Y_4 & 12   & (1,2,6,9,18)          & 36  & -228   & \\
7 &               &  Y_6 & 18   & (1,2,8,18,27)         & 54  & -408   & \\
8 &               &  Y_9 & 42   & (1,2,18,42,63)        & 126 & -720   & \\ 
9 &               &  Y_{10}& 15 & (2,4,9,30,45)         & 90  &        &
\mbox{new} 
\\ \hline 
10 & (4,1,3)       &  Y_1  & 6   & (1,3,4,8,8)           & 24  & -120   & \\ 
11 &               &  Y_9  & 42  & (1,3,24,56,84)        & 168 & -624   & \\
\hline
12  & (5,1,4)       &  Y_4  & 12  & (1,4,10,15,30)        & 60  & -168   &  \\ 
13 & (5,1,4)       &  Y_6  & 18  & (1,4,10,30,45)        & 90  &        & 
\mbox{new} \\ 
14 &               &  Y_{10} & 15& (2,8,15,50,75)        & 150 &        & 
\mbox{new} \\ \hline
15  & (7,1,6)       &  Y_6   & 18 & (1,6,14,42,63)        & 126 & -720   & \\ 
16 &               &  Y_9  & 42  & (1,6,42,98,147)       & 294 & -480   & \\ 
17 &               & Y_{10} & 15 & (2,12,21,70,105)      & 210 & -384   & \\
\hline
18 & (5,2,3)       & Y_1   & 6   & (2,3,5,10,10)         & 30  & -72    & \\
19 &               & Y_6   & 18  & (2,3,10,30,45)        & 90  & -216   & \\ 
20 &               & Y_9   & 42  & (2,3,30,70,105)       & 210 & -384   & \\
\hline 
\hline
\end{array}
\]
\end{table} 

\begin{proof} We test the sufficient condition that
$\ell w_0v_0=v_0\sum_{i=1}^3 w_i+w_0\sum_{j=1}^2 v_j$ where $(w_0,w_1,w_2,w_3)$
is a weight for the $K3$ surface $Y_i$ for $i=1,4,6,9,10$.  
\end{proof} 

Now we consider a particular $K3$ surface.
Let $S$ be a quasi-diagonal surface in $\PP^3 (5,6,22,33)$ of degree 66 
defined by the equation 
$$y_0^{12}y_1 +y_1^{11} +y_2^3+y_3^2 =0.$$ 
This is the $K3$ surface $\# 11$ in Table 2. We consider the product 
$S\times E_i$ with $i=1,2,3$. Let $(y_0:y_1:y_2:y_3)$ and $(x_0:x_1:x_2)$ 
be the projective coordinates for $S$ and $E_i$, respectively.
The split map requires that the first variables $x_0$ and $y_0$ 
of $S$ and $E_i$, respectively, should have the same degree $\ell$. 
As elliptic curves $E_i$ are defined by diagonal equations, and possible 
degrees for $x_0$ are $2,3,4$ and $6$. Hence to meet the degree condition 
for $S$ and $E$, $w_0$ should be either $22$ or $33$ and $S$ be defined by 
the equation 
$$S_1:\ y_0^3+y_1^{12}y_2 +y_2^{11}+y_3^2 =0\ \subset \PP^3 (22,5,6,33)$$ 
or 
$$S_2:\ y_0^2+y_1^{12}y_2 +y_2^{11}+y_3^3 =0\ \subset \PP^3 (33,5,6,22).$$ 
Then the following four pairs satisfy the degree condition: 
$$\begin{array}{ccl} 
S_1\times E_1 & \subset & \PP^3(22,5,6,33) \times \PP^2(1,1,1) \\ 
S_1\times E_3 & \subset & \PP^3(22,5,6,33) \times \PP^2(2,1,3) \\ 
S_2\times E_2 & \subset & \PP^3(33,5,6,22) \times \PP^2(2,1,1) \\ 
S_2\times E_3 & \subset & \PP^3(33,5,6,22) \times \PP^2(3,1,2) 
\end{array}$$ 
Among these four cases, only the third and fourth cases actually yield 
Calabi-Yau threefolds as explained below: 

\noindent (i) The third case $S_2\times E_2$ 
$$\begin{array}{lcl} 
\PP^3(33,5,6,22) & & \PP^2(2,1,1) \\ 
\{ y_0^2+y_1^{12}y_2 +y_2^{11}+y_3^3 =0\} & \times & \{x_0^2+x_1^4+x_2^4=0 \} 
\end{array}$$ 
where we have $\ell =2, w_0=33$ and $v_0=2$. The threefold obtained 
by the twist map from $S_2\times E_2$ is 
$$y_1^{12}y_2 +y_2^{11}+y_3^3+x_1^4+x_2^4=0\ \subset \PP^4(10,12,44,33,33)$$ 
of degree $132$. As this threefold is quasi-smooth and $10+12+44+33+33=132$, 
it is a Calabi--Yau threefold. 

\noindent (ii) The 4th case $S_2\times E_3$ 
$$\begin{array}{lcl} 
\PP^3(33,5,6,22) & & \PP^2(3,1,2) \\ 
\{ y_0^2+y_1^{12}y_2 +y_2^{11}+y_3^3 =0\} & \times & \{x_0^2+x_1^6+x_2^3=0 \} 
\end{array}$$ 
where we have $\ell =2, w_0=33$ and $v_0=3$. The threefold obtained 
by the twist map from $S_2\times E_3$ is 
$$y_1^{12}y_2 +y_2^{11}+y_3^3+x_1^6+x_2^3=0\ \subset \PP^4(15,18,66,33,66)$$ 
of degree $198$. Further, the weight can be reduced to $\PP^4(5,6,22,11,22)$, 
so that the threefold has degree $198/3=66$ in $\PP^4(5,6,22,11,22)$, and
$5+6+22+11+22=66$.  Hence this is a Calabi--Yau threefold. 

Slightly changing the order of weights, we summarize the above observation 
as follows.

\begin{prop}\label{prop5-9}
Let $S_2: y_0^2+y_1^{12}y_2+y_2^{11}+y_3^3=0\subset \PP^3(33,5,6,22)$
be a quasi-diagonal $K3$ surface. Then the products $S_2\times E_2$
and $S_2\times E_3$ give rise to Calabi--Yau threefolds by the twist map. 
They are 
$$\begin{array}{ll} 
z_0^{12}z_1 +z_1^{11}+z_2^4+z_3^4+z_4^3=0 & \subset \PP^4(10,12,33,33,44) \\ 
z_0^{12}z_1 +z_1^{11}+z_2^6+z_3^3+z_4^3=0 & \subset \PP^4(5,6,11,22,22) 
\end{array}$$  
\end{prop}

\section{Singularities}
\label{sect6}

In this section, we describe the singularities of our Calabi--Yau 
threefolds. They are deformations of weighted diagonal hypersurfaces 
or quasi-diagonal hypersurfaces of such forms as 
$$z_0^{d_0}+z_1^{d_1}+z_2^{d_2}+z_3^{d_3}+z_4^{d_4}-\lambda 
z_0^{e_0}z_1^{e_1}-\phi z_2^{e_2}z_3^{e_3}z_4^{e_4}=0$$ 
or 
$$z_0^{d_0}z_1+z_1^{d_1}+z_2^{d_2}+z_3^{d_3}+z_4^{d_4}-\lambda 
z_0z_1z_2z_3z_4=0$$ 
in $\PP^4 (w_0,w_1,w_2,w_3,w_4)$, where $\lambda$ and $\phi$ are 
deformation parameters. We choose the cases where they become 
quasi-smooth and transversal so that all the singularities are coming 
from the ambient space $\PP^4(Q)$. This means that the deformations 
have little effects on the determination of singular loci and local 
descriptions of individual singularities. 

Let $X$ be a quasi-smooth transversal threefold in $\PP^4(Q)$. We have 
$$\begin{array}{l}
\mbox{Sing}\, (X) =X\cap \mbox{Sing}\, (\PP^4(Q)) \\ 
\ \ =\{ (z_0:\cdots :z_4)\in X \mid \gcd (w_i \mid z_i\neq 0)\geq 2 \}.
\end{array}$$
As $\PP^4(Q)$ has only cyclic quotient singularities, so does $X$. 
Since $Q$ is reduced, every quadruplet of $\{w_0,w_1,w_2,w_3,w_4\}$ 
are coprime, and so we start looking for triplets having $\gcd \geq 2$.    
Following is a procedure to determine the singular loci of $X$. 

\begin{enumerate} 
\item Find a triplet $(w_i,w_j,w_k)$ of weights with $\gcd (w_i,w_j,w_k) 
\geq 2$. Letting $z_i$, $z_j$ and $z_k$ be non-zero and other two 
coordinates be $0$, we obtain singularities that form a singular locus 
of dimension $1$ on $X$. 
\item Find a pair $(w_i,w_j)$ of weights with $\gcd (w_i,w_j)\geq 2$. 
Letting $z_i$ and $z_j$ be non-zero and other three coordinates be $0$, 
we obtain singular loci of dimension $0$ on $X$. (It is often the case 
that $0$-dimensional singularities are on the intersection of two 
one-dimensional singular loci.) 
\item Find a weight $w_i$ that is greater than $1$. Letting $z_i\neq 0$ 
and other four coordinates be $0$, we obtain an isolated singularity on 
$X$. (This type of isolated singularities exist only on quasi-diagonal 
hypersurfaces at $(1:0:0:0:0)$ and this point is a singularity if and 
only if $w_0\geq 2$.) 
\item Once we know the singular loci of $X$, we consider the group 
actions around them to determine their singularities. Since each of 
them is a cyclic quotient singularity, we can describe it locally by 
the quotient of some affine space by a cyclic group action. 
Specifically, we proceed as follows. 

For every singularity other than $(1:0:0:0:0)$, the zero coordinates 
around the point give a (part of) local coordinate system for the 
covering affine space; for the singularity at $(1:0:0:0:0)$, $(z_2,z_3,z_4)$ 
gives a local coordinate system. Then by changing the coordinates if 
necessary, we can write the cyclic group action as follows: around 
a $1$-dimensional singularity, it can be written as  
$$(x,y,z) \longmapsto (\zeta x,\zeta^a y,z)$$ 
and around a $0$-dimensional singularity including $(1:0:0:0:0)$, 
it appears as 
$$(x,y,z) \longmapsto (\zeta^{a_1}x, \zeta^{a_2}y, \zeta^{a_3}z )$$  
where $\zeta$ ranges over the cyclic group $\mu_m$ of some order $m$ with  
$\gcd (a,m)=\gcd (a_1,a_2,a_3)=1$. In what follows, we use $\Aff^3 /\mu_m$ to 
denote the singularity above defined locally by the group action $\mu_m$. 
\end{enumerate} 

\begin{lem} \label{action-for-cy} 
Let $X$ be a deformation of a diagonal or quasi-diagonal hypersurface 
in $\PP^4(Q)$. Let $(x,y,z) \longmapsto (\zeta x,\zeta^a y,z)$ and 
$(x,y,z) \longmapsto (\zeta^{a_1}x, \zeta^{a_2}y, \zeta^{a_3}z )$ be 
the group actions described above. Assume that $X$ is Calabi-Yau. Then 
$$1+a\equiv a_1+a_2+a_3\equiv 0\pmod{m}.$$    
\end{lem} 

\begin{proof} 
We choose the case $z_3z_4\neq 0$ with $\gcd (w_3,w_4)=m\geq 2$. (Other 
cases can be discussed similarly.) Then $(z_0,z_1,z_2)$ gives a local 
coordinate system on which the group action is 
$$(z_0,z_1,z_2)\longmapsto (\zeta^{w_0}x, \zeta^{w_1}y, \zeta^{w_2}z )$$ 
where $\zeta$ ranges over $\mu_m$. Since $X$ is Calabi-Yau, the weight 
satisfies $w_0+w_1+w_2+w_3+w_4=d$, where $d$ is the degree of $X$. 
Also, as $X$ is of diagonal or quasi-diagonal type, $w_3\mid d$. Hence 
$m\mid d$ and we have $w_0+w_1+w_2\equiv 0\pmod{m}$. 
\end{proof} 

\begin{rem} The majority of diagonal hypersurfaces do not have isolated 
singularities. Those having isolated singularities are, for instance, 
diagonal hypersurfaces in $\PP^4(1,3,4,4,12)$, $\PP^4(1,6,7,7,21)$ etc. 
In case of $\PP^4(1,3,4,4,12)$, the diagonal hypersurface is defined by 
$$z_0^{24}+z_1^{8}+z_2^{6}+z_3^{6}+z_4^{2}=0.$$ 
It has a 1-dimensional singular locus $z_0=z_1=0$ and two isolated 
singular points $(0:1:0:0:\pm \sqrt{-1})$. Note that the latter points 
are not on the 1-dimensional locus. 
\end{rem} 

\begin{lem} \label{action-diagonal} 
Let $X$ be a deformation of a diagonal hypersurface and 
$P$ be a $0$-dimensional singularity on $X$. Then for a suitable choice 
of $a$ and $b$, the group action above $P$ can be written as 
$$(x,y,z) \longmapsto (\zeta x, \zeta^{a}y, \zeta^{b}z )$$ 
with $\zeta \in \mu_m$. 
\end{lem} 

\begin{proof} 
For example, assume $\gcd (w_3,w_4)=m\geq 2$. Then $(z_0,z_1,z_2)$ can 
be chosen as a local coordinate system and $P$ is isomorphic to the 
singularity at the origin of the quotient of $\Aff^3$ by the action 
$$(z_0,z_1,z_2) \longmapsto (\zeta^{w_0}z_0,\zeta^{w_1}z_1,
\zeta^{w_2}z_2 )$$
where $\zeta \in \mu_m$. Since weight $Q=(w_0,w_1,w_2,w_3,w_4)$ is reduced, 
$\gcd (w_0,m)$, $\gcd (w_1,m)$ and $\gcd (w_2,m)$ are relatively prime 
in pairs. Hence $m$ is divisible by the product $\gcd (w_0,m)\gcd (w_1,m) 
\gcd (w_2,m)$. If $\gcd (w_i,m)\geq 2$ for $0\leq i\leq 2$, then $m$ has 
at least 3 distinct prime divisors. But a case-by-case analysis shows that 
this never happens for diagonal hypersurfaces (and hence for their 
deformations either). 
\end{proof} 

All Calabi-Yau threefolds we consider in this paper, such as those constructed 
in Proposition \ref{prop5-2} and Proposition \ref{prop5-3}, are quasi-smooth. 
Hence the above procedure can be applied to find their singular loci and 
local group actions. We illustrate this with three examples. Other cases 
can be treated similarly. 

\begin{ex} (Diagonal type)  \label{ex-diagonal} 
We consider a Calabi-Yau threefold in $\#9$ of Proposition \ref{prop5-3} 
defined by a diagonal hypersurface: 
$$X:\ z_0^{18}+z_1^9+z_2^6+z_3^3+z_4^3=0\ \subset \PP^4(1,2,3,6,6)$$ 
of degree $18$. It has 2 one-dimensional singular loci: 
$$\begin{array}{ll}
E: & (z_0=z_1=0),\ z_2^6+z_3^3+z_4^3=0\ \subset \PP^2(3,6,6) \\ 
E^{'}: & (z_0=z_2=0),\ z_1^9+z_3^3+z_4^3=0\ \subset \PP^2(2,6,6). 
\end{array}$$
By normalizing the weight as 
$$\PP^2(3,6,6)\cong \PP^2(1,2,2)\cong \PP^2 \mbox{ and } 
\PP^2(2,6,6)\cong \PP^2(1,3,3)\cong \PP^2$$ 
respectively, we see that these curves are isomorphic to 
$$E:\ z_2^3+z_3^3+z_4^3=0\quad \mbox{ and }\quad E^{'}:\ z_1^3+z_3^3+
z_4^3=0$$ 
in $\PP^2$. Hence they are isomorphic to the elliptic curve $E_1$ over 
$\QQ$ defined in Table \ref{table1}. $E$ and $E^{'}$ meet at 3 points 
$P_i=(0:0:0:1:-\omega^i)$, where $\omega$ is a primitive cube root of 
unity in $\CC$ and $i=1,2,3$. The group action around each singularity 
is described as follows. 

\noindent (i) Singularities on $E\setminus \{P_1,P_2,P_3\}$ are 
described locally as $\Aff^3 /\mu_3$, where the $\mu_3$-action is 
$$(x,y,z) \longmapsto (\zeta_3 x,\zeta_3^2 y,z).$$ 

\noindent (ii) Singularities on $E^{'}\setminus \{P_1,P_2,P_3\}$ are 
given by $\Aff^3 /\mu_2$, where the $\mu_2$-action is 
$$(x,y,z) \longmapsto (-x,-y,z).$$ 

\noindent (iii) The singularity at $P_i$ is given by $\Aff^3 /\mu_6$, 
where $\mu_6$ acts as 
$$(x,y,z) \longmapsto (\zeta_6 x, \zeta_6^2 y, \zeta_6^3 z ).$$ 
\end{ex} 

\begin{ex} (Quasi-diagonal type) \label{ex-qd1} 
We consider the second Calabi-Yau threefold of Proposition \ref{prop5-9}: 
$$z_0^{12}z_1 +z_1^{11}+z_2^6+z_3^3+z_4^3=0\ \subset \PP^4(5,6,11,22,22)$$ 
having degree $66$. There are 2 one-dimensional singular loci: 
$$\begin{array}{ll} 
E: & (z_0=z_1=0),\ z_2^6+z_3^3+z_4^3=0\ \subset \PP^2(11,22,22) \\ 
C: & (z_0=z_2=0),\ z_1^{11}+z_3^3+z_4^3=0\ \subset \PP^2(6,22,22). 
\end{array}$$ 
By normalizing the weight as 
$$\PP^2(11,22,22)\cong \PP^2(1,2,2)\cong \PP^2 \mbox{ and } 
\PP^2(6,22,22)\cong \PP^2(3,11,11)\cong \PP^2(3,1,1)$$ 
respectively, we see that these curves are isomorphic to
$$E:\ z_2^3+z_3^3+z_4^3=0\ \subset \PP^2\quad \mbox{ and }\quad 
C:\ z_1+z_2^3+z_4^3=0\ \subset \PP^2(3,1,1).$$ 
$E$ is an elliptic curve over $\QQ$ isomorphic to $E_1$ defined in 
Table \ref{table1} and $C$ is a rational curve over $\QQ$. $C$ and $E$ 
meet at 3 points $P_i=(0:0:0:1:-\omega^i)$, where $\omega$ is a primitive 
cube root of unity in $\CC$ and $i=1,2,3$. The group action around each 
singularity is described as follows. 

\noindent (i) Singularities on $E\setminus \{P_1,P_2,P_3\}$ are 
given by $\Aff^3 /\mu_{11}$, where the $\mu_{11}$-action is 
$$(x,y,z) \longmapsto (\zeta_{11}x,\zeta_{11}^{10}y,z).$$ 

\noindent (ii) Singularities on $C\setminus \{P_1,P_2,P_3\}$ are 
described locally as $\Aff^3 /\mu_2$, where the $\mu_2$-action is 
$$(x,y,z) \longmapsto (-x,-y,z).$$

\noindent (iii) The singularity at $P_i$ is given by $\Aff^3 /\mu_{22}$, 
where $\mu_{22}$ acts as 
$$(x,y,z) \longmapsto (\zeta_{22}x, \zeta_{22}^{10}y, \zeta_{22}^{11}z).$$ 

In addition to these singularities, we now have an isolated singularity, 
namely $P_4=(1:0:0:0:0)$. Since $z_0=1\neq 0$, neither singular locus 
above contains this point. 

\noindent (iv) Locally, $P_4$ is given by taking the quotient of $\Aff^3$ 
by a $\mu_5$-action 
$$(z_2,z_3,z_4)\longmapsto (\zeta_5 z_2,\zeta_5^2 z_3,\zeta_5^2 z_4)$$ 
where $\zeta$ ranges over $\mu_5$. This quotient has an isolated 
singularity at the origin. 
\end{ex} 

\begin{ex} (Quasi-diagonal type) \label{ex-qd2} 
We look at another quasi-diagonal hypersurface which has a genus-two 
curve as a singular locus. Consider a Calabi-Yau threefold defined by  
$$z_0^{12}z_1 +z_1^{11}+z_2^4+z_3^4+z_4^3=0\ \subset \PP^4(10,12,33,33,44)$$ 
of degree $132$. There are 3 one-dimensional singular loci: 
$$\begin{array}{ll}
C_1: & (z_0=z_1=0),\ z_2^4+z_3^4+z_4^3=0\ \subset \PP^2(33,33,44) \\ 
C_2: & (z_0=z_4=0),\ z_1^{11}+z_2^4+z_3^4=0\ \subset \PP^2(12,33,33) \\ 
C_3: & (z_2=z_3=0),\ z_0^{12}z_1 +z_1^{11}+z_4^3=0\ \subset \PP^2(10,12,44) 
\end{array}$$ 
These curves are isomorphic to 
$$\begin{array}{ll}
C_1: & z_2^4+z_3^4+z_4=0\ \subset \PP^2(1,1,4) \\ 
C_2: & z_1+z_2^4+z_3^4=0\ \subset \PP^2(4,1,1) \\ 
C_3: & z_0^{6}z_1 +z_1^{11}+z_4^3=0\ \subset \PP^2(5,3,11).  
\end{array}$$
$C_1$ and $C_2$ are rational curves and $C_3$ is a curve of genus 2. 
$C_1$ and $C_2$ meet at four points $P_i=(0:0:1:\zeta_8^{2i-1}:0)$, where 
$\zeta_8$ is a primitive 8th root of unity in $\CC$ and $i=1,2,3,4$. 

In addition to these singularities, there is a zero-dimensional singularity 
$P_5=(1:0:0:0:0)$. This point is on $C_3$. The group action around each 
singularity is described as follows. 

\noindent (i) Singularities on $C_1\setminus \{P_1,\cdots ,P_4\}$ are 
locally given by $\Aff^3 /\mu_{11}$, where the $\mu_{11}$-action is 
$$(x,y,z) \longmapsto (\zeta_{11}x,\zeta_{11}^{10}y,z).$$ 

\noindent (ii) Singularities on $C_2\setminus \{P_1,\cdots ,P_4\}$ are 
given by $\Aff^3 /\mu_{3}$, where the $\mu_{3}$-action is 
$$(x,y,z) \longmapsto (\zeta_{3}x,\zeta_{3}^{2}y,z).$$ 

\noindent (iii) For $1\leq i\leq 4$, the singularity at $P_i$ is given by 
$\Aff^3 /\mu_{33}$, where $\mu_{33}$ acts as 
$$(x,y,z) \longmapsto (\zeta_{33}x, \zeta_{33}^{21}y, \zeta_{33}^{11}z).$$ 

\noindent (iv) Singularities on $C_3\setminus \{P_5\}$ are described as  
$\Aff^3 /\mu_{2}$, where the $\mu_{2}$-action is 
$$(x,y,z) \longmapsto (-x,-y,z).$$ 

\noindent (v) Finally, $P_5$ is given by taking the quotient of $\Aff^3$ 
by a $\mu_{10}$-action 
$$(z_2,z_3,z_4)\longmapsto (\zeta_{10} z_2,\zeta_{10} z_3,\zeta_{10}^8 z_4)$$ 
where $\zeta$ ranges over $\mu_{10}$. This quotient has an isolated 
singularity at the origin. 
\end{ex} 

\begin{figure} 
\centering 
\includegraphics[width=5cm]{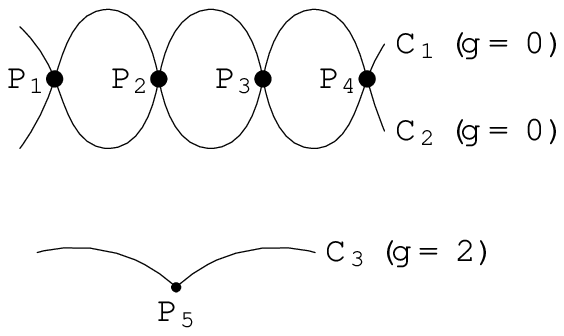} 
\caption{Singular loci of Example \ref{ex-qd2} } 
\label{sing-qd2} 
\end{figure}

\section{Resolution of singularities}  \label{sect-resol} 

We describe resolution of singularities for our Calabi-Yau threefolds. 
For simplicity, we use hypersurfaces of diagonal and quasi-diagonal types 
to illustrate concrete resolutions; the cases with deformations can be 
treated similarly. 

Our threefolds $X$ are quasi-smooth and transversal in $\PP^4(Q)$. This 
implies that their singularities are attributed to the ambient space and 
they can be resolved by applying toroidal desingularizations to $\PP^4(Q)$. 
We note, however, that we do not usually employ a full desingularization of 
$\PP^4(Q)$. Rather, we use a partial desingularization of it so that only 
the singular loci passing through $X$ can be resolved. First we discuss 
the field of definition for a resolution of $X$. 

\begin{prop} \label{field-of-defn} 
Let $X$ be a quasi-smooth hypersurface in $\PP^4(Q)$ over $k$. Then there 
exists a toroidal resolution $\xtil$ of $X$ defined over $k$. Furthermore 
if $X$ has a trivial dualizing sheaf, then there exists a crepant resolution 
$\xtil$ of $X$ such that $\xtil$ is a Calabi-Yau threefolds over $k$. 
\end{prop} 

\begin{proof} 
Since $X$ is quasi-smooth, all the singularities are due to the ambient 
space. Let $\tau :\widetilde{\PP^4}(Q)\lra \PP^4(Q)$ be a composite of 
several toric blow-ups, where $\widetilde{\PP^4}(Q)$ may be still 
singular. Denote by $\xtil_{\tau}$ the strict transform of $X$. Then 
$\xtil_{\tau}$ is non-singular if it has no intersection with singular 
loci of $\widetilde{\PP^4}(Q)$. To show the existence of a crepant 
resolution over $k$, it suffices to prove the existence of such 
$\widetilde{\PP^4}(Q)$ over $k$. We do this in three steps. 

Step 1: Every toric blow-up of $\PP^4(Q)$ can be defined over $k$. 

In fact, since $\PP^4(Q)$ is a toric variety, it is constructed by glueing 
together affine toric varieties $U_{\sigma}=\mbox{Spec}\, R_{\sigma}$, 
where $\sigma$ is a fan and $R_{\sigma}$ is a finitely generated polynomial 
algebra over $k$, namely $R_{\sigma} =k[x_1,x_2,x_3,x_4]^{\mu}$ for some 
group action $\mu$ and free variables $x_1,\cdots x_4$. A toric blow-up 
is a morphism corresponding to a subdivision of a fan and this can be 
realized over $k$ (i.e. new rings $\widetilde{R_{\sigma}}$ are also 
finitely generated over $k$). Hence toric blow-ups are defined over $k$ 
and so is a composite of them. 

Step 2: A toric resolution (which may be singular) of $\PP^4(Q)$ gives 
a smooth resolution of $X$ over $k$. 

For this, we express $\PP^4(Q)$ as a collection of $U_{\sigma}=\mbox{Spec}\, 
R_{\sigma}$ according as $x_i\neq 0$ for some $i$. Here every singular 
locus of $U_{\sigma}$ is defined over $k$. Then we apply toroidal 
desingularizations (i.e. toric blow-ups) only to the sheets $U_{\sigma}$ 
and along the singular loci that intersect $X$. In other words, we resolve 
only singularities of $\PP^4(Q)$ which give rise to the singularities 
of $X$. By Step 1, this resolution 
$\tau :\widetilde{\PP^4}(Q) \lra \PP^4(Q)$ is defined over $k$ and 
$\tau^{-1}(\mbox{Sing}(\PP^4(Q))\cap X)$ is no longer singular on 
$\widetilde{\PP^4}(Q)$. Hence the strict transform $\xtil$ of $X$ 
is smooth and defined over $k$. 

Step 3: There exists a crepant resolution of $X$ when $\omega_X \cong \OO_X$. 

Indeed, if $X$ has a trivial dualizing sheaf, then one knows that its 
crepant resolution can be obtained by applying toric blow-ups on $X$ 
locally. Our procedure above is a composite of toric blow-ups applied 
to the ambient space $\PP^4(Q)$ and by restricting them to $X$, we obtain 
the same toric blow-ups on $X$. Therefore $\xtil$ is a crepant resolution of 
$X$, which is Calabi-Yau and defined over $k$. 
\end{proof} 

\begin{cor} \label{field-defn-mod} 
Let $K$ be a number field and $\fp$ a prime of $K$. Let $X$ be a 
quasi-smooth hypersurface in $\PP^4(w_0,\cdots ,w_4)$ over $K$. Assume 
that $\fp \nmid w_i$ for $0\leq i\leq 4$ and the reduction $X_{\fp}:= 
X\mod \fp$ is quasi-smooth over $\OO_K/\fp$. Then the toric resolution 
$\tau :\xtil \lra X$ of {\rm Proposition \ref{field-of-defn}} commutes with 
the mod-$\fp$ reduction map, i.e. $\xtil \ \mbox{\rm mod }\fp \cong 
\widetilde{X_{\fp}}$. 
\end{cor} 

\begin{proof} 
Since no $w_i$ is divisible by $\fp$, $\PP^4(Q)_{\fp}:=\PP^4(Q) 
\mbox{ mod } \fp$ is a weighted projective $4$-space over $\FF :=\OO_K/\fp$ 
having at most cyclic quotient singularities. As every toric blow-up 
can be defined over $K$ or $\FF$, we may apply the same blow-up to 
$\PP^4(Q)$ and $\PP^4(Q)_{\fp}$, and find $\widetilde{\PP^4}(Q)_{\fp} 
\cong \widetilde{\PP^4}(Q) \mbox{ mod }\fp$. Hence the following 
diagram is commutative, where the vertical arrows are mod-$\fp$ 
reductions and horizontal arrows are toroidal desingularizations: 
$$\begin{matrix} 
\PP^4(Q) & \longleftarrow &  \widetilde{\PP^4}(Q) \\ 
\downarrow & \quad\quad & \downarrow \\ 
\PP^4(Q)_{\fp} & \longleftarrow & \widetilde{\PP^4}(Q)_{\fp} 
\end{matrix}$$
Now $X$ and $X_{\fp}$ are assumed to be quasi-smooth, so that their 
singularities are all coming from the ambient spaces. By restricting 
the above diagram on $X$ or $X_{\fp}$, we obtain a similar commutative 
diagram for them. Therefore $\xtil \ \mbox{\rm mod }\fp$ is isomorphic 
to $\widetilde{X_{\fp}}$. 
\end{proof} 

Our threefolds $X$ have only cyclic quotient singularities and their 
resolution can be described by toric geometry. In principle, we apply 
toroidal desingularizations to the ambient space. But, the necessary 
desingularizations are determined by local data of singularities 
of $X$. Following is a general procedure for doing this.   

\begin{enumerate} 
\item A $1$-dimensional singular locus arising in the quotient by the 
action  
$$(x,y,z) \longmapsto (\zeta x,\zeta^a y,z)$$ 
is similar to a cyclic quotient singularity of a surface. It can be 
resolved by successive blow-ups.  
\item A $0$-dimensional singularity described locally by the action 
$(x,y,z) \longmapsto (\zeta^{a_1}x, \zeta^{a_2}y, \zeta^{a_3}z )$ 
can be resolved by subdividing the cone generated by four vectors.  
A typical case is that $a_1=1$ where the action is 
$$(x,y,z) \longmapsto (\zeta x, \zeta^a y, \zeta^b z )$$ 
with $\zeta \in \mu_m$. Since $X$ is Calabi-Yau, we find $1+a+b
\equiv 0\pmod{m}$. In this case, we draw a triangle in $\RR^3$ 
with vertices 
$$A(0,1,0),\ B(0,0,1) \mbox{ and } C(m,-a,-b).$$ 
Then the subdivision of the cone is equivalent to finding lattice 
points in or on the triangle $ABC$. 
\item The lattice points on the side $AC$ or $BC$ (aside from the 
vertices) correspond to the exceptional divisors arising from a 
$1$-dimensional singular locus. Each point represents a ruled surface 
$C\times \PP^1$. 
\item The lattice points on the interior of $ABC$ correspond to the 
exceptional divisors arising from a $0$-dimensional singularity. 
Each point represents a projective plane $\PP^2$. 
\end{enumerate} 

\begin{figure} 
\centering 
\includegraphics[width=5cm]{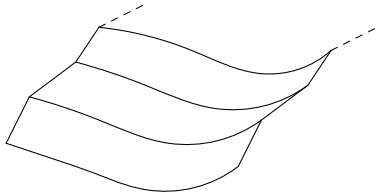} 
\caption{Exceptional divisors from a curve singularity} 
\label{ex-div-dim1} 
\end{figure} 

\begin{rem} 
In many cases, we find $0$-dimensional singularities in the intersection 
of two $1$-dimensional singular loci or as singularities of a 
$1$-dimensional singular locus. When we resolve such $0$-dimensional 
singularities, we can simultaneously obtain the resolution picture 
for the $1$-dimensional singular loci containing them (see examples 
below). 
\end{rem} 

To compute the zeta-function of $X$ or its resolution, 
it is important to find the field of definition for each exceptional 
divisor. The following lemma describes it for our choice of $X$. 

\begin{lem} \label{field-defn} 
Let $X$ be a weighted diagonal hypersurface or a weighted quasi-diagonal 
hypersurface over $\QQ$. Let $\xtil \lra X$ be a crepant resolution of 
$X$. Then $\xtil$ is defined over $\QQ$ and the following assertions hold. 

$(a)$ If $C$ denotes a $1$-dimensional singular locus of $X$, then 
the chain of ruled surfaces arising in the resolution of $C$ can be 
defined over $\QQ$. 

$(b)$ If $P$ is a $0$-dimensional singularity of $X$, then the 
projective planes arising in the resolution of $P$ can be defined 
over a cyclotomic field $\QQ (\zeta_{2m} )$, where $\zeta_{2m}$ is 
a primitive $2m$-th root of unity determined as follows: if 
$P=(1:0:0:0:0)$, then $m=1$ $($i.e. $\QQ (\zeta_{2m} )=\QQ)$. 
If $P\neq (1:0:0:0:0)$, then let $d$ be the degree of the defining 
equation of $X$, and $w_i$ and $w_j$ be the weights corresponding 
to the non-zero coordinates of $P$. Then  
$$m= 
\begin{cases} 
\frac{d-w_1}{\lcm (w_0,w_1)} & \mbox{ if $(i,j)=(0,1)$ and $X$ is 
quasi-diagonal} \\ 
\frac{d}{\lcm (w_i,w_j)} & \mbox{ otherwise. }
\end{cases}$$  

$(c)$ Let $P$ be a singularity defined in $(b)$. The sum of the 
exceptional divisors over the $\mbox{Gal}\, 
(\QQ (\zeta_{2m} )/\QQ)$-conjugates of $P$ is defined over $\QQ$. 
\end{lem} 

\begin{proof} It follows from Proposition \ref{field-of-defn} that 
$\xtil$ is defined over $\QQ$. 

(a) Every locus $C$ is a weighted projective curve on $X$ obtained by  
letting two coordinates zero. Hence it is defined over $\QQ$ and so is 
the ruled surface $C\times \PP^1$. 

(b) Clearly, $P=(1:0:0:0:0)$ is defined over $\QQ$. Consider the case 
where $X$ is quasi-diagonal and $(i,j)=(0,1)$. The non-zero coordinates 
of $P$ satisfy the equation 
$$x_0^{d_0}+x_1^{d_1-1}=0 \quad \subset \PP^1(w_0,w_1).$$
Reducing the weight, we see that it is equivalent to 
$$x_0^{(d-w_1)/\lcm (w_0,w_1)}+x_1^{(d-w_1)/\lcm (w_0,w_1)}=0 \quad 
\subset \PP^1.$$ 
Hence $P$ is defined over $\QQ (\zeta_{2m} )$. 

Other cases can be discussed similarly, where the non-zero coordinates 
of $P$ satisfy the equation $x_i^{d_i}+x_j^{d_j}=0\subset \PP^1(w_i,w_j)$ 
and we can reduce it to 
$$x_i^{d/\lcm (w_i,w_j)}+x_j^{d/\lcm (w_i,w_j)}=0 \quad \subset \PP^1.$$
Hence $P$ is defined over $\QQ (\zeta_{2m} )$. 

(c) The non-zero coordinates of each $P$ satisfy the equation 
$x_i^{d/\lcm (w_i,w_j)}+x_j^{d/\lcm (w_i,w_j)}=0$ which is defined 
over $Q$. To resolve singularities of $X$, we apply toric blow-ups 
to the ambient space. Hence there exists a simultaneous resolution 
for all the Galois conjugates of $P$ and their exceptional divisors 
are also Galois conjugates. Therefore their sum is defined over $\QQ$. 
\end{proof} 

\begin{rem}
(1) When $m$ is odd, $\QQ (\zeta_{2m} )=\QQ (\zeta_m )$. 

(2) A ruled surface $C\times \PP^1$ or a projective plane $\PP^2$ in 
Lemma \ref{field-defn} can indeed be defined over $\ZZ$ or $\ZZ [\zeta_{2m} ]$. 
When we take modulo reduction of $X$ by a prime $p$, the exceptional 
divisors are defined over $\FF_p$ or $\FF_p (\zeta_{2m} )$, where 
$\zeta_{2m}$ is a $2m$-th root of unity in $\overline{\FF_p}$. Note that 
$\zeta_{2m} \in \FF_p$ if and only if $p\equiv 1\pmod{2m}$.  
\end{rem} 

\begin{lem} \label{resol-dim1} 
Let $C$ be a $1$-dimensional singular locus associated with the 
group action 
$$(x,y,z) \longmapsto (\zeta x,\zeta^{m-1} y,z)$$ 
where $\zeta$ ranges over $\mu_m$. Then the minimal resolution at 
$C$ is a chain of $m-1$ sheets of a ruled surface $C\times \PP^1$ 
connected to one another at $C$ $($see {\em Figure \ref{ex-div-dim1}}$)$. 
\end{lem} 

\begin{proof} 
This singularity may be considered as a cyclic quotient singularity 
of type $A_{m,m-1}$ on a surface over the function field $k(C)$. 
Hence we can resolve it by successive blow-ups and obtain a chain 
of $m-1$ sheets of $\PP^1$ over $k(C)$. 
\end{proof} 

In order to compute the zeta-functions of resolutions of $X$ in the 
next section, we discuss the number of rational points on the 
exceptional divisors over finite fields. 

\begin{cor} \label{zeta-resol1} Let $X$ be a threefold over a finite 
field $\FF_q$ of $q$ elements. Let $C$ be a $1$-dimensional singular 
locus of $X$ defined in {\em Lemma \ref{resol-dim1}}. Write $N(C)$ for 
the number of $\FF_q$-rational points on $C$. Then a crepant resolution 
of $X$ acquires $q(m-1)N(C)$ number of $\FF_q$-rational points on the 
exceptional divisor on $C$, where the points on the proper transform 
of $C$ are not counted.
\end{cor} 

\begin{proof} 
Since the exceptional ruled surfaces meet at $C$, every $C\times \PP^1$ 
contains $qN(C)$ new points. As $m-1$ sheets meet transversely, the 
total number of new points is $q(m-1)N(C)$. 
\end{proof} 

\begin{lem} \label{resol-dim0} 
Let $P$ be a $0$-dimensional singularity associated with the 
group action 
$$(x,y,z) \longmapsto (\zeta x, \zeta^a y, \zeta^b z )$$ 
where $\zeta \in \mu_m$ and $1+a+b=m$. Let $e$ be the number of 
lattice points in the interior of the triangle in $\RR^3$ with 
vertices $A(0,1,0)$, $B(0,0,1)$ and $C(m,-a,-b)$. Then the strict 
transform of $P$ on a crepant resolution of $X$ consists of $e$ 
copies of $\PP^2$. If two $\PP^2$s are not disjoint, then they 
intersect at an exceptional line $\PP^1$.  
\end{lem} 

\begin{proof} 
This is known from toric geometry. Details may be found in \cite{Hu}. 
\end{proof} 

\begin{cor} \label{zeta-resol0} Let $X$ be a threefold over a finite 
field $\FF_q$ of $q$ elements. Let $P$ be a $0$-dimensional singularity  
of $X$ defined in {\em Lemma \ref{resol-dim0}}. Assume that $P$ is defined 
over $\FF_q$. Then with the assumptions of {\em Lemma \ref{resol-dim0}}, 
we have $e(q+q^2)$ number of $\FF_q$-rational points on the exceptional 
divisor on $P$, where the proper transform of $P$ is not counted.. 
\end{cor} 

\begin{proof} 
As $P$ is defined over $\FF_q$, every projective plane $\PP^2$ in the 
resolution of $P$ is defined over $\FF_q$ and it contains $1+q+q^2$ 
points. Among these points, one belongs to the threefold $X$ and $q+q^2$ 
points are new. Also, 
when two sheets of $\PP^2$ have an intersection, it is an exceptional 
line $\PP^1$ that sprouts out of a point. Hence every $\PP^2$ gives 
rise to $q+q^2$ new points.
\end{proof} 

\begin{lem} \label{formula-for-e} 
The assumptions and hypothesis of {\em Lemma \ref{resol-dim0}} remain in 
force. Write $a_1=\gcd (m,a)$ and $b_1=\gcd (m,b)$. Then there exist 
$1+a_1+b_1$ points on the sides of triangle $ABC$ and the number 
of exceptional $\PP^2$ from the singularity $P$ is equal to  
$$e=\frac{m+1-a_1-b_1}{2}.$$ 
\end{lem} 

\begin{proof} 
It follows from Lemma 1.1 of \cite{Hu} (an original form of which 
can be found in \cite{RY}) that the number of all lattice points on 
$\triangle ABC$ is equal to 
$$\frac{m+2+(1+a_1+b_1)}{2}.$$ 
Among these, the points on the side $BC$ correspond to the exceptional 
ruled surfaces arising from the $\mu_{a_1}$ action described in Lemma 
\ref{resol-dim1} with $m=a_1$. Hence $BC$ has $a_1-1$ interior points. 
Similarly, there are $b_1-1$ interior points on the side $AC$. Together 
with 3 vertices, there are $1+a_1+b_1$ lattice points on the sides of 
$\triangle ABC$. Hence the number of interior points is calculated 
as asserted.   
\end{proof} 

All Calabi-Yau threefolds we discuss in this paper have at most cyclic 
quotient singularities and the above resolution procedure can be applied 
to them. As it is rather lengthy to list all the cases, we illustrate 
the procedure with three typical examples. Other singularities can be 
resolved similarly. 

\begin{ex} (Diagonal type) \label{resol-diag}  
Consider the diagonal hypersurface discussed in Example \ref{ex-diagonal}: 
$$X:\ z_0^{18}+z_1^9+z_2^6+z_3^3+z_4^3=0\ \subset \PP^4(1,2,3,6,6)$$ 
of degree $18$. It has 2 one-dimensional singular loci meeting at 
3 points $P_i$ $(i=1,2,3)$, namely    
$$\begin{array}{lll}
E: & \ z_2^3+z_3^3+z_4^3=0\ \subset \PP^2 & \mbox{ (fixed by $\mu_3$)} \\ 
E^{'}: & \ z_1^3+z_3^3+z_4^3=0\ \subset \PP^2 & \mbox{ (fixed by $\mu_2$)} \\ 
P_i & =(0:0:0:1:-\omega^i) & 
\end{array}$$
where $E\cong E^{'}\cong E_1$ are elliptic curves over $\QQ$ and $\omega$ 
is a primitive cube root of unity. As every $P_i$ is on both singular loci, 
when we resolve singularity at $P_i$, it will also describe the resolution of 
$E$ and $E^{'}$. 

Recall that the singularity at $P_i$ is locally isomorphic to 
$\Aff^3 /\mu_6$, where $\mu_6$ acts as 
$$(x,y,z) \longmapsto (\zeta_6 x, \zeta_6^2 y, \zeta_6^3 z ).$$ 
We consider a triangle with vertices 
$$A(0,1,0),\ B(0,0,1),\ C(6,-2,-3).$$
By Lemma \ref{formula-for-e}, there are $6$ lattice points on the 
sides of $\triangle ABC$, of which $3\ (=6-3)$ points correspond to 
exceptional divisors. The lattice points on $AC$ (aside from the 
edges) represent a chain of 2 ruled surfaces $E\times \PP^1$ arising 
from a singular locus $E$; a point on $BC$ represents a ruled surface 
$E^{'}\times \PP^1$ from $E^{'}$. On the other hand, there is one lattice 
point in the interior of $\triangle ABC$. This corresponds to a plane $\PP^2$. 

Next we locate the lattice points in and on $\triangle ABC$ exactly 
and then draw lines connecting them in such a way that no two line 
segments cross each other. As in \cite{Hu}, Figure \ref{triangle0} 
shows an example of such subdivision.  

\begin{figure} 
\centering 
\includegraphics[width=5cm]{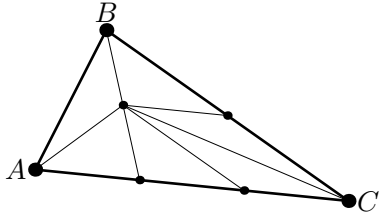} 
\caption{Subdivision of $ABC$ for Example \ref{resol-diag}} 
\label{triangle0} 
\end{figure}

We summarize the exceptional divisors as follows: 
\medskip 

\begin{tabular}{|l|l|l|} \hline 
 & Lattice points & Corresponding exceptional divisors \\ \hline 
On $AC$ & $2$ points &  $2$ copies of $E\times \PP^1$ over $\QQ$ \\ 
On $BC$ & $1$ point &  one $E^{'}\times \PP^1$ over $\QQ$ \\ 
Interior & $1$ point &  one $\PP^2$ over $\QQ(\omega )$ \\ \hline 
\end{tabular} 
\medskip 
 
Note that the exceptional $\PP^2$ appears at each $P_i$. Since there are 
three points in the intersection of $E$ and $E^{'}$, we obtain totally 
$3$ copies of $\PP^2$ each of which is defined over $\QQ(\omega )$. 
These planes intersect ruled surfaces at their exceptional lines. 
The whole resolution picture is given in Figure \ref{e-div-diagonal}. 
\end{ex} 

\begin{figure} 
\centering 
\includegraphics[width=5cm]{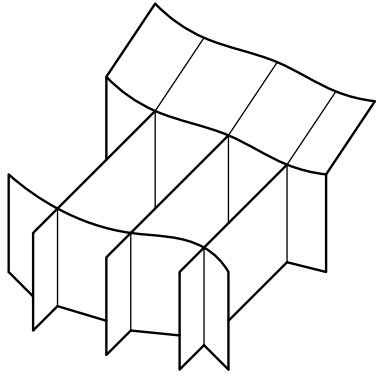} 
\caption{Exceptional divisors for Example \ref{resol-diag}} 
\label{e-div-diagonal} 
\end{figure} 

\begin{ex} (Quasi-diagonal type) \label{reso-qd1} 
We consider the quasi-diagonal hypersurface discussed in Example \ref{ex-qd1}: 
$$z_0^{12}z_1 +z_1^{11}+z_2^6+z_3^3+z_4^3=0\ \subset \PP^4(5,6,11,22,22)$$ 
having degree $66$. There are 2 one-dimensional singular loci: 
$$\begin{array}{lll} 
C: & (z_0=z_2=0),\ z_1^{11}+z_3^3+z_4^3=0\ \subset \PP^2(6,22,22) & 
\mbox{ (fixed by $\mu_2$)} \\ 
E: & (z_0=z_1=0),\ z_2^6+z_3^3+z_4^3=0\ \subset \PP^2(11,22,22) & 
\mbox{ (fixed by $\mu_{11}$)} 
\end{array}$$ 
These curves are isomorphic to 
$$C:\ z_1+z_2^3+z_4^3=0\ \subset \PP^2(3,1,1)\quad \mbox{ and }\quad 
E:\ z_2^3+z_3^3+z_4^3=0\ \subset \PP^2.$$ 
$E$ is an elliptic curve over $\QQ$ isomorphic to $E_1$. $C$ is a rational 
curve over $\QQ$. They meet at 3 points $P_i=(0:0:0:1:-\omega^i)$, where 
$\omega$ is a primitive cube root of unity in $\CC$ and $i=1,2,3$. When we 
resolve singularity at $P_i$, it also gives us the resolution of $C$ and $E$. 

Recall that the singularity at $P_i$ is locally isomorphic to 
$\Aff^3 /\mu_{22}$, where $\mu_{22}$ acts as 
$$(x,y,z) \longmapsto (\zeta_{22}x, \zeta_{22}^{10}y, \zeta_{22}^{11}z).$$ 
We consider a triangle with vertices 
$$A(0,1,0),\ B(0,0,1),\ C(22,-10,-11).$$
By Lemma \ref{formula-for-e}, there are $14$ lattice points on the sides 
of $\triangle ABC$, of which $11$ points correspond to exceptional 
divisors. The lattice points on $AC$ (aside from the edges) represent 
a chain of 10 ruled surfaces $E\times \PP^1$ arising from a singular 
locus $E$; a point on $BC$ represents a ruled surface $C\times \PP^1$ 
from $C$. On the other hand, there are 5 lattice points in the interior 
of $\triangle ABC$. They correspond to 5 copies of $\PP^2$. 

Once we determine the lattice points on $ABC$, we draw lines connecting 
them in such a way that no two line segments cross each other. An example 
of a subdivision of $ABC$ can be found in Figure \ref{triangle1}. 

\begin{figure} 
\centering 
\includegraphics[width=5cm]{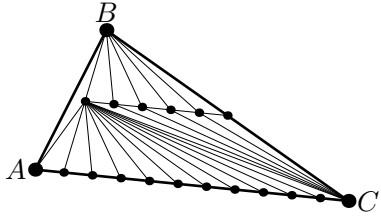} 
\caption{Subdivision of $ABC$ for Example \ref{reso-qd1} (1-dimensional 
loci)} 
\label{triangle1} 
\end{figure}

The exceptional divisors corresponding to $16$ lattice points are 
summarized as follows (note that $5$ copies of $\PP^2$ appear from  
every $P_i$ for $i=1,2,3$): 
\medskip 

\begin{tabular}{|l|l|l|} \hline 
 & Lattice points & Corresponding exceptional divisors \\ \hline 
On $AC$ & $10$ points &  $10$ copies of $E\times \PP^1$ over $\QQ$ \\ 
On $BC$ & $1$ point &  one $C\times \PP^1$ over $\QQ$ \\ 
Interior & $5$ points &  $5$ copies of $\PP^2$ over $\QQ(\omega )$ \\ \hline 
\end{tabular} 
\medskip 
 
In addition to the above singularities, this quasi-diagonal threefold has 
an isolated singularity $P_4=(1:0:0:0:0)$. Since $z_0=1\neq 0$, neither 
singular locus above contains this point. Locally, $P_4$ is given by the 
quotient $\Aff^3/\mu_5$, where the $\mu_5$-action is 
$$(z_2,z_3,z_4)\longmapsto (\zeta_5 z_2,\zeta_5^2 z_3,\zeta_5^2 z_4).$$ 
We consider a triangle with vertices  
$$A(0,1,0),\ B(0,0,1),\ C(5,-2,-2).$$
Apart from these vertices, it has no lattice points on the sides and 
has $2$ lattice points in the interior. Hence the resolution of $P_4$ 
consists of $2$ copies of $\PP^2$ over $\QQ$. The subdivision of $ABC$ 
is given in Figure \ref{triangle2}.

\begin{figure}  
\centering 
\includegraphics[width=5cm]{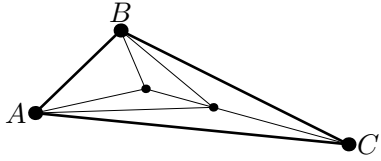} 
\caption{Subdivision of $ABC$ for Example \ref{reso-qd1} (0-dimensional 
loci)} 
\label{triangle2} 
\end{figure}

The intersections between these planes are exceptional lines. 
The whole resolution picture is given in Figure \ref{e-div-q-diago}. 

\begin{figure} 
\centering 
\includegraphics[width=5cm]{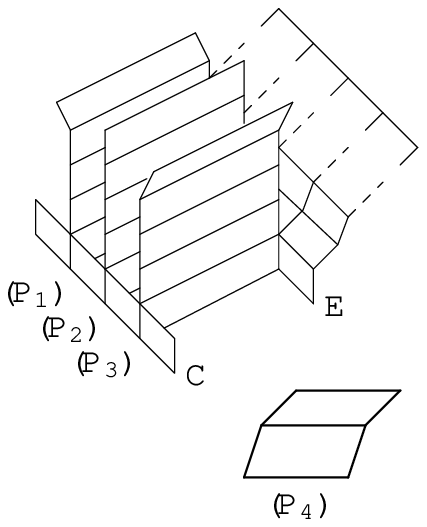} 
\caption{Exceptional divisors for Example \ref{reso-qd1} } 
\label{e-div-q-diago} 
\end{figure} 
\end{ex} 

\begin{ex} (Quasi-diagonal type) \label{reso-qd2} 
We consider another quasi-diagonal hypersurface discussed in Example 
\ref{ex-qd2}: 
$$z_0^{12}z_1 +z_1^{11}+z_2^4+z_3^4+z_4^3=0\ \subset \PP^4(10,12,33,33,44)$$ 
of degree $132$. There are 3 one-dimensional singular loci: 
$$\begin{array}{lll}
C_1: & (z_0=z_1=0),\ z_2^4+z_3^4+z_4=0\ \subset \PP^2(1,1,4) & 
\mbox{ (fixed by $\mu_{11}$)} \\ 
C_2: & (z_0=z_4=0),\ z_1+z_2^4+z_3^4=0\ \subset \PP^2(4,1,1) & 
\mbox{ (fixed by $\mu_3$)} \\ 
C_3: & (z_2=z_3=0),\ z_0^{6}z_1 +z_1^{11}+z_4^3=0\ \subset \PP^2(5,3,11) & 
\mbox{ (fixed by $\mu_2$)} 
\end{array}$$ 
$C_1$ and $C_2$ are rational curves and they meet at four points $P_i= 
(0:0:1:\zeta_8^{2i-1}:0)$, where $\zeta_8$ is a primitive 8th root of unity 
in $\CC$ and $i=1,2,3,4$. $C_3$ is a curve of genus 2, on which there is 
a different type of singularity at $P_5=(1:0:0:0:0)$. (Note that the singular 
locus $C_3$ itself has a singularity at $P_5$.)   

We can find the whole resolution picture by resolving singularities at $P_i$. 

(i) Recall that the singularity at $P_i$ ($i=1,2,3,4$) is locally 
isomorphic to $\Aff^3 /\mu_{33}$, where $\mu_{33}$ acts as 
$$(x,y,z) \longmapsto (\zeta_{33}x, \zeta_{33}^{21}y, \zeta_{33}^{11}z).$$
We consider a triangle with vertices 
$$A(0,1,0),\ B(0,0,1),\ C(33,-21,-11).$$
By Lemma \ref{formula-for-e}, there are $15$ lattice points on the 
sides of $\triangle ABC$, of which $12$ points correspond to 
exceptional divisors. The lattice points on $AC$ (aside from the 
edges) represent a chain of 10 ruled surfaces $C_1\times \PP^1$ arising 
from a singular locus $C_1$; 2 points on $BC$ represent a chain of 2 
ruled surfaces $C_2\times \PP^1$ from $C_2$. On the other hand, there 
are 10 lattice points in the interior of $\triangle ABC$. They 
correspond to 10 copies of $\PP^2$. 

Once we determine the lattice points on $ABC$, we draw lines connecting 
them in such a way that no two line segments cross each other. An example 
of a subdivision of $ABC$ can be found in Figure \ref{triangle3}. 

\begin{figure} 
\centering 
\includegraphics[width=5cm]{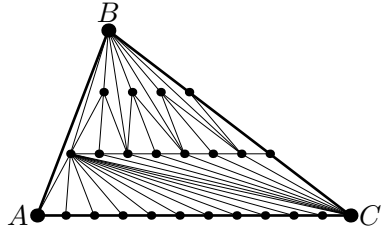} 
\caption{Subdivision of $ABC$ for Example \ref{reso-qd2} (for $C_1$ and $C_2$)} 
\label{triangle3} 
\end{figure}

The exceptional divisors corresponding to $22$ lattice points are 
summarized as follows (note that $10$ copies of $\PP^2$ arise at   
every $P_i$ for $i=1,2,3,4$): 
\medskip 

\begin{tabular}{|l|l|l|} \hline 
 & Lattice points & Corresponding exceptional divisors \\ \hline 
On $AC$ & $10$ points &  $10$ copies of $C_1\times \PP^1$ over $\QQ$ \\ 
On $BC$ & $2$ points &  $2$ copies of $C_2\times \PP^1$ over $\QQ$ \\ 
Interior & $10$ points &  $10$ copies of $\PP^2$ over $\QQ(\zeta_8 )$ \\ \hline 
\end{tabular} 
\medskip 

(ii) Next, we resolve singularity at $P_5$. It is locally isomorphic 
to $\Aff^3$, where $\mu_{10}$-action is 
$$(z_2,z_3,z_4)\longmapsto (\zeta z_2,\zeta z_3,\zeta^8 z_4).$$ 
We consider a triangle with vertices  
$$A(0,1,0),\ B(0,0,1),\ C(10,-1,-8).$$
By Lemma \ref{formula-for-e}, there are $4$ lattice points on the 
sides of $\triangle ABC$, of which $1\ (=4-3)$ point corresponds to an 
exceptional divisor. The lattice point on $AC$ (aside from the edges) 
represents a ruled surface $C_3\times \PP^1$ arising from a singular 
locus $C_3$; there is no lattice point on $BC$. On the other hand, 
there are 4 points in the interior of $\triangle ABC$. They correspond 
to 4 copies of $\PP^2$. A subdivision of $ABC$ is given in Figure 
\ref{triangle4}.

\begin{figure}  
\centering 
\includegraphics[width=5cm]{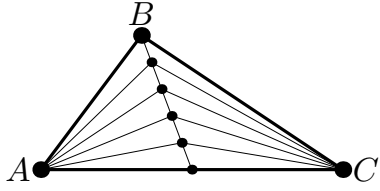} 
\caption{Subdivision of $ABC$ for Example \ref{reso-qd2} (for $C_3$)} 
\label{triangle4} 
\end{figure}

The exceptional divisors corresponding to $5$ lattice points are 
summarized as follows: 
\medskip

\begin{tabular}{|l|l|l|} \hline 
 & Lattice points & Corresponding exceptional divisors \\ \hline 
On $AC$ & $1$ point &  one $C_3\times \PP^1$ over $\QQ$ \\ 
On $BC$ & none &  \\ 
Interior & $4$ points &  $4$ copies of $\PP^2$ over $\QQ$ \\ \hline 
\end{tabular} 
\medskip 

The intersections between these planes are exceptional lines. 
The whole resolution picture is given in Figure \ref{ediv-qd-2}. 

\begin{figure} 
\centering 
\includegraphics[width=8cm]{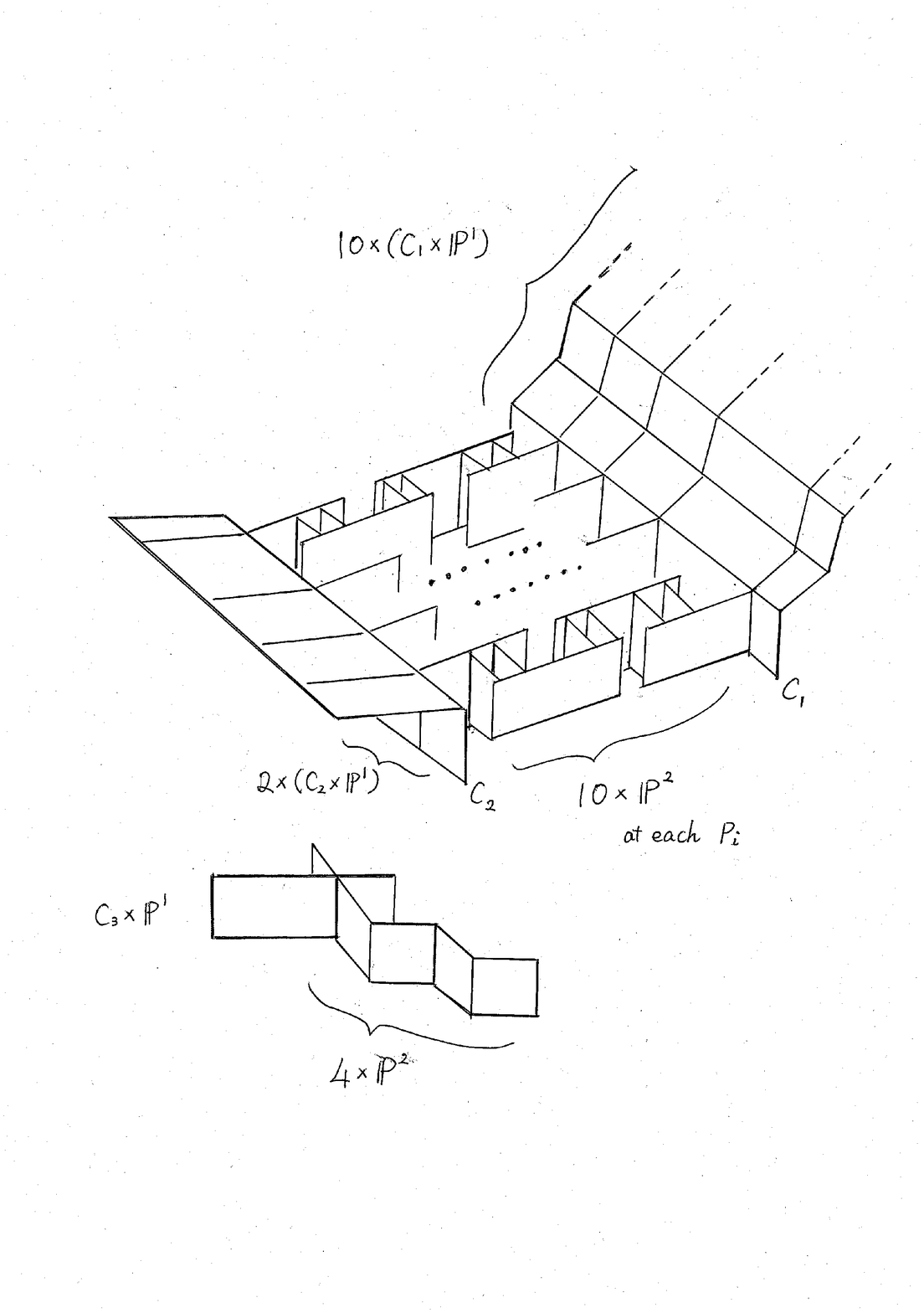} 
\caption{Exceptional divisors for Example \ref{reso-qd2}} 
\label{ediv-qd-2} 
\end{figure} 
\end{ex} 

\section{Cohomology of product and quotient varieties}
\label{sect-cohomology}

In this section, we consider threefolds constructed by taking finite 
quotients of the product of a diagonal curve and a diagonal surface.  
We describe their cohomology over $\ok$. 

Let $Q_1=(w_0,w_1,w_2,w_3)$ and $Q_2=(v_0,v_1,v_2)$. Let $S$ be a weighted 
diagonal surface of degree $d$ in $\PP^3(Q_1)$ and $C$ be a weighted 
diagonal curve of degree $e$ in $\PP^2(Q_2)$. (As we work over $\ok$, $S$ 
and $C$ may be chosen to be of Fermat type.) Write 
$$\begin{array}{l} 
\Gamma_S =\mu_{d/w_0} \times \cdots \times \mu_{d/w_3} /(\mbox{diagonal 
elements}) \\ 
\Gamma_C =\mu_{e/v_0} \times \mu_{e/v_1} \times \mu_{e/v_2}/ 
(\mbox{diagonal elements}).
\end{array}$$ 
In Section \ref{sect2}, we find the following decomposition:   
\begin{equation} \label{h2-s} 
\begin{array}{ll} 
H^2(S,\QQ_{\ell}) & \cong V(0) \oplus \bigoplus_{\ba \in \fA_S} V(\ba ) \\ 
H^1(C,\QQ_{\ell}) & \cong \bigoplus_{\bb \in \fA_C} V(\bb ) 
\end{array} 
\end{equation} 
where 
$$\begin{array}{l} 
\fA_S:=\biggl\{ \ba =(a_0,a_1,a_2,a_3)\,|\, a_i\in (w_i\ZZ/d\ZZ), 
a_i\neq 0, \sum_{i=0}^{3} a_i\equiv 0\pmod d \biggr\} \\ 
\fA_C:=\biggl\{ \bb =(b_0,b_1,b_2)\,|\, b_i\in (v_i\ZZ/e\ZZ), 
b_i\neq 0, \sum_{i=0}^{2} b_i\equiv 0\pmod e \biggr\}. 
\end{array}$$  
and 
$$\begin{array}{l} 
V(\ba )=\{ v\in H^2_{prim}(S,\QQ_{\ell}) \mid  
\gamma^{*}(v)={\zeta}_0^{a_1}{\zeta}_1^{a_1}{\zeta}_2^{a_2} 
{\zeta}_3^{a_3} v,\ \forall \gamma 
=({\zeta}_0,{\zeta}_1,{\zeta}_2,{\zeta}_3)\in \Gamma_S \} \\ 
V(\bb )=\{ v\in H^1(C,\QQ_{\ell}) \mid  
\gamma^{*}(v)={\xi}_0^{b_0}{\xi}_1^{b_0}{\xi}_2^{b_2} v,\ 
\forall \xi =({\xi}_0,{\xi}_1,{\xi}_2)\in \Gamma_C \}.  
\end{array}$$ 

\begin{lem} \label{lem2-1} 
Let $S$ be a weighted diagonal surface and $C$ be a weighted 
diagonal curve. Write $Y=S\times C$. Then the cohomology of $Y$ is 
described as follows: 
$$\begin{array}{ll} 
H^0(Y,\QQ_{\ell})= & H^0(S)\otimes H^0(C) \cong \QQ_{\ell} \\ 
H^1(Y,\QQ_{\ell})= & H^0(S)\otimes H^1(C) \cong H^1(C) \\ 
H^2(Y,\QQ_{\ell})= & H^2(S)\otimes H^0(C) \oplus 
H^0(S)\otimes H^2(C) \cong H^2(S) \oplus \QQ_{\ell}(-1) \\
H^3(Y,\QQ_{\ell})= & H^2(S)\otimes H^1(C) \\ 
H^4(Y,\QQ_{\ell})= & H^2(S)(-1) \oplus 
\QQ_{\ell}(-2) \\
H^5(Y,\QQ_{\ell})= & H^4(S)\otimes H^1(C) \cong 
 H^1(C) (-2)\\
H^6(Y,\QQ_{\ell})= & H^4(S)\otimes H^2(C) \cong \QQ_{\ell}(-3) 
\end{array}$$ 
\end{lem} 

\begin{proof} 
This follows from the K\"unneth formula (\ref{coh2}) in Section 2 and 
the fact that $H^1(S,\QQ_{\ell})=H^3(S,\QQ_{\ell})=0$ for a weighted 
diagonal surface $S$. Also note that $H^2(C)=\QQ_{\ell}(-1)$ and $H^4(S)
=\QQ_{\ell}(-2)$. 
\end{proof} 

For $Y=S\times C$, let $\Gamma_Y$ be a subgroup of $\Gamma_S \times \Gamma_C$ 
acting on $Y$ and consisting of elements of the form 
$$\bgam =(({\zeta}_0,{\zeta}_1,{\zeta}_2,{\zeta}_3), 
({\xi}_0,{\xi}_1,{\xi}_2)).$$ 
Write 
$$\begin{array}{l} 
\fA_S^{\Gamma}:=\{ \ba =(a_0,a_1,a_2,a_3)\in \fA_S \,|\, 
{\zeta}_0^{a_0}{\zeta}_1^{a_1}{\zeta}_2^{a_2}{\zeta}_3^{a_3}=1,\ 
\forall \bgam \in \Gamma_Y \} \\ 
\fA_C^{\Gamma}:=\{ \bb =(b_0,b_1,b_2)\in \fA_C \,|\, 
{\xi}_0^{b_0}{\xi}_1^{b_1}{\xi}_2^{b_2}=1,\ 
\forall \bgam \in \Gamma_Y \} \\ 
\fA_X :=\{ (\ba ,\bb )\in \fA_S \times \fA_C \,|\, 
\zeta_0^{a_0}\zeta_1^{a_1}\zeta_2^{a_2} \zeta_3^{a_3} 
\xi_0^{b_0}\xi_1^{b_1}\xi_2^{b_2} =1,\ 
\forall \bgam \in \Gamma_Y \}.  
\end{array}$$  

\begin{prop} \label{prop7-1} 
Let $S$ be a diagonal surface of degree $d$ in $\PP^3(Q_1)$ and 
$C$ be a diagonal curve of degree $e$ in $\PP^2(Q_2)$. Write 
$Y=S\times C$ and let $\Gamma_Y$ be a subgroup of $\Gamma_S 
\times \Gamma_C$ as above. Assume $(p, \# \Gamma_Y)=1$ and write 
$X=Y/\Gamma_Y$. Then the cohomology of $X$ is described as follows: 
$$\begin{array}{ll} 
H^0(X,\QQ_{\ell})= & \QQ_{\ell} \\ 
H^1(X,\QQ_{\ell})= & \bigoplus\limits_{\bb \in \fA_C^{\Gamma}} V(\bb )  \\ 
H^2(X,\QQ_{\ell})= & V(0)\oplus \bigoplus\limits_{\ba \in \fA_S^{\Gamma}} 
V(\ba ) \oplus \QQ_{\ell}(-1) \\
H^3(X,\QQ_{\ell})= & \bigoplus\limits_{\bb \in \fA_C^{\Gamma}} 
V(0)\otimes V(\bb ) \oplus \bigoplus\limits_{(\ba ,\bb )\in \fA_X} 
V(\ba )\otimes V(\bb ) 
\end{array}$$ 
Other cohomology groups can be calculated from these by the 
Poincar\'e duality. 
\end{prop} 

\begin{proof} 
Note that the actions by $\Gamma_C$, $\Gamma_S$ and $\Gamma_Y$ 
are compatible with the decomposition (\ref{coh1}) of cohomology 
groups. By the assumption $(p,\# \Gamma_Y)=1$ and an isomorphism 
(\ref{coh3}) in Section 2, we have 
$$\begin{array}{ll} 
H^0(X,\QQ_{\ell})= & \QQ_{\ell} \\ 
H^1(X,\QQ_{\ell})= & H^1(C)^{\Gamma_Y} \\ 
H^2(X,\QQ_{\ell})= & H^2(S)^{\Gamma_Y} \oplus \QQ_{\ell}(-1) \\
H^3(X,\QQ_{\ell})= & (H^2(S)\otimes H^1(C))^{\Gamma_Y}.   
\end{array}$$ 

First we consider $(H^2(S)\otimes H^1(C))^{\Gamma_Y}$. As $V(\ba )$ 
and $V(\bb )$ are subspaces of dimension one, there exist non-zero 
vectors $u_1\in V(\ba )$ and $u_2\in V(\bb )$ forming a basis for 
them respectively. Then $u_1\otimes u_2$ is a basis for $V(\ba ) 
\otimes V(\bb )$. Since $\Gamma_Y$ is a subgroup of $\Gamma_S \times 
\Gamma_C$, it acts on $Y=S\times C$ coordinate-wise. Hence $\bgam = 
(({\zeta}_0,{\zeta}_1,{\zeta}_2,{\zeta}_3), ({\xi}_0,{\xi}_1,{\xi}_2)) 
\in \Gamma_Y$ sends $u_1\otimes u_2 \in V(\ba ) \otimes V(\bb )$ to 
$$\begin{array}{ll} 
\bgam^{*}(u_1\otimes u_2) & =\zeta_0^{a_0}\zeta_1^{a_1}\zeta_2^{a_2}
\zeta_3^{a_3} u_1 \otimes \xi_0^{b_0}\xi_1^{b_1}\xi_2^{b_2} u_2 \\ 
 & = \zeta_0^{a_0}\zeta_1^{a_1}\zeta_2^{a_2} \zeta_3^{a_3} 
\xi_0^{b_0}\xi_1^{b_1}\xi_2^{b_2} (u_1\otimes u_2).
\end{array}$$ 
It follows that $V(\ba ) \otimes V(\bb )$ is invariant under 
the $\Gamma_Y$-action if and only if 
$$\zeta_0^{a_0}\zeta_1^{a_1}\zeta_2^{a_2} \zeta_3^{a_3} 
\xi_0^{b_0}\xi_1^{b_1}\xi_2^{b_2} =1$$ 
for all $\bgam \in \Gamma_Y$. Therefore 
$$(H^2(S)\otimes H^1(C))^{\Gamma_Y}\cong 
\bigoplus_{\bb \in \fA_C^{\Gamma}} V(0)\otimes V(\bb ) \oplus 
\bigoplus_{(\ba ,\bb )\in \fA_X} V(\ba )\otimes V(\bb ).$$  

Next we consider $H^1(C)^{\Gamma_Y}$. If $u_2\in V(\bb )$ denotes 
a basis for $V(\bb )$ as above, we have 
$$\bgam^{*}(u_2) =\xi_0^{b_0}\xi_1^{b_1}\xi_2^{b_2} u_2.$$ 
Hence $\Gamma_Y$ fixes $V(\bb )$ if and only if $\xi_0^{b_0}
\xi_1^{b_1}\xi_2^{b_2} =1$, namely, if and only if $\bb \in 
\fA_C^{\Gamma}$. 

The calculation of $H^2(S)^{\Gamma_Y}$ is similar to $H^1(C)^{\Gamma_Y}$. 
The difference is the existence of the subgroup $V(0)$ corresponding 
to the hyperplane section. It is fixed by the action of $\Gamma_Y$ 
and so we obtain the decomposition of $H^2(X,\QQ_{\ell})$ as asserted. 
\end{proof} 

The cohomology of $X$ becomes simpler when $\fA_C^{\Gamma} =\emptyset$. 
Such a case occurs frequently: typical examples are the diagonal inductive 
structure and many cases of the twist map discussed in the next section. 
These cases have the feature described in the following Corollary. 

\begin{cor} \label{cor3-1} 
The assumption and hypothesis of {\em Proposition \ref{prop7-1}} remain in force. 
Assume that $(v_0, e/v_0)=1$ and that $\Gamma_Y$ consists of elements 
of the form $\bgam =(({\zeta}_0,{\zeta}_1,{\zeta}_2,{\zeta}_3), 
({\xi}_0,1,1))$. Then the cohomology of $X$ is described as follows: 
$$\begin{array}{ll} 
H^0(X,\QQ_{\ell})= & \QQ_{\ell} \\ 
H^1(X,\QQ_{\ell})= & 0  \\ 
H^2(X,\QQ_{\ell})= & V(0)\oplus \bigoplus\limits_{\ba \in \fA_S^{\Gamma}} 
V(\ba ) \oplus \QQ_{\ell}(-1) \\
H^3(X,\QQ_{\ell})= & \bigoplus\limits_{(\ba ,\bb )\in \fA_X} 
V(\ba )\otimes V(\bb ) 
\end{array}$$ 
Other cohomology groups can be calculated by the Poincar\'e duality. 
\end{cor} 

\begin{proof} 
The action of $\Gamma_Y$ restricted to $C$ is given as $(\xi_0, 1,1)$. Hence 
$$\fA_C^{\Gamma}=\{ \bb =(b_0,b_1,b_2)\in \fA_C \,|\, 
{\xi}_0^{b_0}=1,\ \forall \xi_0 \in \mu_{e/v_0} \}.$$ 
As $(v_0, e/v_0)=1$, we have $b_0=0$ in $v_0\ZZ/e\ZZ$. Since there is no 
such $\bb$ in $\fA_C$ from the definition of $\fA_C$, $\fA_C^{\Gamma} =
\emptyset$ and the results follows from Proposition \ref{prop7-1}. 
\end{proof} 

\section{Zeta-functions of $K3$-fibered Calabi--Yau threefolds, I}
\label{sect7}

In this section, we compute zeta-functions of $K3$-fibered Calabi--Yau 
threefolds defined earlier. All these Calabi--Yau threefolds are realized 
as quotient threefolds by twist maps. 

\begin{defn}\label{defn7-1} {\rm Let $K$ be a number field, ${\OO}_K$
the ring of integers of $K$.  Let $V$ be a smooth projective
variety defined over $K$.  For a prime
${\fp}\in\mbox{Spec}({\OO}_K)$, let $V_{\fp}$
be the reduction of $V$ modulo ${\fp}$ and put
$q=\#({\OO}_K/{\fp})=\mbox{Norm}(\fp)$.  The
zeta-function of $V_{\fp}$ is defined by
\begin{equation*}
Z(V_{\fp},t)=\mbox{exp}\biggl( \sum_{n=1}^{\infty} 
\frac{\#V_{\fp}(\FF_{q^n})}{n} t^n \biggr)
\end{equation*}
where $t$ is an indeterminate, and $\#V_{\fp}(\FF_{q^n})$ is 
the number of rational points on $V_{\fp}$ over $\FF_{q^n}$.

It is known (Deligne \cite{De}) that $Z(V_{\fp},t)$ is a
rational function and indeed has the form
\begin{equation*}
Z(V_{\fp},t)=\prod_{i=0}^{2\dim (V_{\fp})} P_i(V_{\fp}, 
t)^{(-1)^{i+1}}
\end{equation*}
where $P_i(V_{\fp},t)$ is the characteristic polynomial of
the endomorphism on the $i$--th \'etale cohomology group
$H^i(V_{\fp},\QQ_{\ell})$ induced by the Frobenius morphism 
$\mbox{Frob}_q$ on $V_{\fp}$.  
$$P_i(V_{\fp}, t)=\det (1-\mbox{Frob}_q^{*} t\mid H^i(V_{\fp}, 
\QQ_{\ell})).$$ 
One knows that $P_i(V_{\fp}, t)$ is a polynomial with integer 
coefficients of degree equal to the $i$--th Betti number 
$B_i(V_{\fp})$ and its reciprocal roots have the absolute 
value $q^{1/2}$. } 
\end{defn}

Our goal of this section is to determine the zeta-function of Calabi--Yau
varieties constructed in Section 5. Since our Calabi--Yau varieties 
are quotients of products of lower dimensional varieties. For instance, for
Calabi--Yau threefolds $X$, they are quotients of products $S\times C$ of
surfaces $S$ and curves $C$. Then the eigenvalues of the Frobenius for
$X$ are given by products of the eigenvalues of the Frobenius on the components.
For dimensions $1$ and $2$ Calabi--Yau varieties discussed in the section 5, 
the zeta-functions have been determined by Goto \cite{Go1, Go2}.  

To describe the eigenvalues and the zeta-functions, we need to introduce 
weighted Jacobi sums.  The reader is referred to
Gouv\^ea and Yui \cite{GY} for Jacobi sums and their properties relevant
to our discussions. 

\begin{defn}\label{defn7-2}
{\rm Let $\LL=\QQ(e^{2\pi \sqrt{-1}/d})$ be the $d$--th cyclotomic field 
over $\QQ$, ${\OO}_{\LL}$ the ring of integers of $\LL$.  Let ${\fp}\in
\mbox{Spec}({\OO}_{\LL})$. For every $x\in {\OO}_{\LL}$
relatively prime to ${\fp}$, let
$\chi_{\fp}(x\ \mbox{mod}\, {\fp})=(\frac{x}{\fp})$
be the $d$--th power residue symbol on $\LL$.  If $x\equiv 0 
\pmod{\fp}$, we put $\chi_{\fp}(x\ \mbox{mod}\, {\fp})=0$.
Let $(w_0,w_1,w_2,\cdots, w_{n+1})$ be a weight.  Define the set
\begin{equation*}
\begin{split}
& {\fA}_d(w_0,w_1,\cdots, w_{n+1}) \\
&:=\biggl\{ {\ba}=(a_0,a_1,\cdots,a_{n+1})\,|\, a_i\in (w_i\ZZ/d\ZZ), 
a_i\neq 0,
\sum_{i=0}^{n+1} a_i\equiv 0\pmod d \biggr\}. 
\end{split}
\end{equation*}
For each ${\ba}\in{\fA}_d(w_0,w_1,\cdots, w_{n+1})$, the
{\it Jacobi sum} is defined by
\begin{equation*}
j_{\fp}(\ba)=
j_{\fp}(a_0,a_1,\cdots, a_{n+1}):=(-1)^n \sum\chi_{\fp}(v_1)^{a_1}
\chi_{\fp}(v_2)^{a_2}\cdots \chi_{\fp}(v_{n+1})^{a_{n+1}}
\end{equation*}
where the sum is taken over $(v_1,v_2,\cdots, v_{n+1})\in 
({\OO}_{\LL}/{\fp})^{\times} \times \cdots \times 
({\OO}_{\LL}/{\fp})^{\times}$ subject to the linear
relation $1+v_1+v_2+\cdots + v_{n+1}\equiv 0 \pmod{\fp}$.

Jacobi sums are elements of ${\OO}_{\LL}$ with complex absolute
value equal to $q^{n/2}$ where $q=\mid \mbox{Norm}\,{\fp} \mid \equiv 1
\pmod{d}$. }
\end{defn}

The Jacobi sums defined above are particularly useful when we describe 
zeta-functions of varieties over some extensions of $\LL= 
\QQ(e^{2\pi \sqrt{-1}/d})$. To discuss zeta-functions of varieties over 
$\QQ$, however, we also need Jacobi sums of finite fields.  

\begin{defn}\label{defn7-3}
{\rm Let $\FF_q$ be a finite field of $q$ elements. Let $d$ be a 
positive integer. Assume that $q\equiv 1\pmod{d}$. Fix a character 
$\chi :\FF_q^{\times} \lra \CC^{\times}$ of exact order $d$. Define 
the set ${\fA}_d(w_0,w_1,\cdots, w_{n+1})$ as in Definition \ref{defn7-2}. 
The Jacobi sum of $\FF_q$ associated with $\ba =(a_0 ,\cdots ,a_{n+1})$ 
is defined to be 
$$j_q(\ba ):=(-1)^n\sum {\chi}^{a_1}(v_1)\cdots \chi^{a_{n+1}}(v_{n+1})$$  
where the sum is taken over all $v_i\in \FF_q^{\times}$ satisfying 
$1+v_1+ \cdots +v_{n+1}=0$. } 
\end{defn}

Jacobi sums $j_q(\ba )$ have absolute value $q^{n/2}$. These Jacobi sums 
are used to describe zeta-functions of varieties over $\QQ$ reduced modulo $p$. 

\subsection{Zeta-functions of elliptic curves and $K3$ surfaces} 

\begin{lem}\label{lem7-2} Let $E_i$ $(i=1,2,3)$ be elliptic curves 
listed in {\em Table \ref{table1}}. Assume that the field of 
definition of $E_i$ $(i=1,2,3)$ is a number field $K$ that contains 
$\LL =\QQ(e^{2\pi \sqrt{-1}/d})$. For $\fp\in \mbox{\em Spec}{\OO}_K$, 
let $E_{i,\fp}$ be the reduction of $E_i$ modulo ${\fp}$. Then the 
zeta-function $Z(E_{i,\fp},t)$ has the form
\begin{equation*}
Z(E_{i,{\fp}},t)=\frac{P_1(E_{i,{\fp}},t)}{(1-t)(1-qt)}
\end{equation*}

The polynomial $P_1(E_{i,{\fp}},t)$ is determined in {\em Table 
\ref{table7}}. 
\end{lem} 

\begin{table} 
\caption{Polynomial $P_1(E_{i,{\fp}},t)$}  
\label{table7} 
\[
\begin{array}{c|c|c|c} \hline \hline
E_i & P_1(E_{i,\fp}, t) &  {\fp}  & \LL \\ \hline
E_1 & (1-j_{\fp}(1,1,1)t)(1-j_{\fp}(2,2,2)t) & {\fp}\nmid 
3 & \QQ(e^{2\pi i/3}) \\ 
    &  1                                                     & {\fp}|3     
&  \\ \hline   
E_2 & (1-j_{\fp}(1,1,2)t)(1-j_{\fp}(3,3,2)t) & {\fp}\nmid 
4 & \QQ(e^{2\pi i/4})\\
    &  1                                                     & {\fp}| 4     
&  \\ \hline
E_3 & (1-j_{\fp}(1,2,3)t)(1-j_{\fp}(5,4,3)t) & {\fp}\nmid 
6 & \QQ(e^{2\pi i/6}) \\ 
    &  1                                                     & {\fp}|6  & 
\\ \hline
\hline
\end{array}
\]
\end{table} 

\begin{proof} 
$E_{i,\fp}$ is an elliptic curve over $\FF_q$ with $q=\mid \mbox{Norm}\,{\fp} 
\mid$. As $q\equiv 1\pmod{d}$ for every $\fp$, each $j_{\fp}$ is a Jacobi 
sum of $\FF_q$ and we can apply Weil's method to compute $Z(E_{i,{\fp}},t)$. 
See \cite{Go1} or Lemma \ref{weil} below. 
\end{proof} 

By Lemma \ref{lem7-2}, zeta-functions of $E_i$ over an extension of $\LL$ 
are written in a single form throughout good reductions at $\fp$. When 
$E_i$ is defined over $\QQ$, we need to divide the case according to a 
congruence property of a prime $p$. 

\begin{lem}\label{lem7-2e} Let $E_i$ $(i=1,2,3)$ be elliptic curves 
over $\QQ$ listed in {\em Table \ref{table1}}. For a prime $p\in \QQ$, 
let $E_{i,p}$ be the reduction of $E_i$ modulo $p$. Then the 
zeta-function $Z(E_{i,p},t)$ has the form
\begin{equation*}
Z(E_{i,p},t)=\frac{P_1(E_{i,p},t)}{(1-t)(1-pt)}
\end{equation*}

The polynomial $P_1(E_{i,p},t)$ is determined in {\em Table 
\ref{table7e}}. 
\end{lem} 

\begin{table} 
\caption{Polynomial $P_1(E_{i,p},t)$}  
\label{table7e} 
\[
\begin{array}{c|c|l} \hline \hline
E_i & P_1(E_{i,p}, t) &  p \\ \hline
E_1 & (1-j_p(1,1,1)t)(1-j_p(2,2,2)t) & p\equiv 1\pmod{3} \\ 
 & (1-\sqrt{p}t)(1+\sqrt{p}t) & p\equiv 2\pmod{3} \\ 
 &  1 & p=3 \\ \hline   
E_2 & (1-j_p(1,1,2)t)(1-j_p(3,3,2)t) & p\equiv 1\pmod{4} \\ 
 & (1-\sqrt{p}t)(1+\sqrt{p}t) & p\equiv 3\pmod{4} \\ 
 &  1 & p=2 \\ \hline   
E_3 & (1-j_p(1,2,3)t)(1-j_p(5,4,3)t) & p\equiv 1\pmod{6} \\ 
 & (1-\sqrt{p}t)(1+\sqrt{p}t) & p\equiv 5\pmod{6} \\ 
 &  1 & p=2,3 \\ \hline   
\hline
\end{array}
\]
\end{table} 

\begin{proof} 
When $p\equiv 1\pmod{d}$, the situation is the same as in Lemma \ref{lem7-2} 
and $Z(E_{i,p},t)$ can be expressed with Jacobi sums on $\FF_p$. 

When $p\not\equiv 1\pmod{d}$, we need to consider several extensions of the 
field of definition and collect more data to determine the zeta-functions. 
Let $n$ be the extension degree of $\FF_q$ over $\FF_p$. 

\noindent (i) $E_1$ over $\FF_p$ with $p\equiv 2\pmod{3}$. 

(a) Let $n$ be an odd integer and $q=p^n$. We have $q\equiv 2\pmod{3}$ 
and $E_1$ has the same number of $\FF_q$-rational points as a curve 
$y_0^{\gcd (3,q-1)}+y_1^{\gcd (3,q-1)}+y_2^{\gcd (3,q-1)}=0$, namely 
$$C:\ y_0+y_1+y_3 =0 \quad \subset \PP^2.$$ 
Clearly, $\# E_{1,p} (\FF_q )=\# C_p (\FF_q )=1+q$.   

(b) Let $n=2m$ be an even integer and $q=p^n$. We have $q\equiv 1\pmod{3}$ 
and we can apply Weil's algorithm to obtain 
$$\# E_{1,p} (\FF_q )=1+q-j_{p^2}(1,1,1)^m-j_{p^2}(2,2,2)^m.$$ 
It follows from \cite{BEW}, Theorem 11.6.1 (cf. Remark \ref{g-sum}) that 
$j_{p^2}(1,1,1)=j_{p^2}(2,2,2)=p$. Hence $\# E_{1,p} (\FF_q )=1+q-2p^m$. 

Combining (a) and (b), we determine $\# E_{1,p} (\FF_{p^n} )= 
1+p^n-(\sqrt{p})^n-(-\sqrt{p})^n$, which gives rise to the desired 
formula. 

\noindent (ii) $E_2$ over $\FF_p$ with $p\equiv 3\pmod{4}$. 

(a) Let $n$ be an odd integer and $q=p^n$. We have $q\equiv 3\pmod{4}$ and 
$E_2$ has the same number of $\FF_q$-rational points as  
$y_0^{\gcd (4,q-1)}+y_1^{\gcd (4,q-1)}+ y_2^{\gcd (2,q-1)}=0$, namely 
$$C:\ y_0^2+y_1^2+y_3^2 =0 \quad \subset \PP^2.$$ 
We see immediately that $\# E_{2,p} (\FF_q )=\# C_p (\FF_q )=1+q$.   

(b) Let $n=2m$ be an even integer and $q=p^n$. We have $q\equiv 1\pmod{4}$ and 
$$\# E_{2,p} (\FF_q )=1+q-j_{p^2}(1,1,2)^m-j_{p^2}(3,3,2)^m.$$ 
By \cite{BEW}, Theorem 11.6.1 again, we see $j_{p^2}(1,1,2)=j_{p^2}(3,3,2)
=p$. Hence $\# E_{2,p} (\FF_q )=1+q-2p^m$.

Combining (a) and (b), we determine $\# E_{2,p} (\FF_{p^n} )= 
1+p^n-(\sqrt{p})^n-(-\sqrt{p})^n$. 

\noindent (iii) $E_3$ over $\FF_p$ with $p\equiv 5\pmod{6}$. 

Similarly as above, we see that $\# E_{3,p} (\FF_{p^n} )= 
1+p^n-(\sqrt{p})^n-(-\sqrt{p})^n$ and obtain the polynomial $P_1(E_{3,p},t)$ 
as claimed. 
\end{proof} 

\begin{rem} \label{g-sum}
Let $m$ be a positive integer and $p$ be a prime such that $p^r\equiv -1 
\pmod{m}$ for some $r$. Write $q=p^{2r}$ and let $\chi$ be a character of 
$\FF_q$ of order $m$. Denote by $G_{2r}(\chi )$ the Gauss sum over $\FF_q$ 
associated with $\chi$. Then Theorem 11.6.1 of \cite{BEW} states that 
$$p^{-r}G_{2r}(\chi )=
\begin{cases} 
1 & \mbox{ if } p=2 \\ 
(-1)^{\frac{p^r+1}{m}} & \mbox{ if } p>2 
\end{cases}$$ 
\end{rem} 

Next we calculate the zeta-functions of $K3$ surfaces and discuss their 
decomposition into algebraic and transcendental parts. 

\begin{defn} 
Let $Y$ be a $K3$ surface over $k=\FF_q$. Let $A(Y)$ be the subspace of 
elements in $H^2(Y_{\ok},\QQ_{\ell}(1))$ that are invariant under the 
action of $\mbox{Gal}(\ok /k^{'})$ for some finite extension $k^{'}$ of 
$k$. Let $V(Y)$ be the orthogonal complement of $A(Y)$ in $H^2(Y_{\ok}, 
\QQ_{\ell}(1))$ with respect to the cup-product. We define 
$$\begin{array}{ll} 
P_2(S_{Y},t) & =\det (1-\mbox{Frob}_q^{*} qt\mid A(Y)) \\ 
P_2(T_{Y},t) & =\det (1-\mbox{Frob}_q^{*} qt\mid V(Y)) 
\end{array}$$ 
and call $P_2(S_{Y},t)$ (resp. $P_2(T_{Y},t)$) the {\it algebraic} 
(resp. {\it transcendental}) part of $P_2(Y,t)$. 
\end{defn} 

\begin{rem} 
The above definition is originally due to Zarhin~\cite{Zar} for the 
ordinary case. We simply extend his definition to arbitrary $K3$ 
surfaces. It follows from the definition that $\mbox{Frob}_q^{*}$ 
acts on both $A(Y)$ and $V(Y)$. If the Tate conjecture is true for $Y$, 
then $A(Y)$ is equal to the image of $NS(Y_{\ok})\otimes \QQ_{\ell}$ in 
$H^2(Y_{\ok},\QQ_{\ell}(1))$ under the cycle class map. The subspace 
$V(Y)$ is then isomorphic to $V_{\ell}(Br(Y_{\ok}))$. The Tate 
conjecture is known for $K3$ surface of finite height (cf. \cite{NO}) 
or when $H^2(Y_{\ok},\QQ_{\ell}(1))$ is spanned all by algebraic 
cycles (i.e., $Y$ is supersingular).  
\end{rem} 

\begin{prop}\label{prop7-4} Let $K$ be a number field. Let $Y_i$ 
$(i=1,2,\cdots, 10)$ be the $K3$ surfaces constructed in 
{\em Proposition \ref{prop5-1}} by diagonal equations 
$$y_0^{m_0}+y_1^{m_1}+y_2^{m_2}+y_3^{m_3}=0$$ 
of degree $d$ in $\PP^3(w_0,w_1,w_2,w_3)$ over $K$. Assume that 
$K$ contains the $d$-th cyclotomic field $\QQ (\zeta_d )$. Let 
$\fp$ be a prime of $K$ not dividing $d$ and put $q=\mid 
\mbox{Norm}\,{\fp} \mid$. Then the following assertions hold. 

$(a)$ $q\equiv 1\pmod{d}$ and the zeta-function $Z(Y_{i,\fp},t)$ 
has the form
\begin{equation*}
Z(Y_{i,\fp},t)=\frac{1} {(1-t)P_2(Y_{i,\fp},t)(1-q^2t)}
\end{equation*}
where $P_2(Y_{i,\fp},t)$ is a polynomial of integer coefficients 
whose reciprocal roots have the complex absolute value $q$ and 
$$P_2(Y_{i,\fp},t)=(1-qt)\prod_{{\mathbf a} \in {\fA}_d} 
(1-j_{\fp} ({\mathbf a})t)$$ 
with ${\fA}_d={\fA}_d(w_0,w_1,w_2,w_3)$. 

$(b)$ Let $\ytil_{i,\fp}$ be the minimal resolution of $Y_{i,\fp}$ 
and let $\EE$ denote the proper transform of $\mbox{\rm Sing}\, 
(\overline{Y_{i,\fp}})$ on $\ytil_{i,\fp}$, where $\overline{Y_{i,\fp}}=
Y_{i,\fp} \times \overline{\FF}_q$. Then $\EE$ is a scheme 
over $\FF_q$ consisting of exceptional divisors and 
$$Z(\ytil_{i,\fp},t)=\frac{Z(Y_{i,\fp},t)}{P_2(\EE ,t)}= 
\frac{1} {(1-t)P_2(Y_{i,\fp},t)P_2(\EE ,t)(1-q^2t)}.$$ 
In particular, $P_2(\ytil_{i,\fp},t)=P_2(Y_{i,\fp},t)P_2(\EE ,t)$. 

$(c)$ $P_2(\ytil_{i,\fp},t)$ factors over $\ZZ$ as follows:
\begin{equation*}
P_2(\ytil_{i,\fp},s)= P_2(S_{\ytil_{i,\fp}},t) P_2(T_{\ytil_{i,\fp}},t).
\end{equation*}
Here if $p$ is a rational prime such that $\fp \cap \ZZ =(p)$, then  
$$P_2(T_{\ytil_{i,\fp}},t) =
\begin{cases} 
1 & \mbox{ if $p^r\equiv -1\pmod{d}$ for some $r\in \NN$ } \\ 
\prod_{{\mathbf a} \in \MM_Q} (1-j_{\fp} ({\mathbf a})t) & 
\mbox{ otherwise} 
\end{cases}$$ 
where $\MM_Q$ is the weight motive of $Y_{i,\fp}$, i.e. the 
$(\ZZ /d\ZZ )^{\times}$-orbit of ${\mathbf a}=(w_0,w_1,w_2,w_3)$ in 
${\fA}_d$.
\end{prop} 

\begin{proof} 
(a) Since $K$ contains $\QQ (\zeta_d )$, we have $q\equiv 1\pmod{d}$ 
for every prime $\fp$. As $Y_i$ is a weighted diagonal surface over 
$\FF_q$, $Z(Y_{i,\fp},t)$ can be computed by Weil's classical method 
(see \cite{Go1}). 

(b) Each singular point in $\mbox{Sing}\, (\overline{Y_{i,\fp}})$ may 
not be defined over $\FF_q$, but the subscheme $\mbox{Sing}\, 
(\overline{Y_{i,\fp}}) =\overline{Y_{i,\fp}} \cap \mbox{Sing}\, (\PP^3(Q))$ 
is defined over $\FF_q$. Since $\mbox{Sing}\, (\PP^3(Q))$ is defined over 
$\FF_q$ and an embedded resolution exists for $\overline{Y_{i,\fp}}$, the 
resolution $\ytil_{i,\fp}$ and a subscheme $\EE$ are defined over $\FF_q$. 
As $Y_{i,\fp}$ has only cyclic quotient singularities, $\EE$ consists 
of rational lines. Hence $P_1(\EE ,t)=1$ and $Z(\ytil_{i,\fp},t)$ has the 
asserted form. 

(c) Since $Y_{i,\fp}$ is a weighted diagonal surface, the Tate conjecture 
holds for $\ytil_{i,\fp}$ over $\FF_q$ (cf. \cite{Go1}) and its 
supersingularity is dependent on $p$. (Note that $\fp \nmid d$ implies 
$p\nmid d$.) As $\ytil_{i,\fp}$ is a $K3$ surface, it is known that 
$\ytil_{i,\fp}$ is supersingular if and only if $p^r \equiv -1\pmod{d}$ for 
some positive integer $r$ (see \cite{Go3}). 
When $\ytil_{i,\fp}$ is supersingular, all cycles are algebraic and hence 
$P_2(T_{\ytil_{i,\fp}},t) =1$; otherwise, the weight motive $\MM_Q$ 
corresponds to the transcendental cycles. Therefore $P_2(T_{\ytil_{i,\fp}},t)$ 
has the form as we claim. 
\end{proof} 

\begin{rem} \label{zeta-over-Fp} 
Since we assume $K\supset \QQ (\zeta_d )$, the congruence $q\equiv 1\pmod{d}$ 
holds for every prime $\fp$. If we remove this assumption, it may happen that 
$q\not\equiv 1\pmod{d}$ for some classes of $\fp$. In this case, the number of 
$\OO_K/\fp$-rational points on $Y_{\fp}$ is the same as that of 
$$y_0^{\gcd (m_0,q-1)}+y_1^{\gcd (m_1,q-1)}+y_2^{\gcd (m_2,q-1)}+ 
y_3^{\gcd (m_3,q-1)}=0.$$ 
The degree, say $d^{'}$, of this surface now satisfies $q\equiv 1\pmod{d^{'}}$ 
and we can express the number of rational points in terms of Jacobi sums 
associated with a character of order $d^{'}$. If we carry out this calculation 
over sufficiently many finite extensions of $\FF_q$, we can obtain the 
zeta-function of $Y$ over $\FF_q$ with $q\not\equiv 1\pmod{d}$. Furthermore, 
this algorithm works also over a finite prime field $\FF_p$. This is what 
one can do in general to compute the $L$-series of $Y$ over $\QQ$.  
\end{rem} 

\begin{rem} 
It depends on the prime $(p)=\fp \cap \ZZ$ whether or not $\ytil_{i,\fp}$ 
is supersingular. The above condition for a supersingular prime $p$ 
can be phrased as follows: let $\zeta_d =e^{2\pi i/d}$ and write 
$\LL=\QQ(\zeta_d )$ for the $d$-th cyclotomic field over $\QQ$. Choose 
a prime $\wp \in \mbox{Spec}({\OO}_{\LL})$ over $p$ and assume $\wp  
\nmid d$. Then $\wp$ is unramified in $\LL /\QQ$, so 
that the inertia group of $\wp$ is trivial. Hence its decomposition 
group, $Z_{\wp}$, is isomorphic to $\mbox{Gal}(\FF_q /\FF_p)$, where 
$\FF_q \cong {\OO}_{\LL}/\wp$ and $q=p^f$. Let $\sigma_F$ be the Frobenius 
automorphism that generates $\mbox{Gal}(\FF_q /\FF_p)$. Then $Z_{\wp}$ is 
generated by $\sigma_F$ and it acts on $\LL$ as $\sigma_F (\zeta_d )=
\zeta_d^p$ of order $f$. When $p^r \equiv -1\pmod{d}$, $\sigma_F^r$ acts on 
$\QQ (\zeta_d )$ as 
$$\sigma_F^r (\zeta_d )=\zeta_d^{p^r} =\zeta_d^{-1} =\overline{\zeta_d}.$$ 
In other words, $Z_{\wp}$ contains the complex conjugate map restricted 
to $\QQ (\zeta_d )$. The converse is also true. Therefore $\ytil_{i,\fp}$ 
is supersingular if and only if $Z_{\wp}$ contains the complex conjugate. 
\end{rem} 

Next we deal with quasi-diagonal surfaces; threefolds will be discussed 
in the next section. 

\begin{lem} \label{weil} Let $k=\FF_q$ be a finite field of $q$ elements. 
Let $m_1,\cdots ,m_r$ be $r$ positive integers such that $q\equiv 1\pmod{m_i}$ 
for $1\leq i\leq r$. Write $W_k$ be an affine variety in $\Aff^r_k$ defined by 
the equation  
$$b_0+b_1x_1^{m_1}+\cdots +b_rx_r^{m_r}=0\quad \subset \Aff^r_k$$  
with $b_i\in k^{\times}$ $(0\leq i\leq r)$. Define 
$$\begin{array}{l} 
M=\mbox{\em lcm}\, (m_1,\cdots ,m_r) \\  
M_i=M/m_i  
\end{array}$$   
for $1\leq i\leq r$. Assume $q\equiv 1\pmod{M}$. Fix a character, $\chi$, of 
$k^{\times}$ of exact order $M$. Let $N(W)$ denote the number of $k$-rational 
points on $W_k$. Then 
$$N(W)=q^{r-1}+\sum_{\ba \in \fA_0} {\mathcal J} (\bb ,\ba )$$ 
where 
$$\begin{array}{l}  
\bb =(b_0,\cdots ,b_r) \\  
\fA_0 =\biggl\{ \ba =(a_0 ,\cdots ,a_r ) \Big| \ 
\begin{array}{l} 
a_i \in M_i\ZZ /M\ZZ ,\ a_i \neq 0\ (1\leq i\leq r),\ a_0 \in \ZZ /M\ZZ \\ 
\sum_{i=0}^r a_i =0 
\end{array} \biggr\} \\  
j(\ba )=\sum_{\substack{v_i\in k^{\times} \, (1\leq i\leq r) \\ 
1+v_1+ \cdots +v_r=0}} {\chi}^{a_1}(v_1)\cdots {\chi}^{a_r}(v_r) \\ 
{\mathcal J} (\bb ,\ba )=\chi^{-1}(b_0^{a_0}\cdots b_r^{a_r}) j(\ba ) 
\end{array}$$  
\end{lem} 

\begin{proof} See \cite{IR}, Chap. 8. 
\end{proof} 

\begin{prop}\label{zeta-qd2} Let $Y$ be a quasi-diagonal surface over 
$\FF_q$ defined by an equation 
$$y_0^{m_0}y_1+y_1^{m_1}+y_2^{m_2}+y_3^{m_3}=0$$
in $\PP^3(w_0,w_1,w_2,w_3)$ of degree $d$, where $d=w_0m_0+w_1=w_1m_1
=w_2m_2=w_3m_3$. Write 
$$\begin{array}{l}   
M=\mbox{\em lcm}\, (m_0,m_2,m_3)  \\ 
M_i=M/m_i \qquad (i=0,2,3)    
\end{array}$$   
and assume $q\equiv 1 \pmod{M}$. Fix a character, $\chi$, of $k^{\times}$ 
of exact order $M$. Then the zeta-function of $Y$ has the form   
$$Z(Y,t)=\frac{1}{(1-t)P_2(Y,t)(1-q^2t)}$$
where
\begin{align}
 & P_2(Y,t)=(1-qt) \prod_{a\in \fL} ( 1-q\chi^{aM_3}(-1)t) 
\prod_{{\mathbf a} \in {\fA_Y}} (1-j_{\fp} ({\mathbf a})t) \nonumber \\
 & \fL =\{ a\in \ZZ \mid 0<a<m_3,\ w_3a\equiv 0\pmod{w_2} \} \nonumber \\ 
 & \fA_Y =\biggl\{ {\mathbf a} =(a_0,a_1,a_2,a_3) \Big| 
\begin{array}{l}  
a_1\in \ZZ /M\ZZ ,\ a_i\in M_i\ZZ /M\ZZ \ (i=0,2,3), \\
a_i\neq 0 \ (0\leq i\leq 3),\ \sum_{i=0}^{3}a_i=0\ \text{\em 
and } \ a_0+m_1a_1=0
\end{array}
\biggr\} \nonumber \\
\end{align} 
\end{prop} 

\begin{proof}
Here we give an outline and cohomological explanation of the proof. 
A more detailed and combinatorial account of it (in the case 
of threefolds) may be found in Theorem \ref{zeta-qd3}. 

Note first that $Y$ is a Delsarte surface with matrix 
$$\left[ \begin{array}{cccc}
m_0 & 1 & 0 & 0 \\  
0 & m_1 & 0 & 0 \\ 
0 & 0 & m_2 & 0 \\
0 & 0 & 0 & m_3   
\end{array} \right] .$$  
The zeta-functions of such quasi-diagonal surfaces are computed in 
\cite{Go2} by a combinatorial argument. The main idea there is to 
split the surface $Y$ according to $y_1=0$ or $y_1\neq 0$. The subset 
corresponding to $y_1=0$ consists of $r=d/\mbox{lcm}\, (w_2,w_3)$ 
lines and this gives rise to a factor $\prod_{a\in \fL} 
( 1-q\chi^{aM_3}(-1)t)$ of $P_2(Y,t)$. The other factor of it  
comes from the diagonal surface of degree $M$. As described in \cite{Go3}, 
$Y$ is covered by the diagonal surface $F$ of degree $m_0m_1m_2m_3$ and 
there is a group action $\Gamma$ on $F$ so that $Y$ is birational to 
$F/\Gamma$. Computing the cohomology of $F/\Gamma$, we obtain the 
factor $\prod_{{\mathbf a} \in {\fA}} (1-j_{\fp} ({\mathbf a})t)$ of 
$P_2(Y,t)$.
\end{proof} 

\begin{rem} 
In general, the modulus $M$ above is different from the degree $d$ 
of the surface. Proposition \ref{zeta-qd2} shows that the character 
set $\mathfrak A$ of a quasi-diagonal surface can be embedded in  
${\fA}_M (1,1,1,1)$. 
\end{rem} 

\begin{cor} \label{cor7-5} Let $Y$ be a quasi-diagonal surface over a 
number field $K$ defined by an equation 
$$y_0^{m_0}y_1+y_1^{m_1}+y_2^{m_2}+y_3^{m_3}=0\ \subset 
\PP^3(w_0,w_1,w_2,w_3).$$  
Assume that $K$ contains the $M$-th cyclotomic field $\QQ (\zeta_M )$. 
Let $\fp$ be a prime of $K$ not dividing $d$. Then the assertions 
$(a)$, $(b)$ and $(c)$ of {\em Proposition \ref{prop7-4}} hold for 
$Y_{\fp}$ if we replace $d$ by $M$.  
\end{cor} 

\begin{ex}\label{zeta-y11} Let $Y_{11}$ be the $K3$ surface over $K$ 
constructed in Proposition \ref{prop5-1}:   
$$Y_{11}:\ y_0^{12}y_1+y_1^{11}+y_2^3+y_3^2=0 \subset \PP^3(5,6,22,33).$$ 
For the values defined in Proposition \ref{zeta-qd2}, we find $M=12$ and 
$r=1$. If $\fp \nmid 2\cdot 3\cdot 5\cdot 11$, then the zeta-function 
$Z(Y_{11,\fp},t)$ has the form
\begin{equation*}
Z(Y_{11,\fp},t)=\frac{1} {(1-t)P_2(Y_{11,\fp},t)(1-q^2t)}
\end{equation*}
where $P_2(Y_{11,\fp},t)$ is a polynomial of degree $5$ and decomposed as 
$$P_2(Y_{11,\fp},t)=(1-qt)\prod_{{\mathbf a} \in {\fA}_{Y_{11}}} 
(1-j_{\fp} ({\mathbf a})t)$$ 
with ${\fA}_{Y_{11}}=\{ (1,1,4,6), (5,5,8,6), (7,7,4,6), (11,11,8,6) \} \subset 
{\fA}_{12} (1,1,1,1)$. Furthermore, if $Y_{11,\fp}$ is not supersingular, then 
$$P_2(T_{Y_{11,\fp}},t)=\prod_{{\mathbf a} \in {\fA}_{Y_{11}}}  
(1-j_{\fp} ({\mathbf a})t).$$ 
\end{ex} 

\section{Zeta-functions of $K3$-fibered Calabi--Yau threefolds, II} 
\label{sect10}

Now we calculate the zeta-functions of Calabi--Yau threefolds. In 
order to have a simple exposition, we slightly raise 
the field of definition and express the zeta-function in a single 
form (independent of the choice of prime reduction). 

As we noted in Section \ref{sect2}, we may bring in non-trivial 
coefficients $c_i$ to the defining equations of our hypersurfaces. 
The geometry will be the same, but the arithmetic will be different. 
We will give more details about the differences in a subsequent paper. 

We begin with some notations to describe singular loci and resolution 
of singularities for our threefolds. 

\begin{defn} \label{index-sing} 
Let $X$ be a weighted hypersurface in $\PP^4(w_0,w_1,w_2,w_3,w_4)$ 
over $k$ defined by $f=0$. Let $I_1$ be a set of indecies which 
parameterizes one-dimensional singular loci of $X$, namely 
$$I_1=\{ \bi =(i,j,k) \mid \gcd (w_i,w_j,w_k) \geq 2\}.$$ 
For each $\bi =(i,j,k)$, write $\hat{\bi} =\{ 0,1,2,3,4 \} \setminus 
\{ i,j,k\}$ and let $C_{\bi}$ be a subscheme of $X$ over $k$ defined by 
$$C_{\bi}:\ f=0 \mbox{ and } x_h =0\ (h\in \hat{\bi}).$$ 
Similarly, let $I_0$ be a set of indecies which parameterizes 
zero-dimensional singular loci of $X$, namely 
$$I_0=\{ \bj =(i,j) \mid \gcd (w_i,w_j) \geq 2\}.$$ 
For each $\bj =(i,j)$, write $\hat{\bj} =\{ 0,1,2,3,4 \} \setminus 
\{ i,j\}$ and let $P_{\bj}$ be a subscheme of $X$ over $k$ defined by 
$$P_{\bj}:\ f=0 \mbox{ and } x_h =0\ (h\in \hat{\bj}).$$ 
Finally, for the case of quasi-diagonal hypersurfaces, we put  
$$P_0 :=(1:0:0:0:0).$$ 
\end{defn} 

In most cases, $C_{\bi}$ are irreducible (i.e. varieties over $k$) and 
$P_{\bj}$ are reducible. As a scheme, $P_{\bj}$ is defined over $k$; 
but, if we decompose it as a set of points, each point may not be 
defined over $k$. 

It follows from the discussion of Section \ref{sect6} that if $X$ is 
a weighted diagonal hypersurface, then $\mbox{Sing}\, (X)$ is a scheme 
over $k$ and we have   
$$\mbox{Sing}\, (X) =\bigcup_{\bi \in I_1} C_{\bi} \cup 
\bigcup_{\bj \in I_0} P_{\bj}.$$ 

\begin{lem} \label{zeta-e-div} 
Let $X$ be a quasi-smooth weighted hypersurface in $\PP^4(Q)$ over $K$. 
Let $\fp$ be a prime of $K$ and $X_{\fp}$ be the reduction of $X$ 
modulo $\fp$. Put $q=\# (\OO_K /\fp)$. Assume that $X_{\fp}$ is quasi-smooth. 
Then there exists a crepant resolution $\xtil_{\fp}$ of $X_{\fp}$ that is 
obtained by applying a partial toroidal desingularization to $\PP^4(Q)$. Let 
$\widetilde{C}_{\bi}$ {\rm (}resp. $\widetilde{P}_{\bj}${\rm )} be the 
strict transform of $C_{\bi}$ {\rm (}resp. $P_{\bj}${\rm )} 
with respect to $\xtil_{\fp} \lra X_{\fp}$. Write $\widetilde{C}_{\bi}^{\circ}$ 
{\rm (}resp. $\widetilde{P}_{\bj}^{\circ}${\rm )} for the open subset 
of $\widetilde{C}_{\bi}$ (resp. $\widetilde{P}_{\bj}$) defined by 
subtracting the proper transform of $C_{\bi}$ {\rm (}resp. $P_{\bj}${\rm )}. 
Then we have 
$$Z(\widetilde{C}_{\bi}^{\circ} ,t)= 
\frac{P_1(C_{\bi},qt)^{n_{\bi}}}{(1-qt)^{n_{\bi}}(1-q^2t)^{n_{\bi}}}, 
\qquad 
Z(\widetilde{P}_{\bj}^{\circ} ,t)=\frac{1}{P_0(P_{\bj},qt)^{e_{\bj}} 
P_0(P_{\bj},q^2t)^{e_{\bj}}}$$ 
where $n_{\bi}$ is the number of ruled surfaces over ${C}_{\bi}$ and 
${e_{\bj}}$ is the number of projective planes over ${P}_{\bj}$. 
\end{lem} 

\begin{proof} 
It follows from Lemma \ref{field-defn} that $\widetilde{C}_{\bi}^{\circ}$ 
consists of $n_{\bi}q^{\nu}N_{\nu}(C_{\bi})$ rational points over the 
$\nu$-th extension of $\OO_K /\fp$. Similarly, Lemma \ref{zeta-resol0} 
implies that $\widetilde{P}_{\bj}^{\circ}$ acquires $e_{\bj}(q^{\nu} 
+q^{2\nu})$ rational points. Hence their zeta-functions are computed 
as above. 
\end{proof} 

\begin{thm}\label{thm7-5} 
Let $K$ be a number field. Let $X$ be a singular Calabi--Yau threefold 
over $K$ of diagonal type constructed in {\em Section $4$} by the equation 
$$z_0^{m_0}+z_1^{m_1}+z_2^{m_2}+z_3^{m_3}+z_4^{m_4}=0$$ 
in $\PP^4(w_0,w_1,w_2,w_3,w_4)$ of degree $d=w_0+\cdots +w_4$. Assume 
that $K$ contains the $d$-th cyclotomic field $\QQ (\zeta_d )$. Let 
$\fp$ be a prime of $K$ not dividing $d$. Write $X_{\fp}$ for the 
reduction of $X$ modulo $\fp$ and put $\FF_q :=\OO_K /\fp$ with 
$q=\mbox{\em Norm}\,{\fp}$. Then the following assertions hold. 

$(a)$ The zeta-function $Z(X_{\fp},t)$ has the form
\begin{equation*}
Z(X_{\fp}, t)=\frac{P_3(X_{\fp},t)}{(1-t)(1-qt)(1-q^2t)(1-q^3t)} 
\end{equation*}
where $P_3(X_{\fp},t)$ is a polynomial of integer coefficients 
whose reciprocal roots have the complex absolute value $q^{3/2}$ and 
$$P_3(X_{\fp},t)=\prod_{{\mathbf a} \in {\fA}_d(Q)} 
(1-j_{\fp} ({\mathbf a})t)$$ 
with ${\fA}_d (Q)={\fA}_d(w_0,w_1,w_2,w_3,w_4)$. 

$(b)$ There exists a crepant resolution $\xtil_{\fp}$ of $X_{\fp}$ 
such that $\xtil_{\fp}$ is obtained by applying a toroidal 
desingularization to $\PP^4(Q)$ and $\xtil_{\fp}$ is a Calabi--Yau 
threefold over $\FF_q$. 

$(c)$ With the notation of $(b)$ and {\rm Lemma \ref{zeta-e-div}}, 
we have 
$$Z(\xtil_{\fp},t)=\frac{P_3(\xtil_{\fp},t)}{(1-t)P_2(\xtil_{\fp},t)
P_4(\xtil_{\fp},t)(1-q^3t)}$$ 
with 
\begin{align} 
P_2(\xtil_{\fp},t) & =(1-qt) 
\prod_{\bi \in I_1} (1-qt)^{n_{\bi}} 
\prod_{\bj \in I_0} P_0(P_{\bj},qt)^{e_{\bj}} \nonumber \\ 
P_3(\xtil_{\fp},t) & =\prod_{{\mathbf a} \in {\fA}_d(Q)} 
(1-j_{\fp} ({\mathbf a})t) 
\prod_{\bi \in I_1} P_1(C_{\bi},qt)^{n_{\bi}} \nonumber \\ 
P_4(\xtil_{\fp},t) & =P_2(\xtil_{\fp},qt) \nonumber 
\end{align} 
\end{thm}

\begin{proof} 
(a) Since $X_{\fp}$ is a weighted diagonal threefold over $\FF_q$ and 
$q\equiv 1\pmod{d}$, $Z(X_{\fp},t)$ can be computed by Weil's classical 
method (see \cite{Go1} for details). 

(b) The existence of $\xtil_{\fp}$ is proven in Lemma \ref{zeta-e-div} 
and its Calabi--Yau property follows from $d=w_0+\cdots +w_4$. 

(c) The assertions are deduced from Lemma \ref{zeta-e-div} by noting 
that 
\begin{align} 
P_2(\xtil_{\fp},t) & =P_2(X_{\fp},t)
\prod_{\bi \in I_1} P_2(\widetilde{C}_{\bi}^{\circ}, t) 
\prod_{\bj \in I_0} P_2(\widetilde{P}_{\bj}^{\circ}, t) \nonumber \\ 
P_3(\xtil_{\fp},t) & =P_3(X_{\fp},t)
\prod_{\bi \in I_1} P_3(\widetilde{C}_{\bi}^{\circ}, t) 
\prod_{\bj \in I_0} P_3(\widetilde{P}_{\bj}^{\circ}, t)  \nonumber 
\end{align} 
and $P_3(\widetilde{P}_{\bj}^{\circ}, t)=1$.  
\end{proof} 

\begin{rem} 
As we noted in Remark \ref{zeta-over-Fp}, the assumption $K\supset \QQ (\zeta_d )$ 
enables us to have $q\equiv 1\pmod{d}$ for every prime $\fp$. If we relax 
this assumption, then we can still calculate the number of $\OO_K/\fp$-rational 
points on $X_{\fp}$ via the threefold  
$$z_0^{\gcd (m_0,q-1)}+z_1^{\gcd (m_1,q-1)}+z_2^{\gcd (m_2,q-1)}+
z_3^{\gcd (m_3,q-1)}+z_4^{\gcd (m_4,q-1)}=0.$$ 
Since the degree, $d^{'}$, of this threefold satisfies $q\equiv 1\pmod{d^{'}}$, 
we can express the number of rational points in terms of Jacobi sums. We 
carry out this calculation over sufficiently many finite extensions of 
$\FF_q$ to obtain $Z(X_{\fp},t)$ with $q\not\equiv 1\pmod{d}$. Further, 
this algorithm works also over $\FF_p$ and one can determine the $L$-series 
of $X$ over $\QQ$ (cf. Lemma \ref{lem7-2e}). More details of it will be 
discussed in a subsequent paper. 
\end{rem} 

Next we discuss the quasi-diagonal threefolds. The description is almost 
identical with that of the diagonal case. 

\begin{thm}\label{zeta-qd3} Let $X$ be a quasi-diagonal threefold over 
$\FF_q$ defined by an equation 
$$z_0^{m_0}z_1+z_1^{m_1}+z_2^{m_2}+z_3^{m_3}+z_4^{m_4}=0$$
in $\PP^4(k_0,k_1,k_2,k_3,k_4)$ of degree $d$, where $d=k_0m_0+k_1=k_1m_1
=k_2m_2=k_3m_3=k_4m_4$. Let $C$ be a weighted diagonal curve of degree $d$ 
in $\PP^2(k_2,k_3,k_4)$ defined by 
$$z_2^{m_2}+z_3^{m_3}+z_4^{m_4}=0.$$
Write 
$$\begin{array}{l}   
M=\mbox{\em lcm}\, (m_0,m_2,m_3,m_4)  \\ 
M_i=M/m_i \qquad (i=0,2,3,4).    
\end{array}$$   
Let $Z(C,t)=P_1(C,t)/(1-t)(1-qt)$ denote the zeta-function of $C$ over $\FF_q$. 
Assume $q\equiv 1 \pmod{M}$. Then the zeta-function of $X$ has the form   
$$Z(X,t)=\frac{P_3(X,t)}{(1-t)(1-qt)(1-q^2t)(1-q^3t)}$$
where
$$P_3(X,t)=P_1(C,qt)\prod_{{\mathbf a} \in {\fA}_X} (1-j_{\fp} ({\mathbf a})t)$$ 
$$\fA_X =\biggl\{ {\mathbf a} =(a_0,\cdots ,a_4) \Big| 
\begin{array}{l}  
a_1\in \ZZ /M\ZZ ,\ a_i\in M_i\ZZ /M\ZZ \ (i=0,2,3,4), \\
a_i\neq 0 \ (0\leq i\leq 4),\ \sum_{i=0}^{4}a_i=0\ \text{\em 
and } \ a_0+m_1a_1=0
\end{array}
\biggr\}$$ 
\end{thm} 

\begin{rem} 
The main differences of zeta-functions between diagonal and quasi-diagonal 
threefolds are $P_3(X,t)$ and the modulus for $\fA$. Since $M$ is often 
smaller than $d$, the factor of $P_3(X,t)$ associated with the weight 
motive can be smaller than that of a diagonal threefold of the same degree. 
Here the weight motive means the $(\ZZ /M\ZZ )^{\times}$-orbit of ${\mathbf a}=
(M_0,M_1,M_2,M_3,M_4)$ in ${\fA}_X$ with $M_1:=M-M_0-M_2-M_3-M_4$. We will 
discuss more details about weight motives elsewhere (cf. \cite{Yui06}). 
\end{rem} 

\begin{proof} Let $\overline{X}$ be the affine quasi-cone of $X$ in $\Aff^5$. 
Write $k_{\nu}$ for the extension of $k:=\FF_q$ of degree $\nu$. Denote by 
$N_{\nu} (X)$ for the number of $k_{\nu}$-rational points on $X$.  
By Corollary 1.4 of \cite{Go1}, we have $N_{\nu} (\overline{X} )=1+(q^{\nu} -1) 
N_{\nu} (X)$. Let $V$ be the closed subset of $X$ defined by $z_1=0$; i.e. 
$$V:\ \ z_2^{m_2}+z_3^{m_3}+z_4^{m_4}=0 \ \mbox{ with $z_0$ being free } \ 
\subset \Aff^4 .$$  
Write $U$ for the open subset $X\setminus V$:   
\begin{equation} \label{u-qd3} 
U:\ \ z_0^{m_0}z_1+z_1^{m_1}+z_2^{m_2}+z_3^{m_3}+z_4^{m_4}=0 \ 
\mbox{ with $z_1\neq 0$ } \ \subset \Aff^5 .  
\end{equation}  
Then $X_{k_{\nu}}=V_{k_{\nu}}\cup U_{k_{\nu}}$ (disjoint union) for every 
$\nu \geq 1$. Regarding $z_1$ as a constant in $k_{\nu}^{\times}$, write 
$U(z_1)$ for the affine hypersurface in $\Aff^4_{k_{\nu}}$ defined by the 
equation (\ref{u-qd3}). 
Then for $\nu \geq 1$, 
$$N_{\nu} (\overline{X} )=N_{\nu} (V)+\sum_{z_1\in k_{\nu}^{\times}} 
N_{\nu} (U(z_1)).$$  
Write $P_1(C,t)=\prod_i (1-\alpha_i t)$. It follows from Weil's classical
results that    
$$N_{\nu} (V)=q^{\nu} \bigl\{ q^{2\nu} -(q^{\nu} -1)\sum_i \alpha_i^{\nu}
\bigr\}.$$  
Applying Lemma \ref{weil} to the case $r=4,\ b_0=z_1^{m_1},\ b_1=z_1$ and 
$b_2=b_3=b_4=1$ (and also interchanging $a_0$ and $a_1$), we find 
$$N_{\nu} (U(z_0))=q^{3\nu}-\sum_{\ba \in \fA_0} \chi_{\nu}^{-1}
(z_1^{m_1a_1+a_0}) j_{\nu} (\ba )$$   
where $\chi_{\nu}$ is a non-trivial character of $k_{\nu}^{\times}$ defined by 
$\chi_{\nu} =\chi \circ N_{k_{\nu} /k}$, $j_{\nu}$ is the sum 
$$j_{\nu} (\ba )=-\sum_{\substack{v_i\in k_{\nu}^{\times} \, (1\leq i\leq 4) \\ 
1+v_1+ \cdots +v_4=0}} {\chi_{\nu}}^{a_1}(v_1)\cdots {\chi_{\nu}}^{a_4}(v_4)$$ 
and 
$$\fA_0 =\biggl\{ \ba =(a_0 ,\cdots ,a_4 ) \Big| \ 
\begin{array}{l} 
a_1 \in \ZZ /M\ZZ,\ a_i \in M_i\ZZ /M\ZZ \ (i=0,2,3,4), \\ 
a_i \neq 0\ (i=0,2,3,4),\ \sum_{i=0}^4 a_i =0   
\end{array} \biggr\}.$$ 
Noting that 
$$\sum_{z_1\in k_{\nu}^{\times}} \chi_{\nu}^{-1} (z_1^{a_0+m_1a_1})= 
\begin{cases} q^{\nu} -1 & \mbox{ if } a_0+m_1a_1\equiv 0\pmod{M} \\ 
0 & \mbox { otherwise } \end{cases}$$   
we obtain  
\begin{align*} 
\sum_{z_1\in k_{\nu}^{\times}} N_{\nu}(U(z_1)) & =q^{3\nu}(q^{\nu} -1)- 
\sum_{z_1\in k_{\nu}^{\times}} \sum_{\ba \in \fA_0} \chi_{\nu}^{-1} 
(z_1^{a_0+m_1a_1})j_{\nu}(\ba ) \\  
 &= q^{3\nu}(q^{\nu} -1)-\sum_{\ba \in \fA_0} j_{\nu} (\ba )  
\sum_{z_1\in k_{\nu}^{\times}} \chi_{\nu}^{-1} (z_1^{a_0+m_1a_1}) \\  
 &=q^{3\nu}(q^{\nu} -1)-(q^{\nu} -1)\sum_{\substack{\ba \in \fA_0 \\ 
a_0+m_1a_1=0}} j_{\nu} (\ba ) \\ 
 &=q^{3\nu}(q^{\nu} -1)-(q^{\nu} -1)\sum_{\ba \in \fA_X} j_{\nu} (\ba ).   
\end{align*} 
($a_0\neq 0$ implies $a_1\neq 0$.) It holds now that $j_{\nu} (\ba ) = 
j_{\fp} (\ba )^{\nu}$. Combining this with $N_{\nu} (V)$, we conclude  
$$N_{\nu} (\overline{X})=q^{4\nu}-q^{\nu} (q^{\nu} -1)\sum_i 
\alpha_i^{\nu} -(q^{\nu} -1) \sum_{\ba \in \fA_X} j_{\fp} (\ba )^{\nu}.$$    
Therefore it follows from $N_{\nu} (\overline{X})=1+(q^{\nu} -1)N_{\nu}(X)$ 
that  
$$N_{\nu}(X)=1+q^{\nu} +q^{2\nu}+q^{3\nu} -\sum_i (q\alpha_i )^{\nu}  
-\sum_{\ba \in \fA_X} j_{\fp} (\ba )^{\nu}.$$  
This gives rise to the zeta-function of $X_k$. 
\end{proof}   

\begin{rem} 
In Proposition \ref{zeta-qd3}, the weight $(k_2,k_3,k_4)$ of the 
2-space $\PP^2(k_2,k_3,k_4)$ may not be normalized. But the 
computation of the zeta-function of $C$ can be carried out in the 
same way as for a normalized weight. The polynomial $P_1(C,t)$ of 
$C$ then has the form  
$$P_1(C,t)=\prod_{{\mathbf a} \in \fA_C} (1-j({\mathbf a})t)$$ 
with $\fA_C =\{ {\mathbf a} =(a_2,a_3,a_4)\ \mid \ a_i\in 
k_i\ZZ /d\ZZ,\ a_i\neq 0 \ (i=2,3,4),\ \sum_{i=2}^{4}a_i=0\}$.  
\end{rem} 

As in the diagonal case, $\mbox{Sing}\, (X)$ for a quasi-diagonal 
threefold is a scheme over $k$ and now $P_0 :=(1:0:0:0:0)$ becomes a 
possible singularity. We have 
$$\mbox{Sing}\, (X) =\bigcup_{\bi \in I_1} C_{\bi} \cup 
\bigcup_{\bj \in I_0 \cup \{0\}} P_{\bj}.$$ 

\begin{thm}\label{zeta-qd3-smooth} 
Let $K$ be a number field. Let $X$ be a singular Calabi--Yau threefold 
over $K$ of quasi-diagonal type constructed in {\em Section $4$} by the 
equation 
$$z_0^{m_0}z_1+z_1^{m_1}+z_2^{m_2}+z_3^{m_3}+z_4^{m_4}=0$$ 
in $\PP^4(w_0,w_1,w_2,w_3,w_4)$ of degree $d=w_0+\cdots +w_4$. Assume 
that $K$ contains the $M$-th cyclotomic field $\QQ (\zeta_M )$, 
where $M$ is the integer defined in {\rm Theorem \ref{zeta-qd3}}. 
Let $\fp$ be a prime of $K$ not dividing $M$. Write $X_{\fp}$ for 
the reduction of $X$ modulo $\fp$ and put $\FF_q :=\OO_K /\fp$ with 
$q=\mbox{\em Norm}\,{\fp}$. Then the following assertions hold. 

$(a)$ The zeta-function $Z(X_{\fp},t)$ has the form
\begin{equation*}
Z(X_{\fp}, t)=\frac{P_3(X_{\fp},t)}{(1-t)(1-qt)(1-q^2t)(1-q^3t)} 
\end{equation*}
where $P_3(X_{\fp},t)$ is a polynomial of integer coefficients 
whose reciprocal roots have the complex absolute value $q^{3/2}$ and 
$$P_3(X_{\fp},t)=P_1(C,qt)\prod_{{\mathbf a} \in {\fA}_X} 
(1-j_{\fp} ({\mathbf a})t)$$ 
with $C$ and ${\fA}_X$ as being defined in {\rm Theorem \ref{thm7-5}}. 

$(b)$ There exists a crepant resolution $\xtil_{\fp}$ of $X_{\fp}$ 
such that $\xtil_{\fp}$ is obtained by applying a toroidal 
desingularization to $\PP^4(Q)$ and $\xtil_{\fp}$ is a Calabi--Yau 
threefold over $\FF_q$. 

$(c)$ With the notation of $(b)$ and {\rm Lemma \ref{zeta-e-div}}, 
we have 
$$Z(\xtil_{\fp},t)=\frac{P_3(\xtil_{\fp},t)}{(1-t)P_2(\xtil_{\fp},t)
P_4(\xtil_{\fp},t)(1-q^3t)}$$ 
with 
\begin{align} 
P_2(\xtil_{\fp},t) & =(1-qt) 
\prod_{\bi \in I_1} (1-qt)^{n_{\bi}} 
\prod_{\bj \in I_0 \cup \{0\}} P_0(P_{\bj},qt)^{e_{\bj}} \nonumber \\ 
P_3(\xtil_{\fp},t) & =P_1(C,qt)\prod_{{\mathbf a} \in {\fA}_X} 
(1-j_{\fp} ({\mathbf a})t)
\prod_{\bi \in I_1} P_1(C_{\bi},qt)^{n_{\bi}} \nonumber \\ 
P_4(\xtil_{\fp},t) & =P_2(\xtil_{\fp},qt) \nonumber 
\end{align} 
\end{thm}

\begin{proof} 
The assertions can be proven in the same way as in Theorem 
\ref{thm7-5} by noting the contributions of a curve $C$ and a 
singularity $P_0$ to the zeta-function of $\xtil_{\fp}$. 
\end{proof}

\section{Deformations of Calabi-Yau threefolds \& zeta-function}\label{sect12}
 
In this section we recall some results on the variation of the zeta 
function of a variety. 

In this context we need to use a different cohomology theory than in the 
previous sections, namely rigid 
cohomology. We will not define rigid cohomology in complete detail, but give a
simplified presentation, 
which works for quasi-smooth hypersurfaces. For a good introduction to the
theory of rigid cohomology 
we refer to \cite{BerSMF} and \cite{BerFin}. 

\begin{defn} 
Let $q=p^n$, with $p$ a prime number. Let $\QQ_q$ (resp. $\ZZ_q$) be the 
unique unramified extension of $\QQ_p$ (resp. $\ZZ_p$ of degree $n$. 
Denote by $\pi$ the maximal ideal of $\ZZ_q$. 

Let $\PP_{\FF_q}$ be a shorthand for $\PP^{n+m-1}_{\FF_q}(v_0 
\mathbf{w},w_0\mathbf{v})$. 
\end{defn} 

Let $\overline{X_{\lambda,\mu}}\subset \PP_{\FF_q}$ be a family of 
hypersurfaces, say given by $\overline{F_{\lambda,\mu}}=0$, such that 
the general element is quasi-smooth. Let $\overline{U_\lambda,\mu}$ be 
the complement $\PP_{\FF_q}\setminus \overline{X_{\lambda,\mu}}$. 
Since
\[
Z(\overline{U_{\lambda,\mu}},t)Z(\overline{X_{\lambda,\mu}},t)=Z(\PP_{\FF_q},
t)=\frac{1}{(1-t)(1-qt) 
\dots (1-q^{n+m-1}t)} \]
it suffices to determine the zeta-function of $\overline{U_{\lambda,\mu}}$, 
if one wants to know that the zeta-function of $\overline{X_{\lambda,\mu}}$.

Choose now a weighted homogenous polynomial $F_{\lambda,\mu} \in
\ZZ_q[\lambda,X_0,\dots, X_n]$ 
such that $F_{\lambda,\mu} \equiv \overline{F_{\lambda,\mu}} \bmod \pi$. 
Let $X_{\lambda,\mu} \subset \PP_{\QQ_q}$ be the zero set of
$F_{\lambda,\mu}$, let $U_
{\lambda,\mu}$ be $\PP_{\QQ_q}\setminus X_{\lambda,\mu}$. Since
$U_{\lambda,\mu}$ is affine, we 
can write $U_{\lambda,\mu}=\spec R_{\lambda,\mu}$, with 
\[ R_{\lambda,\mu}=\QQ_q [\lambda,\mu,Y_0,\dots,
Y_m]/(G_{1,\lambda,\mu},\dots,G_{k,\lambda,\mu}).\]

For $\lambda_0,\mu_0$ in the $p$-adic unit disc, set
\[ R^{\dagger}_{\lambda_0,\mu_0} = \frac{\{ H \in \QQ_q[[Y_0,\dots,Y_m]] \colon
\mbox{ the radius of 
convergence of } H \mbox{ is at least } r>1
\}}{(G_{1,\lambda_0,\mu_0},\dots,G_{k,\lambda_0,\mu_0})}.\]
 Then $R^{\dagger}_{\lambda_0,\mu_0}$ is called the {\em overconvergent
completion} (or weakly completion) of $R_{\lambda_0,\mu_0}$. 

\begin{defn} 
Let $\iota: \PP^n_{\QQ_q} \to \PP_{\QQ_q}$ be the natural quotient map. Let 
$G:=\times {\bmu}_{w_i}/
\Delta$ be the group associated with this quotient. Set $\tilde{R}^\dagger$ to
be the overconvergent 
completion of the coordinate ring of $\PP^n\setminus \iota^{-1}(X)$. Then on
$\Omega^i_{\tilde{R}^\dagger}$ there is a natural $G$-action. Set
$\Omega^i_{R^{\dagger}}=(\Omega^i_{\tilde{R}^\dagger})^G$. 
The {\em  $i$-th 
Monsky-Washnitzer cohomology group $H^i(U,\QQ_q)$} is the $i$-th cohomology
group of the complex $\Omega^\bullet_{R^{\dagger}}$.
\end{defn}

\begin{defn} Let $R$ be a ring over $\ZZ_q$. Let $\pi$ be the maximal ideal of
$\ZZ_q$. 
A {\em lift of Frobenius}  is a ring homomorphism 
$\Frob_q^*: R \rightarrow R$ such that  its reduction modulo $\pi$
 \[ \Frob_q^* \bmod \pi :  R\otimes_{\ZZ_q} \FF_q  \rightarrow R\otimes_{\ZZ_q}
\FF_q \]
is well-defined and equals $x\mapsto x^q$. 
\end{defn}

Fix a lift of Frobenius $\Frob_q^*$ to $R_{\lambda}$, such that
$\Frob_q^*(\lambda)=\lambda^q$ , $
\Frob_q^*(\mu)=\mu^q$ and $\Frob_q^*$ maps $R_{\lambda_0^q}^{\dagger}$ to
$R_{\lambda_0}^{\dagger}$. Denote by $\Frob^*_q$ also  the induced 
morphism on $H^i(U_{\lambda_0},\QQ_q)$.

{From} now on let \[F_{\lambda,\mu}(x_1,\dots,x_n,y_1,\dots,y_m)=\sum
x_i^{d_i}-\sum y_j^{e_j}+ 
\lambda\prod x_i^{a_i}-\mu \prod y_j^{b_j}.\] 

\begin{lem} Let
\[G_{\lambda,\mu}(x_1,\dots,x_n,y_1,\dots,y_n):=F_{\lambda,\mu}(x_1^{v_0w_1},
\dots,x_n^{v_0w_n},y_1^{w_0v_1},\dots,y_m^{w_0v_n}).\] Then $G_{\lambda,\mu}=0$
defines a smooth 
hypersurfaces in $\PP^{n+m-1}(1,\dots,1)$ if and only if $F_{\lambda,\mu}=0$
defines a smooth 
hypersurface in $\PP$.
\end{lem}
\begin{proof}
This is a straightforward calculation.
\end{proof}

This Lemma allows us to obtain results on $H^i(U_{\lambda_0,\mu_0},\QQ_q)$ that
are very similar to 
the results of \cite[Section 3]{mondef}. We summarize this in  the following
Theorem:

\begin{thm} \label{thmLTF} Suppose $\overline{X_{\lambda_0,\mu_0}}$ is
quasi-smooth. Then $H^i(U_{\lambda_0,\mu_0},\QQ_q)=0$ for $i\neq 0, n+m-1$ 
and
\[ q^{n+m-1}+ (-q)^{n+m-1} \trace( (\Frob_q^*)^{-1} |
H^{n+m-1}(U_{\lambda_0,\mu_0})) = \# \overline{U_
{\lambda_0,\mu_0}}(\FF_q).\]
\end{thm}

The proof is very close to the one-dimensional case (i.e., $\mu=0$). For this
reason we leave it out. See 
\cite[Section 3]{mondef} for the proof in the one-dimensional case.

For the central fiber, we can use the results of the previous Sections, namely:
\begin{prop}[{\cite{mondef}}]\label{prpDia} Let $\kv$ be an admissible monomial
type. Then
 \[\Frob_{q,0}^*\omega_\kv= d_{\kv,q} \omega_{\overline{q}\kv}, \]
 with $d_{\kv,q}\in \QQ_q$. If $d\equiv 1 \bmod q$ then
$q^{n+m-1}/d_{\kv,q}^{-1}$ is a Jacobi-sum.
 \end{prop}

We describe an approach to determine the characteristic polynomial of 
Frobenius on $H^{n+m-1}(U_{\lambda_0,\mu_0},\QQ_q)$. We can study the 
Frobenius action on $H^n(U_{\lambda,\mu},\QQ_q)$ as a 
function in $\lambda$ and $\mu$.
 
Fix a point $(\lambda_0,\mu_0)$ and an (analytic) curve
$\nu\to(\lambda(\nu),\mu(\nu))$ connecting 
$(0,0)$ with $(\lambda_0, \mu_0)$. $\mu(\nu )^q=\mu (\nu^q)$ and 
$\lambda(0)=\nu(0)=0$.
 
Following N. Katz \cite{Katz}, consider the  commutative diagram 
\[\xymatrix{ \hspace{-.5in} H^{n+m-1}(U_{\lambda(\nu)^q,\mu(\nu)^q}) 
\ar[r]^(0.6){\Frob_{q,\lambda(\nu),\mu(\nu)}^*} \ar[d]_{A(\nu^q)} & 
H^{n+m-1}(U_{\lambda(\nu),\mu(\nu)},\QQ_q) \hspace{-1.2in} \ar[d]_{A(\nu)} \\  
H^{n+m-1}(U_{0,0}) \ar[r]^{\Frob_{q,0,0}^*} & \qquad H^{n+m-1}(U_{0,0},\QQ_q) 
} \] 
where  $\Frob_{q,\lambda,\nu}$ is the Frobenius acting on the complete family.
Since it maps the fiber 
over $(0,0)$ to the fiber over $(0,0)$ this map can be restricted to
$U_{(0,0)}$.  Katz studied the 
differential equation  associated to $A(\nu)$. He remarked in a note that it is
actually the solution of the 
Picard-Fuchs equation. In \cite{Katz} only the case  $\PP=\PP^n$ is discussed.
In \cite{mondef} it is 
discussed how to deduce the general case from this particular case.

In \cite{mondef} a procedure to compute $A(\nu)$ is given. For this we need some
notation

\begin{notation} \label{notRap}Let $\PP' $ be a $k$-dimensional weighted
projective space with weight 
$(u_i)$ and coordinates $z_i$. Let $\Omega:= \prod_i z_i\sum_j (-1)^j u_j
\frac{dz_0}{z_0}\wedge \dots 
\wedge \widehat{\frac{dz_j}{z_j}} \wedge \dots \wedge \frac{dz_k}{x_k}$.
\end{notation}
Calculating in $H^{n+m-1}(U_{\lambda_0,\mu_0},\QQ_q)$ is relatively easy, due to
the following well-known observation: 
\begin{rem} \label{remRap}
Let $G$ be the defining equation for a quasi-smooth hypersurface $X$ in a
$k$-dimensional weighted 
projective space $\PP'$, let $U$ be its complement. 
The vector space $H^n(U,\QQ_q)$ is the quotient of the infinitely-dimensional
vector space spanned by
\[ \frac{H}{G^t} \Omega\]
with $\deg(H)=t\deg(G_{\lambda_0})-w$  by the relations
\[ \frac{(t-1)H G_{z_i}-GH_{z_i}}{G^t} \Omega, \]
where the subscript $z_i$ means the partial derivative with respect to  a
coordinate $z_i$ on $\PP$. 
\end{rem}

We return to our family $X_{\lambda,\mu}$.
We continue by identifying particular differential forms, such that their
classes generated $H^{n+m-1}(U_{\lambda_0,\mu_0},\QQ_q)$:

\begin{defn}  Assume that $d=\deg(X)$ is divisible by all the $u_i$.
A {\em monomial type} $\mathbf{m}=(\overline{m_0},\dots,\overline{m_n})$ is an
element of $\prod_i u_i\ZZ/d\ZZ$ such that   
$\sum \overline{m_i}=0$ in $\ZZ/d\ZZ$. Choose representatives $m_i\in \ZZ$ 
of $\overline{m_i}$ such that $0\leq m <d$. The {\em relative degree} of 
$\mathbf{m}$ is then $\sum m_i/d$.
 
A monomial type $\mathbf{k}$ is called {\em admissible} if there exist 
integers $k_i, i=0,\dots n$ such that $0\leq k_i\leq d_i-1$ and 
$\mathbf{k}:=(u_0(k_0+1),\dots,u_n(k_n+1))$. Let $t$ be the relative 
degree of $\kv$. With $\mathbf{k}$ we associate the differential form
\[ \omega_\kv:=  \frac{\prod z_i^{k_i}}{F_\lambda^t} \Omega.\]\end{defn}

\begin{prop} \label{prpBasis} Suppose $\overline{X_{\lambda_0,\mu_0}}$ is
quasi-smooth. The set
 \[  \left\{ \omega_\kv \colon \kv \mbox{ an admissible monomial type}\right\}\]
  is a basis for $H^n(U_{\lambda_0,\mu_0},\QQ_q)$.
 \end{prop}

The operator $A(\nu)$ can be calculated as follows: Let $\kv$ be an admissible
monomial type of 
relative degree $t$. Then $A(\nu)\omega_{\kv}$ is the reduction of
\begin{eqnarray*} \omega_{\kv} &= &\frac{\prod x_i^{k_i} \prod
y_j^{m_j}\Omega}{(\sum x_i^{d_i}-\sum 
y_j^{e_j} +\lambda(\nu) F_1 - \mu(\nu) F_2 )^t} \\&=& \sum_{j=0}^\infty \left(
\begin{matrix} t+j-1 \\ j \end
{matrix} \right) \frac{  \prod x_i^{k_i} \prod y_j^{m_j} (-\lambda(\nu) F_1 +
\mu(\nu) F_2)^j}{(\sum x_i^{d_i}-
\sum y_j^{e_j})^{t+j} } \Omega\end{eqnarray*}
in $H^n(U_{0,0})$, provided that $|\nu|$ is small enough. For $\nu$ with larger
norm, we can use 
analytic continuation to obtain $A(\nu)$. 
 
The operator $A(\nu)$ depends on the chosen path, but its value at any point
$(\lambda_0,\mu_0)$  is 
independent of the path. Hence it makes sense to describe $A$ in terms of
$(\lambda,\mu)$. 

\section{Calculating of the deformation matrix}\label{sect12a}
In this section we calculate $A(\lambda,\mu)\omega_{\nv}$ for an admissible
monomial type $\nv$.
Let 
$d'_i$ be  the order  of $a_i \bmod d_i$ in $ \ZZ/d_i\ZZ$. Let $d'$ be least
common multiple of all the 
$d_i'$.
 Set $b_i = f_i d'/d_i$. 
 Let 
$e'_j$ be  the order  of $b_j \bmod e_j$ in $ \ZZ/e_j\ZZ$. Let $e'$ be least
common multiple of all the 
$e_j'$.
 Set $g_j = b_j e'/e_i$.

 In the following proposition and its proof we identify elements in $a \in
\ZZ/m\ZZ$ with their 
representative $\tilde{a}\in \ZZ$ such that $0\leq \tilde{a} \leq m-1$. Denote
by $(t)_m$ the Pochhammer 
symbol $t(t+1)\dots(t+m-1)$.

\begin{prop}\label{PrpCoeffb}Let $\kv$ be an admissible monomial type. Let $t$
be the relative degree 
of $\ v$. Write  $A(\lambda,\mu)\omega_\kv =\sum c_\mv(\lambda,\mu) \omega_\mv$,
where  the sum is 
taken over all admissible monomial types. Then $c_\mv(\lambda,\mu)$  is non-zero
only if   
there exist $r_0,s_0\in \ZZ$  with $0\leq r_0\leq d'-1$, $0\leq s_0\leq e'-1$
and such that $\mathbf{m}-
\mathbf{k}=\overline{r_0} \mathbf{a}+\overline{s_0}\mathbf{b}$. If this is the
case then $c_{\mv}(\lambda,
\mu)\omega_{\mv}$ is the product of
\[ \left( \begin{matrix} t+r_0-1 \\ r_0 \end{matrix}\right) (-\lambda)^{r_0} 
{}_{d'}F_{d'-1}\left( \begin{array}{c} 
\alpha_{i,s} \\ \frac{r_0+1}{d'} \; \frac{r_0+2}{d'} \; \dots \; \widehat{1} \;
\dots \;\frac{r_0+d'}{d'} \end{array}; 
\prod_{i\colon a_i\neq 0} \left(\frac{  a_i}{d_i}\right)^{f_i} (-\lambda)^{d'}  
 \right),\]
and
\[ \left( \begin{matrix} t+s_0-1 \\ s_0 \end{matrix}\right) (-\mu)^{s_0} 
{}_{e'}F_{e'-1}\left( \begin{array}{c} 
\beta_{j,s} \\ \frac{s_0+1}{e'} \; \frac{s_0+1}{e'} \; \dots \; \widehat{1} \;
\dots \;\frac{s_0+e'}{e'} \end{array}; 
\prod_{i\colon b_i\neq 0} \left(\frac{  b_i}{e_i}\right)^{g_i} (-\mu)^{e'}   
\right),\]
with
\[ {\red \frac{\prod x_i^{a_i r_0+k_i}  \prod
y_j^{b_js_0+m_j}}{F_{0,0}^{t+r_0+s_0}} \Omega},\]
where
\[ \alpha_{i,s} = \frac{(s-1)d_i+1+a_i r_0+k_i}{a_i d'}, s=1,\dots,f_i;
i=1,\dots n\]
and
\[ \beta_{j,s} = \frac{(s-1)e_i+1+b_i s_0+m_i}{b_i e'}, s=1,\dots,g_i; i=1,\dots
m.\]Where $\red$ means 
reducing in the cohomology group $H^n(U_{0,0},\QQ_q)$. 

\end{prop}

\begin{proof} This proof is very similar to the proof of \cite[Proposition
5.2]{mondef}. For notational 
convenience we calculate  the deformation matrix for the family $\sum
x_i^{d_i}+\sum y_i^{e_i}+\lambda 
\prod x_i^{a_i}+\mu \prod y_j^{b_j}$. At the end of the proof we show that our
original family has the 
same deformation matrix.

One starts by expanding
\[ \frac{\prod x_i^{k_i} \prod y_j^{m_j}}{(\sum x_i^{d_i}+\sum y_i^{e_i}+\lambda
\prod x_i^{a_i}+\mu \prod 
y_j^{b_j})^t}\]
as follows
\[\sum_{r_1=0}^{d'-1}\sum_{s_1=0}^{e'-1} \sum_{j_1,j_2\geq0} 
B_1B_2 \frac{ \prod x_i^{k_i+(r_1+j_1d') a_i} \prod y_j^{m_j+(s_1+j_2e')b_j}
}{F^{t+r_1+s_1+j_1d'+j_2e'}},\]
with $F=\sum x_i^{d_i}+\sum y_j^{e_j}$,
\[B_1=\left( \begin{matrix} t+r_1+s_1+j_1d'+j_2e'-1 \\ t-1 \end{matrix}\right)
\mbox{ and } B_2=\left( \begin
{matrix} r_1+s_1+j_1d'+j_2e' \\ r_1+j_1d' \end{matrix}\right).\]
In each of the reduction steps the exponent of one of the $x_i$ (resp. $y_j$)
will be lowered by $d_i$ 
(resp. $e_j$.) This implies that we have to distinguish between exponents that
are different modulo one of the $d_i$ or $e_j$ and hence we only need to 
consider the summand for $r_1=r_0, s_1=s_0$.
 
Write the reduction of this summand as  $c_0 (-\lambda)^{k_0}(-\mu)^{m_0} 
\sum c_{j_1,j_2} (-\lambda)^{d'j_1}(-\mu)^{e'j_2}$. 
A calculating similar to the one done in \cite[Proposition 5.2]{mondef} shows
that
$c_{j_1+1,j_2}/c_{j_1,j_2}$ is an element of  $\QQ(j_1)$ and
$c_{j_1,j_2+1}/c_{j_1,j_2}$ is an element 
of $\QQ(j_2)$, i.e., they are rational functions in $j_1$, resp. $j_2$. This
implies that 
\[\begin{matrix}c_0 (-\lambda)^{k_0}(-\mu)^{m_0} \sum c_{j_1,j_2}
(-\lambda)^{d'j_1}(-\mu)^{e'j_2}=\\= 
c_0 (-\lambda)^{k_0}(-\mu)^{m_0}  \left(\sum c^{(1)}_{j_1}
(-\lambda)^{d'j_1}\right) \left(\sum c^{(2)}_{j_2} 
(-\mu)^{e'j_2}\right),\end{matrix}\]
where $c^{(t)}_{j_t+1} /c^{(t)}_{j_t+1} \in \QQ(j_t)$ for $t=1,2$, i.e., this
sum is a product of two 
hypergeometric functions.

The actual calculation of the parameters of these hypergeometric function is
similar to the calculation in 
\cite[Proposition 5.2]{mondef} and we leave this to the reader.

It remains to show that the family $\sum x_i^{d_i}+\sum y_i^{e_i}+\lambda \prod
x_i^{a_i}+\mu \prod y_j^{b_j}$ and  our original family have the same 
deformation matrix.

Let $\zeta$ be a primitive $2e$-th root of unity. Then the map $\varphi$ mapping
$x_i$ to $x_i$ and $y_j
$ to $\zeta^{v_j}y_j$, defines an isomorphism between $X_{\lambda,\mu}$ and
$X'_{\lambda,\mu}$, 
defined over $\QQ_q(\zeta)$. Since $q\equiv 1 \bmod e$ we obtain that
$[\QQ_q(\zeta):\QQ_q]\leq 2$, 
and equality holds if only if $q\not \equiv 1 \bmod 2e$. If $q\equiv 1 \bmod 2e$
then there is nothing to  
prove. So assume that $q\not \equiv 1 \bmod 2e$. Then 
\begin{eqnarray*}\Frob_q \omega_{\nv} &= &\varphi^{-1} \Frob_q' 
\varphi(\omega_{\nv})=\varphi^{-1} 
\Frob_q' (\zeta^{\sum (m_j+1)w_j}\omega_{\nv}\\&= &
\varphi^{-1} \zeta^{q\sum (m_j+1)w_j}  \Frob_q' \omega_{\nv} = \zeta^{(q-1)\sum
(m_j+1)w_j} 
\Frob_q'\omega_{\nv}\\ &=& (-1)^{\sum (m_j+1)w_j}\omega_{\nv}
=(-1)^{e-w_0(m_0+1)}\omega_{\nv} .
\end{eqnarray*}
where $m_0$ is chosen in such a way that $\mv$ is a monomial type for $V_2$. 
Decompose $H^{n+m-1}(U_{\lambda, \mu},\QQ_q)=H_1(\lambda,\mu)\oplus
H_2(\lambda,\mu)$, where 
$H_1$ is spanned by the $\omega_{\nv}$ with odd $m_0$ and $H_2$ is spanned by
the $\omega_{\nv}$ with even $m_0$. 

Since this $m_0$ is the same for all monomial types of the form $\nu+\lambda
\mathbf{a}+\mu \mathbf{b}$, we can write $A=A_1\oplus A_2$ such that $A_i$ 
is defined on $H_i$. Let $A'_i$ be a similar 
decomposition of the deformation matrix $A'$ of $X'_{\lambda,\mu}$.

Since $\omega_{\nv}$ is an eigenvector of $\Frob_{q,0,0}^{'*}$ (by
Proposition~\ref{prpDia}) we obtain 
that $\Frob_{q,0,0}^{'*}$  respects the decomposition $H_1(0,0) \oplus
H_2(0,0)$. 

Collecting every thing we have on $H_1$ that
\begin{eqnarray*}A_1(\lambda^q,\mu^q) \Frob_{q,0,0}
A_1(\lambda,\mu)^{-1}&=&\Frob_{\lambda,\mu} 
w_\nv=-\Frob'_{\lambda,\mu} w_\nv\\&=& A_1'(\lambda^q,\mu^q) (-\Frob'_{q,0,0})
A_1'(\lambda,\mu)^{-1}\\
&=&A_1'(\lambda^q,\mu^q) \Frob_{q,0,0} A_1'(\lambda,\mu)^{-1}.\end{eqnarray*}
An easy calculation shows that  $A_1'=A_1$, and a similar approach shows that
$A_2'=A_2$, whence 
$A'=A$.
\end{proof}
Note that this Proposition gives almost a complete reduction of the form
$\Frob_{q,\lambda,\mu}^* (\omega_\kv)$, in the sense 
that $c_{\mv}(\lambda,\mu)$ is described as the product of a 
hypergeometric function and the reduction of a rational function in the $x_i$
multiplied by $\Omega$. 
This last form can be easily reduced using the following Lemma:

\begin{lem}\label{redformLem} Fix non-negative integers $b_i$ such that $\sum
b_iw_i+w=td$ for some 
integer $t$.
Then the cohomology class $\omega:=\frac{\prod x_i^{b_i}}{F^t}\Omega \in
H^n(U_0,\QQ_q)$ equals
\[ \frac{\prod_i ((c_i+1)w_i/d)_{q_i}} {(s)_{t-s}} \frac{\prod
x_i^{c_i}}{F^{s}}\Omega\in H^n(U_0,\QQ_q),\]
where $0\leq c_i< d_i$  and $q_i,s$ are integers such that $b_i= q_i d_i+c_i$,
and $sd=\sum c_iw_i+w$. Moreover, if for one of the $i$ we have $c_i=d_i-1$, 
then $\omega$ reduces to zero in cohomology. 
\end{lem}

These three results enable us, using Katz' result, to give a complete
description of $\Frob_{\lambda,\mu}^*$.

Suppose now that $X_{\lambda,\mu}$ is a family obtained by the twist
construction. I.e., $X_{\lambda,
\mu}\stackrel{\sim}{\dashrightarrow} (V_{1,\lambda}\times V_{2,\mu})/\bmu_\ell$.
 We want to relate the 
zeta-function of $X_{\lambda,\mu}$ with the zeta function of $V_{1,\lambda}$ 
and $V_{2,\mu}$. This can 
be done purely geometrically, in the sense that one can factor the birational
map $X_{\lambda,\mu}
\stackrel{\sim}{\dashrightarrow} (V_{1,\lambda}\times V_{2,\mu})/\bmu_\ell$ in
proper modifications and 
then determine the zeta-function of the center of each of the proper
modification and the zeta function of 
each of the exceptional divisors. However, there is a more straight-forward
procedure that is more 
combinatorical.

We start by giving a technical definition concerning monomial types:
\begin{defn}
Let $\kv$ be an admissible monomial type for $V_1$ and let $\mv$  be an
admissible monomial type for 
$V_2$. We call $\kv,\mv$ is an \emph{admissible couple} if  $k_0+m_0+2\equiv 0
\bmod \ell$.

For an admissible couple $\kv,\mv$ define 
\[\kv \oplus  \mv := (v_0w_1(k_1+1),\dots,v_0w_n(k_n+1),w_0v_1(m_1+1),\dots,
w_0v_n(m_m+1)).\]
Then $\kv\oplus \mv$ is an admissible monomial type for $X_{\lambda,\mu}$.
\end{defn}

\begin{rem}
Let $\sigma$ be a generator of $\bmu_\ell$. Then
$\sigma^{*}\omega_{\kv}=\zeta^{k_0+1}\omega_{\kv}$ 
and $\sigma^*\omega_{\mv}=\zeta^{m_0+1}\omega_{\mv}$.

This implies that the set 
\[\{(\omega_{\kv},\omega_{\mv})\mid \kv,\mv \mbox{ an admissible couple}\}\]
is a basis for \[\oplus_i (H^{n+1}(U_{1,\lambda})^{\sigma^*-\zeta^i} \times
H^{m+1}(U_{2,\mu})^
{\sigma^*-\zeta^{\ell-i}}).\]
\end{rem}

The cohomology groups $H^{n-1}(V_{1,\lambda})_{\prim} \cong H^n(U_{1,\lambda})$
and $H^{m-1}(V_{2,\mu})_{\prim} \cong H^m(U_{2,\mu})$ can be studied using 
the results of \cite{mondef}. Using the 
birational map $(V_{1,\lambda} \times V_{2,\lambda})\bmu_\ell \dashrightarrow
X_{\lambda,\mu}$ we 
can relate a subspace of $H^n(U_{1,\lambda})\times H^m(U_{2,\mu})$ with a
subspace of $H^{n+m-2}
(X_{\lambda,\mu})$:
\begin{prop} We have the following commutative diagram
\[ \xymatrix{ H^{n+m-2}((V_{1,\lambda}\times V_{2,\mu})/\bmu_\ell)_{\prim}
\ar[r] & H^{n+m-2}(X_{\lambda,\mu})_{\prim}     \\
\oplus_i H^{n}(U_{1,\lambda})_{{\zeta^i}} (1)\times
H^{m}(U_{2,\mu})_{\zeta^{\ell-i}}(1) \ar[u] {\ar[r]^(0.62){\psi}} 
& H^{n+m-1}(U_{\lambda,\mu}) (1)\ar[u] 
}\]
Here one should consider the cohomology groups in the upper row as rigid
cohomology groups. The 
subscript $\zeta^i$ indicates that we take $\zeta^i$-eigenspace of $\sigma^*$.

The two vertical arrows are residue maps (i.e., isomorphisms), the upper
horizontal arrow is induced by 
the birational map $X_{\lambda,\mu} \dashrightarrow (V_{1,\lambda}\times
V_{2,\mu})/\bmu_{\ell}$. The 
lower horizontal arrow is the (unique) map making this diagram commutative and
is given by
\[ \psi(\omega_{\kv} , \omega_{\mv}) \mapsto \omega_{\kv \oplus  \mv},\]
where $\kv,\mv$ is an admissible couple of monomial types. In particular, $\psi$
is injective.\end{prop}

Of course, one should describe the image of $\psi$: 
\begin{lem} Let $\nv=(v_0w_1(k_1+1),\dots,v_0w_n(k_n+1),w_0v_1(m_1+1),\dots,
w_0v_n(m_m+1))$ 
be an admissible monomial type for $X_{\lambda,\mu}$.
Let $k_0,m_0$ be the unique elements of $\ZZ/\ell\ZZ$ such that 
$\kv:=(w_0(k_0+1),\dots,w_n(k_n+1))$ 
and  $\mv:=(v_0(m_0+1),\dots.,v_m(m_m+1))$ are (not necessary admissible)
monomial types for $V_1$ and $V_2$. 

Then $\omega_{\nv}$ is in the image of $\psi$ if and only if $\kv$ and $\mv$ 
are admissible monomial types, or 
if and only if $k_0\not \equiv -1 \bmod \ell$. 
\end{lem}
\begin{proof}

The first `if and only if' follows directly from the equality 
$\omega_{\mathbf n}=\omega_{\kv \oplus \mv}$.

For the second `if and only if' observe that since $\nv$ is admissible we have
$k_i\not \equiv -1 \bmod d_i$ for $i>0$ and $m_j\not \equiv -1 \bmod e_j$ 
for $j>0$.  So $\kv$ is admissible if and only if $k_0\not 
\equiv -1\mod \ell$ and $\mv$ is admissible if and only if $m_0\not \equiv -1
\bmod \ell$. 

An easy calculation shows that $k_0+m_0\equiv -2 \bmod \ell$. This implies that
if one of $k_0,m_0$ is 
$-1$ modulo $\ell$ then so is the other and that if one of $\kv,\mv$ is
admissible so is the other. This 
finishes the proof. 
\end{proof}

We summarize the results of this section 
\begin{thm} Let $\overline{X_{\lambda,\mu}}$ be a 2-parameter family obtained by
applying the twist 
construction to two families $V_{1,\lambda}$ and $V_{2,\mu}$, that are both
monomial deformation of 
diagonal hypersurfaces. Then $Z(\overline{X_{\lambda_0,\mu_0}},t)$ is the
characteristic polynomial of
\[ \lim_{(\lambda,\mu) \to (\lambda_0,\mu_0) } A(\lambda,\mu)^{-1} \Frob_{0,0}^*
A(\lambda^q,\mu^q).\] 
Let $S_i$ be the  vector space generated by all monomial types for $V_i$. Let
$\bmu_{\ell}$ act on $S_i$ in such a way that its action is compatible with 
the map $\kv \mapsto \omega_{\kv}$.

Let $B_1(\lambda)$ (resp. $B_2(\mu)$) be the deformation matrices for $V_1$
(resp. $V_2$). 

Then there exists operators $A_1(\lambda)$ on $S_1$ and $A_2(\mu)$ on $S_2$ such
that
$A(\lambda,\mu)=(A_1(\lambda)\otimes A_2(\mu))|_{(S_1\otimes S_2)^{\bmu_\ell}}$,
where we 
identified $H^{n+m-1}(U_{\lambda,\mu})$ with $(S_1\otimes S_2)^{\bmu_\ell}$ by
sending $\kv \oplus \mv$ to $\omega_{\kv \oplus \mv}$.

Suppose $\kv$ is an admissible monomial type for $V_i$ then $A_i \kv=B_i 
\omega_{\kv}$, where we 
identified $\omega_{\mv}$ with $\mv$.
\end{thm}

In more geometric terms, this theorem tells that the deformation matrix $A$ of
$X_{\lambda,\mu}$ is 
essentially the same as the tensor product of the deformation matrices $B_i$ of
the $V_i$,  that the 
difference between the deformation matrix $A$ and the product $B_1\otimes B_2$ 
is completely due to 
admissible monomial types $\nv$ for $X_{\lambda,\mu}$ that cannot be obtained as
the image of an 
admissible couple and that even in this case similar formulas hold.

\section{An example}\label{sect13}
We would like to consider some families of varieties that come out of the twist
construction. The 
examples in Table~\ref{table4} and~\ref{table5} do not give interesting
examples. We start by explaining 
this fact: 

\begin{rem}The Calabi-Yau threefolds that are listed in Table~\ref{table4} and
Table~\ref{table5}  do not 
give interesting examples: 
The examples in Table~\ref{table4} and~\ref{table5} are quotients  of $E_\lambda
\times S_\mu$, by an 
abelian group of order $\ell \in \{3,4 , 6\}$ that acts faithfully and has fixed
points on $E$. This forces the 
elliptic curve $E$ to have $j$-invariant 0 ($\ell=3,6$) or 1728 ($\ell=4$).
The possible variation of the zeta function of $E_\lambda$ is very limited since
for every $\lambda\in 
\FF_q$ there exists a field extension $K/\FF_q$ of degree 4 or 6 such that we
have that $(E)_K\cong (E_\lambda)_K$. This implies that the Frobenius action 
on $H^2(\PP\setminus E_\lambda)$ is the 
Frobenius action on $H^2(\PP \setminus E)$ twisted by a quartic or sextic
character.

Using the K\"unneth decomposition one gets that the deformation of the zeta
function of $X_{\lambda, \mu}$ equals the deformation  of the zeta function 
of $S_\mu$ twisted by a the quartic or sextic character. 
This can be easily determined using the results of \cite{mondef}. For this
reason we will give an example 
of deformations of Calabi-Yau threefolds mentioned in Table~\ref{table6}. 
\end{rem}

Let $C_\lambda$ be the family of genus 25 curves
$x_0^6+x_1^{12}+x_2^{12}+\lambda x_1^{a_1}x_2^{a_2}$ in $\PP^1(2,1,1)$. 
Let $S_\mu$ be the family of $K3$-surfaces
$y_0^6+y_1^{6}+y_2^{3}+y_3^3+
\mu y_1^{\alpha_1}y_2^{\alpha_2}y_3^{\alpha_3}$ in  $\PP^3(1,1,2,2)$. Then the
twist construction 
gives a two-parameter family of quasi-smooth threefold
\[X_{\lambda,\mu}:  x_1^{12}+x_2^{12}+\lambda x_1^{a_1}x_2^{a_2}-y_1^{6} 
-y_2^{3}-y_3^3-\mu y_1^{b_1}y_2^{b_2}y_3^{b_3} \]
in  $\PP^4(1,1,2,4,4)$. In this example we want to describe the zeta-function of
both $X_{\lambda,\mu}$ 
and its resolution of singularities. The latter is a Calabi-Yau threefold by
Proposition~\ref{prop5-8}. 

We explain first how one can compute the deformation associated with
$C_{\lambda}$.

\begin{ex}
The group $H^2(\PP^2-C_\lambda,\QQ_q)$  is generated by the forms
\[ \omega_{\mathbf c}:=\frac{x_0^{k_0}x_1^{k_1}x_2^{k_2}}{F_\lambda} \Omega\]
with $k_0 \in \{0,1,2,3,4\}, k_1\in \{0,1,\dots,10\} \setminus \{9-2k_0\}$ and 
$k_2=8-2k_0-k_1$ or 
$k_2=20-2k_0-k_1$, depending which one of these two lies between 0 and 10. In
particular, $C_\lambda$ has genus 25. Note that a generator $\sigma$ of 
$\bmu_{6}$ acts on $\omega_{\mathbf c}$ by 
sending it to $\zeta_6^{k_0+1}\omega_{\mathbf c}$.  

The forms $\omega_{\mathbf k}$ with $k_0=2$ are pulled back from the genus 5
hyperelliptic curve 
$x_0^2+x_1^{12}+x_2^{12}+\lambda x_1^{a_1}x_2^{a_2}$, the form with $k_0=1,3$
are pulled back 
from the genus 10 trigonal curve $x_0^3+x_1^{12}+x_2^{12}+\lambda
x_1^{a_1}x_2^{a_2}$.

Take now $a_1=1$. Then $a_2=11$. The following result can be obtained using
Proposition~\ref{PrpCoeffb}.
 Fix $k_0\in \{0,1,2,3,4\}$. Let $k_1\in \{0,1,\dots 11\}\setminus \{9-2k_0\}$,
let $k_2=8-2k_0-k_1$ if $2k_0+k_1\leq 8$ or $20-2k_0-k_1$ otherwise. 

The entry in $B_1(\lambda)$ at $(2\cdot (k_0+1),k_1+1,k_2+1) \times (2\cdot
(m_0+1),m_1+1,m_2+1)$ 
is non-zero only if $m_0=k_0$.

Set $p_0:=(1+11(m_1-k_1)+k_2)/132 $. Using Proposition~\ref{PrpCoeffb} we obtain
that this entry equals 
\[(-\lambda)^{m_1-k_1} {_{12}F_{11}} \left( \begin{matrix}\frac{1+m_1}{12}\;
{p_0} \;p_0+\frac{1}{11} \; 
\dots \; p_0+\frac{10}{11} \\ \frac{m_1-k_1+1}{12} \; \frac{m_1-k_1+2}{12}
\;\hat{1}\; \frac{m_1-k_1+12}{12}
\end{matrix} ;\frac{11^{11}}{12^{12}} \lambda^{12} \right) \mbox{ if }k_1\leq
m_1\]
and
\[(-\lambda)^{m_1-k_1+12} {_{12}F_{11}} \left( \begin{matrix}\frac{13+m_1}{12}\;
{p_0+1} \;p_0+\frac{12}{11} \; \dots \; p_0+\frac{21}{11} \\ 
\frac{m_1-k_1+13}{12} \;
\frac{m_1-k_1+14}{12} \;\hat{1}\; \frac{m_1-k_1
+24}{12}\end{matrix}; \frac{11^{11}}{12^{12}} \lambda^{12}  \right) 
\mbox{ if }k_1> m_1.\]

Actually these functions are ${_{10}}F_{9}$, i.e., two of the parameters in the
upper row appear also in 
the lower row, namely $\frac{m_1+1}{12}$ and one of $p_0+\frac{10-k_2}{12}$ and
$p_0+\frac{22-k_2}{12}$ appear in both the upper and lower row. 

For $k_0=m_0=5$ then $B_1\omega_{\kv}=0$, but the above formula make senses, and
yields an expression for $A_1\omega_{\kv}$, 
$(0,k_1+1,k_2+1) \times (0,m_1+1,m_2+1)$ 
provided that none of the $k_i,m_i$ equals $11 \bmod 12$. This is the
deformation matrix associated 
with the finite set $x_1^{12}+x_2^{12}+\lambda x_1x_2^{11}=0$ consisting of 12
geometric points.
\end{ex}

We proceed by doing a similar calculation for $S_\mu$.
\begin{ex}We want to calculate the deformation matrix $B_2(\mu)$. 
 Take $b_1=2,b_2=b_3=1$. Let $(k_0,k_1,k_2,k_3)$ and $(m_0,m_1,m_2,m_3)$ be
admissible 
monomial types.  Note that a generator $\sigma$ of $\bmu_{6}$ acts on
$\omega_{\mathbf c}$ by 
sending it to $\zeta_6^{k_0+1}\omega_{\mathbf c}$.

The entry of $B(\mu)$ at $\mathbf{k} \times \mathbf{m}$ is non-zero if either
$\mathbf{k} =\mathbf{m}$ or 
$k_0=m_0$, $k_1\neq m_1$, $k_2=k_3$ and $m_2=m_3$. We list now all non-zero
entries.

The entry at $(k_0+1,5-k_0,2\cdot 1,2\cdot 2)\times (k_0+1,5-k_0,2\cdot 1,2 
\cdot 2)$, (which equals the 
entry $(k_0+1,5-k_0,2\cdot 2, 2\cdot 1)\times (k_0+1,5-k_0,2\cdot 2,2\cdot 1)$)
is
\[{}_1F_0\left( \begin{matrix} \frac{5-k_0}{6} \\ -\end{matrix} ; \frac{-1}{27}
\mu^3 \right).\]

The entry at $(4,4,2\cdot 1,2\cdot 1)\times (4,4,2\cdot 1,2\cdot 1)$ equals
\[{}_1F_0\left( \begin{matrix} \frac{1}{3}\\ -\end{matrix} ; \frac{-1}{27} \mu^3
\right).\]

The entry at $(2,2,2\cdot 2,2\cdot 2)\times (2,2,2\cdot 2,2\cdot 2)$.
\[{}_1F_0\left( \begin{matrix} \frac{2}{3}\\ - \end{matrix} ; \frac{-1}{27}
\mu^3 \right).\]

The entryblock at $(3,5,2\cdot 1,2\cdot 1) ,(3,1,2\cdot 2,2\cdot 2)$ equals
\[ \left(\begin{matrix}
{_2F_1 \left(\begin{matrix} \frac{5}{6}  \; \frac{1}{3} \\ \frac{2}{3}
\end{matrix} ; \frac{-1}{27} \mu^3 \right)}
& \frac{1}{54}\mu^2 \; {_2F_1 \left(\begin{matrix} \frac{5}{6}  \; \frac{4}{3}
\\ \frac{2}{3} \end{matrix} ; \frac{-1}
{27} \mu^3 \right)}
\\
-\frac{1}{12}\mu \; {_2F_1 \left(\begin{matrix} \frac{7}{6}  \; \frac{2}{3} \\
\frac{4}{3} \end{matrix} ; \frac{-1}
{27} \mu^3 \right)}
&
{_2F_1 \left(\begin{matrix} \frac{1}{6}  \; \frac{2}{3} \\ \frac{1}{3}
\end{matrix} ; \frac{-1}{27} \mu^3 \right)}
   \end{matrix} \right).\]

The entryblock  at $(5,3,2\cdot 1,2\cdot 1) , (5,5,2 \cdot 2,2\cdot 2)$ equals
\[ \left(\begin{matrix}
{_2F_1 \left(\begin{matrix} \frac{1}{2}  \; \frac{1}{3} \\ \frac{2}{3}
\end{matrix} ; \frac{-1}{27} \mu^3 \right)}
& \frac{1}{6^4}\mu^2  \; {_2F_1 \left(\begin{matrix} \frac{3}{2}  \; \frac{4}{3}
\\ \frac{5}{3} \end{matrix} ; \frac{-1}{27} \mu^3 \right)}
\\
-\mu  \;{_2F_1 \left(\begin{matrix} \frac{5}{6}  \; \frac{2}{3} \\ \frac{4}{3}
\end{matrix} ; \frac{-1}{27} \mu^3 
\right)}
&
{_2F_1 \left(\begin{matrix} \frac{5}{6}  \; \frac{2}{3} \\ \frac{1}{3}
\end{matrix} ; \frac{-1}{27} \mu^3 \right)}
   \end{matrix} \right).\]

The entryblock at $(1,1,2\cdot 1,2\cdot 1), (1,3,2\cdot 2,2\cdot 2)$ equals
\[ \left(\begin{matrix}
{_2F_1 \left(\begin{matrix} \frac{1}{6}  \; \frac{1}{3} \\ \frac{2}{3}
\end{matrix} ; \frac{-1}{27} \mu^3 \right)}
& \frac{1}{6^4} \mu^2  \;{_2F_1 \left(\begin{matrix} \frac{7}{6}  \; \frac{4}{3}
\\ \frac{5}{3} \end{matrix} ; \frac{-1}{27} \mu^3 \right)}
\\
-\mu \; {_2F_1 \left(\begin{matrix} \frac{1}{2}  \; \frac{2}{3} \\ \frac{4}{3}
\end{matrix} ; \frac{-1}{27} \mu^3 
\right)}
&
{_2F_1 \left(\begin{matrix} \frac{1}{2}  \; \frac{2}{3} \\ \frac{1}{3}
\end{matrix} ; \frac{-1}{27} \mu^3 \right)}
   \end{matrix} \right).\]

 It remains to calculate the entry block that appears in $A_2$ but not in $B_2$,
namely the block at 
$(0,2,2\cdot 1,2\cdot 1)\times (0,4,2\cdot 2,2\cdot 2)$
\[ \left(\begin{matrix}
{_2F_1 \left(\begin{matrix} \frac{1}{3}  \; \frac{1}{3} \\ \frac{2}{3}
\end{matrix} ; \frac{-1}{27} \mu^3 \right)}
& \frac{1}{54} \mu^2  \;{_2F_1 \left(\begin{matrix} \frac{4}{3}  \; \frac{4}{3}
\\ \frac{5}{3} \end{matrix} ; \frac
{-1}{27} \mu^3 \right)}
\\
-\mu \; {_2F_1 \left(\begin{matrix} \frac{2}{3}  \; \frac{2}{3} \\ \frac{4}{3}
\end{matrix} ; \frac{-1}{27} \mu^3 
\right)}
&
{_2F_1 \left(\begin{matrix} \frac{2}{3}  \; \frac{2}{3} \\ \frac{1}{3}
\end{matrix} ; \frac{-1}{27} \mu^3 \right)}
   \end{matrix} \right).\]
   This final block i  also the deformation block of the elliptic curve
$x_1^3+x_2^3+x_3^3+\lambda 
x_1x_2x_3$.
\end{ex}

\begin{rem} 
One can show that  image of $H^2(U_{1,\lambda})\otimes H^3(U_{2,\mu})$ has
codimension 22 in  
$H^4(U_{\lambda,\mu})$.
\end{rem}

\begin{rem}
One can easily show that for $\overline{\lambda }\in\FF_q$  we have that 
\[ _1F_0\left(\begin{matrix} \frac{a}{6} \\ -\end{matrix} ; c \lambda^3\right)
{_1F_0\left(\begin{matrix} \frac
{a}{6} \\ -\end{matrix} ; c \lambda^{3q}\right)^{-1}}\]
is a sixth root of unity, where $\lambda$ is the Teichm\"uller lift of
$\overline{\lambda}$.

Suppose $q\equiv 1 \bmod d$. Then the above remark implies that there is a
120-dimensional subspace 
of $H'$ of the $202$-dimensional space $H^3(U_{\lambda,\mu})$ on which the
Frobenius action is the 
Frobenius action  on the ``corresponding'' subspace of $H^3(U_{\lambda,0})$
twisted by a sextic 
character. 
\end{rem}

In this way we get enough information to determine the zeta-function of
$\overline{X_{\lambda,\mu}}$. 
However, we are interested in determining the zeta-function of its resolution.
Up to now we found only a 
part of the cohomology of the desingularization
$\widetilde{\overline{X_{\lambda,\mu}}}$ of $\overline
{X_{\lambda,\mu}}$. 

\subsection{Desingularization}
The singular locus of $\overline{X_{\lambda,\mu}}$ consists of
$x_1=x_2=y_1^{6}+y_2^{3}+y_3^3+\mu 
y_1^{\alpha_1}y_2^{\alpha_2}y_3^{\alpha_3}=0$.
In order to resolve these singularities we resolve the weighted projective
space 
$\PP(1,1,2,4,4)$. This can be done by toric methods as follows (cf. \cite{Ful}):

Let $v_i$ be  the four standard basis vectors of $\RR^4$. Let
$v_0=(-1,-2,-4,-4)$. Consider the fan $\Sigma$ consisting of the cones 
generated by the proper subsets of
$\{v_0,v_1,v_2,v_3,v_4\}$. Then it is 
well-known that the toric variety associated with this fan is isomorphic to
$\PP(1,1,2,4,4)$. In order to 
resolve the singularities we have to split the cones in this fan in smaller
cones.

Let $w_1=(v_0+v_1)/2$. Let $\sigma\in \Sigma$ be a four-dimensional cone through
$v_0$ and $v_1$. 
Let $v_i$ and $v_j$ be the other generators. Then we can split up $\sigma$ into
the union of two cones 
$\sigma_1$ and $\sigma_2$, where $\sigma_1$ is generated by $v_0, w_1,v_i, v_j$
and $\sigma_2$ is 
generated by $w_1,v_1,v_i,v_j$. Define a new fan $\Sigma'$ generated by the four
dimensional cones 
in $\Sigma$ that do not contain both $v_0$ and $v_1$, and the cones obtained by
splitting up the cones 
containing both $v_0$ and $v_1$.
 
The toric variety obtained in this way is the blow-up of $\PP(1,1,2,4,4)$ along
$x_1=x_2=0$. This variety 
is still singular. In order to resolve the singularities completely we need to
split up every four-dimensional cone containing $v_2,w_1$ into a  cone 
generated by $w_1,w_2$ and one cone generated 
by $v_2$ and $w_2$, with $w_2=(w_1+v_2)/2$.
One easily sees that the toric variety $\tilde{\PP}$ corresponding to this fan
is smooth.

Consider now again our threefold $X$. Let $\tilde{X}$ be the strict transform of
$X$ in $\tilde{\PP}$. Let $\EE_1$ and $\EE_2$ be the intersection of the 
two exceptional divisors with $\tilde{X}$. 
A computation in local coordinates shows that $\EE_1$ is a ruled surfaces 
over $y_1^{6}+y_2^{3}+y_3^3+\mu y_1^{\alpha_1}y_2^{\alpha_2}y_3^{\alpha_3}=0$ 
blown-up in three (geometric) points lying in the fibers over the 
intersection of $y_1=0$ with  $B:=\spec (\FF_q[(y_2/y_3)]/((y_2/y_3)^3-1)$. 

Using the toric representation of our variety one easily sees that 
$\EE_2\cong F_2 \times B$, where $F_2$ means the Hirzebruch surface $F_2$. 
Hence, if $q\equiv 1 \bmod 3$ then $\EE_2$ consists of three copies of the 
Hirzebruch surface $F_2$, lying over the points with $x_1=x_2=y_1=0$, 
and that the intersection of $\EE_1$ and $\EE_2$ consists of 3 times 
the union of two $\PP^1$s intersecting in a point.

\subsection{$Z(\tilde{X},t)$}
One easily computes that $Z(F_2,t)=\frac{1}{(1-t)(1-pt)^2(1-p^2t)}$ and that
\[ Z(\EE_2,t)= \left\{ \begin{array}{ll} Z(F_2,t)^2Z(F_2,-t)  & \mbox{if } q\equiv
2 \bmod 3 \\ Z(F_2,t)^3 &
\mbox{if } q\equiv 1 \bmod 3\end{array}\right. \]

To compute $Y(\EE_1,t)$ note that a ruled surface over a curve $C$ has
zeta-function
\[ Z(S,t)= Z(C,t)Z(C,pt).\]

The intersection of $\EE_1$ and $\EE_2$ consists of geometrically three 
connected components. In the case that $q\equiv 1\bmod 3$ all three are 
connected over $\FF_q$. Otherwise only one connected component is defined 
over $\FF_q$. Each connected component is isomorphic to a union of two 
$\PP^1$ intersecting a point. 

Collecting everything we see that
\[ Z(\tilde{X},t)=\frac{Z(X,t)}{Z(C,t)Z(C,qt)(1-qt)^2(1-q^2t)^2
(1-(-1)^{(q-1)/3}t)(1-(-1)^{(q-1)/3}qt)} \]

It remains to determine $Z(C,t)$. This is relatively easy. Let $C:
y_0^6+y_1^3+y_2^3+\mu y_0^{\alpha_0}y_1^{\alpha_1}y_2^{\alpha_2}$ in 
$\PP(1,2,2)$. This implies that $\alpha_0$ is divisible by 2. 
The curve $C$ is isomorphic to $x_0^3+x_1^3+x_2^3+\mu
y_0^{\alpha_0/2}y_1^{\alpha_1}y_2^
{\alpha_2}$. In the case $\alpha_0=2, \alpha_1=\alpha_2=1$. We have the
well-known Hasse pencil. 
This family has the following deformation matrix
\[ \left(\begin{array}{cc} 
 \;_{2}F_1\left( \begin{array}{c} \frac{1}{3} \; \frac{1}{3} \\ \frac{2}{3}
\end{array} ; \frac{-\mu^3}{27}
\right)  & 
\frac{\mu^2}{54} \;_{2}F_1\left( \begin{array}{c} \frac{4}{3} \; \frac{4}{3} \\
\frac{5}{3} \end{array} ; \frac{-\mu^3}{27}
\right)  \\
-\mu \;_{2}F_1\left( \begin{array}{c} \frac{2}{3} \; \frac{2}{3} \\ \frac{4}{3}
\end{array} ; \frac{-\mu^3}{27} 
\right)  & 
 \;_{2}F_1\left( \begin{array}{c} \frac{2}{3} \; \frac{2}{3} \\ \frac{1}{3}
\end{array} ; \frac{-\mu^3}{27}
\right)  

\\ \end{array}\right). \]

\end{document}